\documentclass[a4paper]{amsart}
\usepackage[utf8]{inputenc}
\usepackage{amssymb,amsmath,enumerate,amsthm,url}
\usepackage[all,cmtip]{xy}
\usepackage{graphicx}
\usepackage{tikz-cd}
\usepackage[numbers]{natbib}
\usepackage{bibentry}
\usepackage{hyperref}

\theoremstyle{plain}
\newtheorem{theorem}{Theorem}[section]
\newtheorem{lemma}[theorem]{Lemma}
\newtheorem{proposition}[theorem]{Proposition}

\newtheorem{fact}[theorem]{Fact}
\newtheorem*{fact*}{Fact}
\newtheorem{corollary}[theorem]{Corollary}
\newtheorem{assumption}[theorem]{Assumption}
\newtheorem*{assumption*}{Assumption}
\newtheorem{assumptions}[theorem]{Assumptions}
\newtheorem*{assumptions*}{Assumptions}

\theoremstyle{definition}
\newtheorem{definition}[theorem]{Definition}
\newtheorem*{definition*}{Definition}

\newtheorem*{notation*}{Notation}
\theoremstyle{remark}
\newtheorem{claim}{Claim}[theorem]
\newtheorem*{claim*}{Claim}
\newtheorem{remark}[theorem]{Remark}
\newtheorem*{remark*}{Remark}
\newtheorem{example}[theorem]{Example}
\newtheorem*{example*}{Example}

\newtheorem*{note*}{Note}
\newtheorem{question}[theorem]{Question}
\newtheorem*{question*}{Question}
\begin{document}
\providecommand{\defn}[1]{{\bf #1}}
\providecommand{\bdl}{\boldsymbol\delta}
\providecommand{\calT}{\mathcal{T}}
\renewcommand{\P}{\mathbb{P}}
\newcommand{\pow}{\mathcal{P}}
\renewcommand{\o}{\circ}
\newcommand{\C}{{\mathbb{C}}}
\newcommand{\R}{{\mathbb{R}}}
\newcommand{\N}{{\mathbb{N}}}
\newcommand{\Z}{{\mathbb{Z}}}
\newcommand{\K}{{\mathbb{K}}}
\newcommand{\A}{{\mathbb{A}}}
\newcommand{\G}{{\mathbb{G}}}
\newcommand{\U}{{\mathcal{U}}}
\newcommand{\fin}{{\operatorname{fin}}}
\newcommand{\tp}{{\operatorname{tp}}}
\newcommand{\st}{{\operatorname{st}}}
\newcommand{\cl}{{\operatorname{cl}}}
\newcommand{\trd}{{\operatorname{trd}}}
\newcommand{\Div}{{\operatorname{Div}}}
\newcommand{\Pic}{{\operatorname{Pic}}}
\newcommand{\acl}{{\operatorname{acl}}}
\newcommand{\locus}{{\operatorname{locus}}}
\newcommand{\Aut}{{\operatorname{Aut}}}
\newcommand{\SO}{{\operatorname{SO}}}
\newcommand{\PSO}{{\operatorname{PSO}}}
\newcommand{\PO}{{\operatorname{PO}}}
\newcommand{\GL}{{\operatorname{GL}}}
\newcommand{\PGL}{{\operatorname{PGL}}}
\newcommand{\Fix}{{\operatorname{Fix}}}
\newcommand{\x}{{\overline{x}}}
\newcommand{\y}{{\overline{y}}}
\renewcommand{\a}{{\overline{a}}}
\renewcommand{\d}{{\overline{d}}}
\renewcommand{\L}{{\mathcal{L}}}

\newcommand{\cstr}{constructible$^0$}

\providecommand{\FIXME}[1]{({\bf \small FIXME: #1})}
\providecommand{\TODO}[1]{({\bf \small TODO: #1})}

\newcommand\subqed[1]{\qedhere$_{\mathit{\ref{#1}}}$}

\def\Ind#1#2{#1\setbox0=\hbox{$#1x$}\kern\wd0\hbox to 0pt{\hss$#1\mid$\hss}
\lower.9\ht0\hbox to 0pt{\hss$#1\smile$\hss}\kern\wd0}
\def\ind{\mathop{\mathpalette\Ind\emptyset }}
\def\notind#1#2{#1\setbox0=\hbox{$#1x$}\kern\wd0\hbox to 0pt{\mathchardef
\nn=12854\hss$#1\nn$\kern1.4\wd0\hss}\hbox to
0pt{\hss$#1\mid$\hss}\lower.9\ht0 \hbox to
0pt{\hss$#1\smile$\hss}\kern\wd0}
\def\nind{\mathop{\mathpalette\notind\emptyset }}

\title[Elekes-Szabó for collinearity on cubic surfaces]
{Elekes-Szabó for collinearity on cubic surfaces}
\author{Martin Bays, Jan Dobrowolski, and Tingxiang Zou}
\address{Martin Bays, Mathematical Institute, University of Oxford,
Andrew Wiles Building, Radcliffe Observatory Quarter, Woodstock Road,
Oxford, OX2 6GG}
\address{Tingxiang Zou, Institut für Mathematische Logik und 
Grundlagenforschung, Fachbereich Mathematik und Informatik, Universität 
Münster, Einsteinstrasse 62, 48149 Münster, Germany}
\address{Jan Dobrowolski, Xiamen University Malaysia,
	Department of Mathematics,
	Jalan Sunsuria, Bandar Sunsuria, 43900 Sepang, Selangor Darul Ehsan, Malaysia \emph{and}\newline 
 Instytut Matematyczny, Uniwersystet Wrocławski, Wrocław, Poland}
\email{mbays@sdf.org, dobrowol@math.uni.wroc.pl, tzou@uni-muenster.de}
\thanks{Dobrowolski was supported by DFG project BA 6785/2-1 and by EPSRC
Grant EP/V03619X/1. All authors were supported by the DFG via BA 6785/1-1 
(part of the ANR-DFG project Geomod).}
\footnotetext{MSC Primary: 52C10; Secondary: 03C98, 14C17.}
\maketitle
\begin{abstract}
  We study the orchard problem on cubic surfaces. We classify possibly 
  reducible cubic surfaces $X\subseteq \P^3(\C)$ with smooth components on 
  which there exist families of finite sets (of unbounded size) with 
  quadratically many 3-rich lines which do not concentrate (in a natural 
  sense) on any projective plane. Namely, we prove that such a family exists 
  precisely when $X$ is a union of three planes sharing a common line.

  Along the way, we obtain a general result about nilpotency of groups 
  admitting an algebraic action satisfying an Elekes-Szabó condition.
\end{abstract}

\tableofcontents

\section{Introduction}
\subsection{Summary of results}
We prove the following result for the 3-dimensional orchard problem on a 
complex cubic surface $S$, showing that if a large finite subset of $S$ has 
many collinear triples, this must be due to a coplanar subset.
\begin{theorem}[Corollary~\ref{c:mainOrchardFinite}, symmetric case]
  \label{t:mainOrchardFiniteIntro}
  For any $\epsilon > 0$ there exists $\eta > 0$ and $N_0 \in \N$ such that
  if $S \subseteq  \P^3(\C)$ is a cubic surface with smooth irreducible components
  which is not the union of three planes sharing a common line, and $A \subseteq  S$ 
  is a finite subset with $|A| \geq  N_0$, and $|C| \geq  |A|^{2-\eta}$ where
  $C := \{(a_1,a_2,a_3) \in A^3 : a_1,a_2,a_3 \text{ are distinct, collinear, 
  and not contained in a line contained in $S$}\}$,
  then for some plane $\pi \subseteq  \P^3(\C)$
  we have $|C \cap \pi^3| \geq  |A|^{2-\epsilon}$.
\end{theorem}

\subsection{Background and motivation}
\subsubsection{Orchard problems}
The (planar) \emph{orchard problem} asks for finite configurations of points 
in the real plane such that many lines pass through exactly three of the 
points. Call such lines \emph{3-rich}. See \cite{BGS-orchard} for details 
of the long history of this problem, and the result that for any $n\geq 3$, 
taking a subgroup of order $n$ of an elliptic curve yields a configuration 
with $\lfloor n(n-3)/6\rfloor+1$ 3-rich lines. In \cite[Theorem~1.3, 
\S9]{GT-ordinaryLines}, Green and Tao show that this is optimal for large $n$, 
and moreover that taking finite subgroups of irreducible cubics is the only 
source of such optimal configurations.

In \cite{ES-orchard}, Elekes and Szabó consider a coarse asymptotic version of 
the problem: for what configurations of $n$ points in $\R^2$ are there $\sim 
n^2$ lines passing through at least three of the points? Call such lines
\emph{triple lines}.
In this context, any (possibly reducible) planar cubic curve provides such 
configurations, again by considering group structures induced by collinearity. 
If one restricts to the case that the configuration is to be a subset of an 
algebraic curve, the problem becomes an instance of that considered in 
\cite{ES-groups}, where algebraic groups are seen to be the only possible way 
to achieve the quadratic bound; \cite{ES-orchard} uses this to show that only
cubic curves yield quadratically many triple lines. A subsequent strengthening
by Raz, Sharir, and de Zeeuw \cite[Corollary~6.2]{RSdZ-ESR} gives a bound
$O_d(n^{2 - \frac16})$ on the number of such lines for $n$ points on an
irreducible curve of degree $d>3$. 

The motivation behind this paper is the problem of extending the main
higher-dimensional result of \cite{ES-groups} (Theorem 27)
by dropping the restrictive general position assumption made 
in that paper. As a natural example of this, we consider the 
\emph{spatial orchard problem}, in which one now considers configurations of 
$n$ points in $\R^3$ and counts the triple lines. See 
\cite{dZ-space} and its references for what literature there is on 
these spatial problems. We can again formulate a problem of Elekes-Szabó-type
by taking an algebraic surface in $\R^3$ and asking for configurations of $n$
points on the surface with $\sim n^2$ triple lines. We have uninteresting
solutions to this: if $S$ is a cubic surface, then the intersection of $S$
with a plane is a cubic curve, which typically supports a configuration as
above. Distributing our points across a family of $m=m(n)$ such planes, we can
get $\Omega(\frac{n^2}m)$ triple lines with $O(\frac {n^2}{m^2})$ of them on
any given plane.
To rule out these unoriginal solutions, we may phrase the question as follows 
(where we also pass to the more algebraically natural generality of surfaces 
in complex projective space rather than $\R^3$):

\begin{question} \label{q:introQ}
  For which algebraic surfaces $S \subseteq  \P^3(\C)$ are there $\epsilon>0$ 
  and $c>0$ such that arbitrarily large finite subsets $A \subseteq  S$
  support $\geq  c|A|^2$ triple lines,
  no $|A|^{2-\epsilon}$ of which are coplanar?
\end{question}

We answer this question with additional conditions on $S$: we assume $S \subseteq  
\P^3(\C)$ is a possibly reducible cubic surface, i.e.\ the zero-set of a cubic 
homogeneous polynomial with no repeated factors, and that each irreducible 
component is smooth. We find that such an $S$ satisfies the conditions of 
Question~\ref{q:introQ} if and only if $S$ is the union of three planes sharing a 
common line. This is a consequence of Corollary~\ref{c:mainOrchardFinite} stated 
above, which in fact has better and uniform (albeit inexplicit) bounds.


\subsubsection{Higher-dimensional Elekes-Szabó}
As mentioned above, our real aim with this paper is to develop techniques to help tackle a much more general problem, what we term the \emph{higher-dimensional Elekes-Szabó problem}.
One symmetric formulation of this for surfaces is as follows. Let $R 
\subseteq  S_1 \times S_2 \times S_3$ be a constructible ternary relation 
between algebraic surfaces, such that the projection to any pair $S_i \times 
S_j$ is dominant with generically finite fibres. Suppose that for any $n>0$, 
we can find $N \geq  n$ and $A_{i,n} \subseteq  S_i$ with $|A_{i,n}| = N$ such 
that $|R \cap \prod_i A_{i,n}| \geq  N^{2-\frac1n}$. Then what can we say 
about $R$ or the $A_{i,n}$? In \cite{ES-groups}, slightly refined in 
\cite{BB-cohMod}, it is shown that either $R$ is the image under algebraic 
correspondences on the co-ordinates of the graph of a commutative group 
operation, or for some $i$ and $\epsilon>0$, there is no bound of the form 
$|C\cap A_{i,n}| = O(|A_{i,n}|^\epsilon)$ on the size of the intersection with 
algebraic curves $C \subseteq S_i$. However, this second case of 
non-negligible intersections with curves can occur;
a concrete example is given in \cite[\S 4]{BB-cohMod}. All known examples
arise from algebraic actions of nilpotent algebraic groups (and their
approximate subgroups), and we expect this to be the correct general context;
see \cite[Question~7.5]{homogES} for a precise formulation. The present work
analyses this problem in the particular case that $R$ is the collinearity
relation on $S_1=S_2=S_3$ a cubic surface (discounting collinearities along
lines contained in the surface). Our results confirm our general expectation in
this case, and our methods already generalise to a wider class of such
problems, as we discuss below.

\subsection{Overview}
Our treatment of collinearity on a cubic surface splits into two largely 
disjoint cases, according to whether the surface is reducible or not. In the 
reducible case, the surface is either the union of a smooth quadric surface 
$Q$ and a plane $P$, or the union of three distinct planes. In the former 
case, for $z \in Q \setminus P$ we obtain an automorphism $\gamma_z \in 
\Aut(Q) \cong  \PO_4$ of the quadric, defined generically by $\gamma_z(x) = y$ 
if $(x,y,z)$ are distinct and
collinear. Considering this along with the assumption that we have
quadratically many collinear triples, we prove a general result
(Theorem~\ref{t:cohAct}) showing that in such a situation we obtain an approximate
subgroup of the automorphism group, of a kind which contradicts results of
Breuillard-Green-Tao \cite{BGT-lin} unless its Zariski closure is a nilpotent
subgroup. Analysing the possibilities for such subgroups yields the result in
this case. The case of three planes is similar (and easier).

In the case of a smooth irreducible cubic surface, the corresponding 
birational maps $\gamma_z$, known as \emph{Geiser involutions} in this case, 
do not generate an algebraic group, and so the techniques of the previous 
paragraph do not apply. Instead, we argue as follows. If $A \subseteq  S$ supports 
quadratically many 3-rich lines, then the ``energy'' relation
\begin{align*} E := \{ &((x_1,x_2),(y_1,y_2)) \in S^2 \times S^2 : \exists z \in S.\; \\
  &((x_1,y_1,z) \text{ and } (x_2,y_2,z) \text{ are collinear, and } x_1,y_1,x_2,y_2,z \text{ are distinct})\}
\end{align*}
has $|E\cap A^4| \sim |A|^3$, which would contradict known bounds of 
Szemerédi-Trotter type if $E$ omitted the bipartite complete graph $K_{2,s}$ 
for sufficiently small $s$. In fact $E$ does not have this property, but we 
show in Theorem~\ref{t:coplanarK2s} that any (sufficiently generic) $K_{2,s}$ 
in (a slightly modified version of) $E$ is planar after deleting a bounded finite subset. Adapting the proof of 
such Szemerédi-Trotter-type bounds (we use the version due to 
Chernikov-Galvin-Starchenko, \cite{CGS,CPS}), we deduce that the energy of $A$ 
concentrates on planes. Using the bounds in Tóth's result \cite{Toth} on 
Szemerédi-Trotter in $\P^2(\C)$, we reduce to the case that the energy 
concentrates on a single plane, which yields our result.

These combinatorial arguments leave us with the purely geometric problem of 
proving Theorem~\ref{t:coplanarK2s}. For this, we consider a curve $C$ of 
fixed points of the composition of four (sufficiently generic) Geiser 
involutions $(\gamma_{a_i})_{i<4}$, and use algebro-geometric techniques to 
show that $C$ is planar, and coplanar with the $a_i$. The main part of this 
argument is a careful analysis ruling out hypothetical exceptional cases where 
the $a_i$ are singular points on $C$; for this we use divisor class 
calculations and the arithmetic genus.

In handling the combinatorial aspects of these arguments, our first move will 
be to pass to pseudofinite sets (ultraproducts of finite sets),
following \cite{Hr-psfDims}. This pseudofinite setting converts asymptotic 
questions about arbitrarily large finite sets into more convenient questions 
about individual pseudofinite sets. We make the further move of working with 
complete types in a suitable language, in the sense of first-order logic. This 
small abstraction leads to cleaner statements and proofs; in particular it 
makes ``popularity arguments'' an invisible part of the formalism. 
Furthermore, it brings into play useful tools from model theory, particularly 
independence calculus via Hrushovski's coarse pseudofinite dimension, which 
converts asymptotic exponents in finite cardinalities into a real-valued 
dimension with good properties. This leads us naturally to the notion of 
\emph{weak general position}, which is a useful formalisation of the idea of 
the smallest subvariety on which a configuration concentrates, and is a key 
technical tool in our analysis.

We conclude this section by discussing possible generalisations of our methods 
to the general higher-dimensional Elekes-Szabó problem discussed above. On one 
hand, this paper illustrates that a straightforward solution is unlikely, 
since even in this special case the analysis is difficult and intricate. More 
positively, some of the methods and ideas of this paper are applicable to more 
general cases, as some further work (\cite{homogES}, \cite{autC2}) has already 
demonstrated.
Firstly, Theorem~\ref{t:cohAct} confirms that when an action of an algebraic 
group supports a higher-dimensional Elekes-Szabó configuration, the relevant 
part of the group must indeed be nilpotent; this is useful for 
generalisations, in particular when combined (as in \cite{homogES}) with the 
homogeneous space version of the group configuration theorem, which produces 
such a group action. Meanwhile, the case of a smooth irreducible surface $S$
can be seen (via a birational map $S \to \P^2$) as a matter of the action on
$\P^2$ of an algebraically Ind-definable group, namely the subgroup of
the Cremona group of birational automorphisms of $\P^2$ generated by the
Geiser involutions. As above, we expect in general that Elekes-Szabó
configurations obtained from such actions reduce to (nilpotent) algebraic
subgroups of the Ind-definable group, and our result here can be seen as
a case of this (with the algebraic group being that of planar cubic curve on
which the configuration concentrates). Question~\ref{q:introQ} corresponds to
a generalisation to the group of generically finite-to-finite correspondences
of a surface. Studying $K_{2,s}$'s for energy via analysing curves of fixed
points of compositions, which was the crucial step in the present case of
collinearity, turns out to be useful for this general problem of actions of
Ind-definable groups, as in \cite{autC2}.

\subsection{Outline of the paper}
\S\ref{s:prelims} contains preliminary definitions and results required for the 
rest of the paper, in particular developing what we need of the calculus of 
pseudo-finite dimensions.
In \S\ref{s:EStri}, we use this calculus to reduce our main results to a 
certain condition that ``wES-triangles are planar''. The following sections 
then prove this separately in the cases mentioned above. \S\ref{s:reducible} 
deals with the reducible case. This relies on an abstract result, independent of the main text and proven in Appendix~\ref{s:cohAct}, on nilpotence of groups with a ``coherent action''. 
We apply this first in the warm-up case \S\ref{s:planes} of three 
non-collinear planes, then in the case \S\ref{s:quadric} of a quadric union a 
plane. Next we handle in \S\ref{s:irreducible} the case of an irreducible 
smooth cubic surface. This again makes use of some results independent of the main 
text and proven in appendices: \S\ref{s:STvariant} on our revised version of 
the Szemerédi-Trotter bounds, and \S\ref{s:geiser} on the algebraic geometry 
of Geiser involutions, including our key result Theorem~\ref{t:coplanar} on 
fixed curves of compositions of four Geiser involutions. Finally, 
\S\ref{s:concl} brings the various cases together, and returns to the finitary 
setting to obtain (a slight generalisation of) 
Theorem~\ref{t:mainOrchardFiniteIntro}.

\subsection{Acknowledgements}
Thanks to Emmanuel Breuillard for originally suggesting the problem, and for helpful 
initial discussion.
Thanks to Yifan Jing for helpful discussion around the topic of 
Remark~\ref{r:bourgain}.
We are grateful to the anonymous referees for suggestions which significantly
improved the presentation and structure of the paper.

\section{Preliminaries}
\label{s:prelims}
\subsection{Non-standard setup}
As discussed above, our proofs make extensive use of coarse pseudofinite 
dimensions.
We give an abbreviated version of the necessary setup. For a more leisurely 
account of the same approach, see \cite[Section~2]{BB-cohMod}.

Let $\U$ be a non-principal ultrafilter on $\N$.
For a set $Z$ the \defn{internal} subsets of the ultrapower $Z^\U$ are
the elements of $\pow(Z)^\U \subseteq \pow(Z^\U)$.
The finite cardinality function $|\cdot| : \pow(Z) \to \N \cup \{\infty\}$ 
induces a function $|\cdot| : \pow(Z)^\U \to \N^\U \cup \{\infty\}$ defining the 
non-standard cardinality $|X|$ of an internal subset $X$ of $Z^\U$.
We fix some $\xi \in \R^\U$ greater than any standard real,
and for $\alpha \in \R^\U_{\geq 0}$ we set $\bdl(\alpha) = \bdl_\xi(\alpha) := 
\st \log_\xi \alpha \in \R \cup \{-\infty,\infty\}$,
where $\st : \R^\U \cup \{-\infty,\infty\} \to \R \cup \{-\infty,\infty\}$ is 
the standard part map,
and we set $\bdl(X) := \bdl(|X|)$.
In other words, $\bdl(\prod_{i\rightarrow \U} X_i)$ is (when finite) the metric 
ultralimit in $\R$ along $\U$ of $\log_{\xi_i}(|X_i|)$, where $\xi = 
(\xi_i)_i/\U$.

A \defn{$\bigwedge$-internal} subset of $Z^\U$ is an intersection $\bigcap_{i \in 
\omega} X_i$ of countably many internal subsets $X_i$.
We define $\bdl(\bigcap_{i \in \omega} X_i) := \inf_{n \in \omega} 
\bdl(\bigcap_{i<n} X_i)$.

A $\bigwedge$-internal set $X$ is \defn{broad} if $0 < \bdl(X) < \infty$.

Recall that $Z^\U$ is $\aleph_1$-compact, meaning that a $\bigwedge$-internal
set $\bigcap_{i \in \omega} X_i$ is non-empty as long as each finite 
intersection $\bigcap_{i < n} X_i$ is non-empty.


\begin{lemma} \label{l:intInType}
  Any $\bigwedge$-internal set $X$ contains an internal subset $X' \subseteq  X$ with 
  $\bdl(X') = \bdl(X)$.
\end{lemma}
\begin{proof}
  Say $X = \bigwedge_i X_i$ with each $X_i$ internal.
  Suppose $a := \bdl(X) \in \R$ (the case $a=\infty$ is similar, and in the 
  case $a=-\infty$ we can take $X' := \emptyset $).
  Then finite intersections of
  \[\{ \{ S \in \pow(Z)^\U : |S| \geq \xi^{a - \frac1m} \} : m > 0 \}
  \cup \{ \{ S \in \pow(Z)^\U : S \subseteq X_i \} : i \in \omega \}\]
  are non-empty, as witnessed by $\bigcap_{i<n} X_i$. So by 
  $\aleph_1$-compactness (applied now to $\pow(Z)^\U$ rather than to $Z^\U$), 
  there is some $X'\in\pow(Z)^\U$ satisfying all these conditions. Recalling the 
  definition of $\bdl$, this means that $\bdl(X') \geq a = \bdl(X)$ and $X' 
  \subseteq X$, and we conclude by observing that $X' \subseteq X$ implies 
  $\bdl(X') \leq \bdl(X)$.
\end{proof}

Throughout the paper, we let $\K := \C^\U$.\footnote{We use $\C$ here only for concreteness; all we 
will actually use about $\K$ is that it is an ultraproduct across $\U$ of algebraically closed  
fields of characteristic 0.}

\subsection{Coarse pseudofinite independence}
\label{s:psfInd}
Here we define $\bdl$ for tuples and recall its basic properties.
All tuples will be assumed to be of finite length.

Consider $\K$ as a structure in a countable language $\L$ expanding the ring 
language, and suppose that every $\L$-definable set of $\K$ is internal.

Then for $C \subseteq  \K$ countable and $a$ a tuple from $\K$,
let $\bdl(a/C) = \bdl^\L(a/C) := \bdl(\tp(a/C)(\K)) \in [0,\infty]$ where $\tp(a/C)(\K) = 
\tp^\L(a/C)(\K) = \bigcap_{\phi \in \tp^\L(a/C)}\phi(\K)$ is considered as a 
$\bigwedge$-internal set.
We include $\L$ in the notation only when we need to make it explicit.
We write $a \ind ^{\bdl}_C b$ to mean $\bdl(a/C) = \bdl(a/Cb)$.

We assume that $\bdl$ is continuous for $\L$, i.e.\ given $n,m\ge 1$, $\alpha 
\in \R$, $\epsilon>0$ and a  $\emptyset$-definable set $Y \subseteq K^{n} 
\times K^m$ there is a $\emptyset$-definable set $W \subseteq K^m$ such that 
$$ \{b \in K^m : \bdl(Y_{b}) \geq \alpha+\epsilon\} \subseteq W \subseteq \{b 
\in K^m : \bdl(Y_{b}) \geq \alpha\},$$ where $Y_{b}$ is the fibre $\{x \in 
K^n: (x,b) \in Y\}$. See \cite[Section~2.1.6]{BB-cohMod} for the fact that it 
is harmless to assume this, since one can ``close off'' the language to a 
countable language in which $\bdl$ is continuous and definability still 
implies internality.

The reader should be aware that our language $\L$ will not be fixed throughout 
the paper; we sometimes increase $\L$ by adding new internal sets as 
predicates or adding constants, while always ensuring that $\L$ is countable and 
$\bdl$ is continuous. Nonetheless, we mostly suppress $\L$ from the notation.

We have the following basic properties, which we will use throughout without 
comment. We write $a' \equiv _C a$ to mean $\tp(a'/C) = \tp(a/C)$.
\begin{fact} \label{f:bdlBasics}
  Let $a,b,d$ be tuples from $\K$ and $C \subseteq  \K$ a countable subset,
  and suppose $\bdl(a/C),\bdl(b/C),\bdl(d/C) < \infty$
  \begin{itemize}
  	\item $\trd(a/C)=0\Rightarrow \bdl(a/C)=0$ \hfill{(Algebraicity)}
  	\item $a \equiv _C b \Rightarrow  \bdl(a/C) = \bdl(b/C)$ \hfill{(Invariance)}
  \item $\bdl(ab/C) = \bdl(a/Cb) + \bdl(b/C)$ \hfill{(Additivity)}
  \item $a \ind ^{\bdl}_C b \Leftrightarrow  b \ind ^{\bdl}_C a$ \hfill{(Symmetry)}
  \item $ab \ind ^{\bdl}_C d \Leftrightarrow  (b \ind ^{\bdl}_C d \wedge a \ind ^{\bdl}_{Cb} d)$ 
  \hfill{(Monotonicity and Transitivity)}
  \item For any tuple $e$ from $\K$, there exists $a' \equiv _C a$ such that $a' 
  \ind ^{\bdl}_C e$; \hfill{(Extension)}\\
  more generally, any broad $\bigwedge$-definable set $X$ over $C$ contains an 
  element $x$ with $\bdl(x/C) = \bdl(X)$.
  \end{itemize}
\end{fact}
\begin{proof}
	Algebraicity follows from the fact that $\bdl(X)=0$ if $X$ is finite.
  Invariance is immediate from the definition.
  For Additivity, see e.g.\ \cite[Fact~2.8]{BB-cohMod}.
  Symmetry and Monotonicity and Transitivity follow.
  The second clause of Extension is a consequence of $\aleph_1$-compactness 
  and the countability of $\L$ and $C$,
  and the first clause follows, since it yields $a' \in \tp(a/C)(\K)$ with 
  $\bdl(a'/Ce) = \bdl(\tp(a/C)(\K)) = \bdl(a/C)$.
\end{proof}

We will often apply these definitions and properties to tuples of elements of 
projective algebraic varieties or of arbitrary algebraic groups, rather than 
just to tuples of elements of $\K$. This can be understood by extending the 
above definitions and arguments to such elements, or by using elimination of 
imaginaries in $\operatorname{ACF}$ to reduce to the case of tuples, as in 
\cite[Section~2.1.10]{BB-cohMod}; in the case of a projective variety, this 
just comes down to taking co-ordinates in one of the standard affine patches 
$\mathbb{A}^n \subseteq  \P^n$.

By \emph{constructible} (over $C \subseteq  \K$) we mean: definable (over $C$) in the 
field $(\K;+,\cdot)$. By quantifier elimination in the theory of algebraically 
closed fields, the constructible subsets of a variety are precisely the 
boolean combinations of Zariski closed subsets. The dimension $\dim(X)$ of a 
constructible set is the dimension of its Zariski closure.

\begin{lemma} \label{l:genericFibre}
  Suppose $f: X \rightarrow  Y$ is a constructible map of constructible sets,
  and $A \subseteq  X$ is a broad $\bigwedge$-definable subset,
  all over $C \subseteq  \K$,
  and $a \in A$ satisfies $\bdl(a/C) = \bdl(A)$.

  Then $a$ is also $\bdl$-generic in its fibre of $f$,
  namely $\bdl(a/Cf(a)) = \bdl(f^{-1}(f(a)) \cap A)$.
\end{lemma}
\begin{proof}
  Let $F := f^{-1}(f(a)) \cap A$,
  and let $a' \in F$ with $\bdl(a'/Cf(a)) = \bdl(F)$.

  Then
  \begin{align*} \bdl(a/Cf(a)) &= \bdl(a/C) - \bdl(f(a)/C)
  \\&= \bdl(A) - \bdl(f(a)/C)
  \\&\geq  \bdl(a'/C) - \bdl(f(a)/C)
  \\&= \bdl(a'/Cf(a)) = \bdl(F) ,\end{align*}
  and the converse inequality is immediate.
\end{proof}

We furthermore assume that the set of interpretations of the constants in our 
language $\L$ is an algebraically closed subfield of $\K$.
We can always ensure this assumption holds by adding parameters.

\begin{definition}\label{d:zero}
  Let $E \leq  \K$ be the algebraically closed field of interpretations of the 
  constants in our language $\L$. Then we use superscript-0 to refer to the 
  ring language expanded by constants for $E$, as follows.
  Let $C$ be a subset of $\K$.
  \begin{itemize}\item $\trd^0(a/C) := \trd(a/CE)$ is transcendence degree relativised to $E$.
  \item $\ind ^0$ means algebraic independence over $E$; that is, $a \ind ^0_C b$ if and only if 
  $\trd^0(a/Cb) = \trd^0(a/C)$.
  \item $\acl^0(C) := \acl^{\operatorname{ACF}}(CE)$ is the smallest algebraically closed 
  subfield of $\K$ containing $CE$.
  \item $a \equiv ^0_C b$ means that $a$ and $b$ have the same field type over $CE$, 
  i.e.\ they are in the same constructible sets over $CE$.
  \item If $V$ is a variety over $CE$ and $a \in V(\K)$, then $\locus^0(a/C)$ is the smallest 
  algebraic subvariety $W$ of $V$ over $CE$ which contains $a$.
  \item A \defn{$C$-{\cstr}} set is a set definable in the ring language with 
  constants
  for $E$ and $C$, in other words a constructible set over $CE$.
  \end{itemize}

  In the case $C=\emptyset$, we omit it from the notation.
  In particular, a \defn{{\cstr}} set is a $\emptyset$-{\cstr} set.
  When we want to make the choice of language (and hence the choice of $E$) 
  explicit, we write $\trd^{0,\L}$ etc.
\end{definition}

Note that $\trd^0$ and $\ind ^0$ satisfy the properties of Fact~\ref{f:bdlBasics}. We 
will use these properties without comment.

If a variety $V$ is constructible over a countable $C \subseteq  \K$,
a \defn{generic} point of $V$ over $C$ is an $a \in V$ such that $\trd(a/C) = \dim(V)$,
i.e.\ $a$ is contained in no lower-dimensional subvariety of $V$ over $C$.
Such an $a$ exists by $\aleph_1$-compactness of $\K$. We say $a$ is generic in
$V$ if it is generic in $V$ over some $C$ (over which $V$ is constructible).

%
%
%
%

To summarise, we are making the following assumptions about our language $\L$ 
interpreted on $\K$, and the constant $\xi$:
\begin{assumptions}\label{a:L}
  $\L$ is a countable expansion of the ring language, every $\L$-definable set 
  of $\K$ is internal, $\bdl=\bdl_\xi$ is continuous for $\L$, and the 
  interpretations of the constants in $\L$ form an algebraically closed 
  subfield.
\end{assumptions}

\subsubsection{Weak general position}
\begin{definition} \label{d:wgpType}
  For a countable set $C \subseteq  \K$ and a tuple $a \in \K$,
  the complete type $\tp(a/C)$ is \defn{in weak general position} (or is 
  \defn{wgp})
  if $\bdl(a/C) < \infty$ and
  for any tuple $b \in \K$ such that $\trd^0(a/Cb) < \trd^0(a/C)$,
  we have $\bdl(a/Cb) < \bdl(a/C)$.

  Equivalently: for any $b$, if $a \ind ^{\bdl}_C b$ then $a \ind ^0_C b$.
\end{definition}
\begin{remark}
  It follows from invariance of $\bdl$ that wgp is indeed a property of the 
  type rather than its realisation.
\end{remark}

\begin{definition} \label{d:wgpSet}
  A $\bigwedge$-internal set $A$ is \defn{wgp in} a constructible set $X \supseteq A$ if 
  $\bdl(A) < \infty$ and
  $\bdl(A \cap Y) < \bdl(A)$ for any constructible $Y \subseteq  X$ over $\K$ with 
  $\dim(Y) < \dim(X)$.
\end{definition}

\begin{remark} \label{r:wgpUnif}
  Compactness yields uniformity in this definition:
  if $A$ is wgp in $X$,
  and $(V_b)_b$ is a constructible family of lower dimensional 
  constructible subsets of $X$,
  then there exists $\eta>0$ such that
  $\bdl(A \cap V_b) \leq  \bdl(A) - \eta$ for all $b$.
\end{remark}

\begin{lemma} \label{l:wgpWgp}
	  Let $V$ be a variety and let $a \in V(\K)$.
 Then $\tp(a/C)$ is wgp if and only if the $\bigwedge$-internal set $\tp(a/C)(\K)$ is 
  wgp in $\locus^0(a/C)$.
\end{lemma}
\begin{proof}
  Suppose $\tp(a/C)(\K)$ is not wgp in $\locus^0(a/C)$,
  so say $V_b \subseteq  \locus^0(a/C)$ is a constructible subset of lower dimension, 
  but $\bdl(\tp(a/C)(\K) \cap V_b(\K)) = \bdl(a/C)$.
  Then say $a'$ realises a completion of the partial type $\tp(a/C) \cup \{x 
  \in V_b\}$ with $\bdl(a'/Cb) = \bdl(a/C) = \bdl(a'/C)$. Then $a' \in V_b$, so 
  $\trd^0(a'/Cb) < \trd^0(a'/C)$, so $\tp(a/C) = \tp(a'/C)$ is not wgp.

  Conversely, if $\tp(a/C)$ is not wgp,
  then say $b$ is such that $\trd^0(a/Cb) < \trd^0(a/C)$ but $\bdl(a/Cb) = \bdl(a/C)$.
  Let $W := \locus^0(a/Cb) \subseteq  \locus^0(a/C)$.
  Then $\dim(W) < \dim(\locus^0(a/C))$
  but $\bdl(\tp(a/C)(\K) \cap W(\K)) = \bdl(a/C)$,
  so $\tp(a/C)(\K)$ is not wgp in $\locus^0(a/C)$.
\end{proof}

The following lemma, which is essentially a reformulation of 
\cite[Lemma~2.13]{Hr-psfDims},
will allow us to assume wgp by adding suitable parameters to the language.

\begin{lemma} \label{l:wgpEventually}
  Suppose $\bdl(a/C) < \infty$, where $a \in \K$ is a tuple and $C \subseteq  \K$ is a 
  countable set.
  \begin{enumerate}[(i)]\item
  There exists a tuple $d \in \K$ with $\bdl(d/C) = 0$ such that $\tp(a/Cd)$ is 
  wgp.
  \item
  $\tp(a/C)$ is wgp if and only if $a \ind ^0_C d$ for all tuples $d \in \K$ with 
  $\bdl(d/C) = 0$.
  \end{enumerate}
\end{lemma}
\begin{proof}
  \begin{enumerate}[(i)]\item
  By induction on $\trd^0(a/C)$.
  If $\tp(a/C)$ is wgp, we are done (with $d$ the empty tuple).
  Otherwise, say $a \nind^0_C b$ but $\bdl(a/Cb) = \bdl(a/C)$.
  By the induction hypothesis, by extending $b$ we may assume that $\tp(a/Cb)$ 
  is wgp.
  Let $n \in \N$, and let $\a=(a_i)_{i<n}$ be a $\bdl$-independent tuple of 
  realisations of $\tp(a/\acl^0(Cb))$, meaning $a_i \equiv _{\acl^0(Cb)} a$ and 
  $a_i \ind ^{\bdl}_{\acl^0(Cb)} a_{<i}$. Since $\tp(a/Cb)$ is wgp, $\a$ is also 
  a $\ind ^0$-independent sequence of realisations of $\tp^0(a/\acl^0(Cb))$.
  By a standard result (see e.g.\ \cite[Lemma~1.2.28(ii)]{GST}), taking $n$ 
  large enough,
  it follows that there exists a tuple $b' \in \acl^0(Cb) \cap \acl^0(\a)$ 
  such that $a \ind ^0_{b'} Cb$.
  Then $\bdl(b'/C) = \bdl(\a/C) - \bdl(\a/Cb') \leq  \bdl(\a/C) - \bdl(\a/Cb) = 
  0$.

  Now $\trd^0(a/Cb') < \trd^0(a/C)$ and $\bdl(a/Cb') \leq  \bdl(a/C) < \infty$,
  so by the induction hypothesis,
  say $d'$ is such that $\tp(a/Cb'd')$ is wgp,
  and $\bdl(d'/Cb') = 0$.
  Then $d := b'd'$ is as required.
  \item
  Suppose $a \ind ^0_C d$ for all tuples $d \in \K$ with $\bdl(d/C) = 0$.
  Let $d$ be as in (i), so $\tp(a/Cd)$ is wgp and $\bdl(d/C) = 0$, so $a 
  \ind ^0_C d$. We show that $\tp(a/C)$ is wgp. So suppose $a \nind ^0_C b$. Then 
  $a \nind ^0_{Cd} b$, so $\bdl(a/Cb) = \bdl(a/Cdb) < \bdl(a/Cd) = \bdl(a/C)$ 
  (using $\bdl(d/C) = \bdl(d/Cb) = 0$), as required.
  The converse follows immediately from the definition of wgp once we observe 
  that $a \ind ^{\bdl}_C d$ if $\bdl(d/C)=0$.\qedhere
  \end{enumerate}
\end{proof}

\begin{lemma} \label{l:wgpConstants}
  Given a countable set $C \subseteq \K$, let $\L^0(C)$ be the language 
  obtained from $\L$ by adding constants for the elements of $\acl^0(C)$.
  Then $\L^0(C)$ satisfies Assumptions~\ref{a:L} (without changing $\xi$).

  Furthermore, if $\tp^\L(a/C)$ is wgp, then $\tp^{\L^0(C)}(a)$ is wgp.
\end{lemma}
\begin{proof}
  Continuity of $\bdl$ easily extends from $\L$ to $\L^0(C)$ (see 
  \cite[Remark~2.7]{BB-cohMod}).
  The constants of $\L^0(C)$ form an algebraically closed subfield by 
  definition.

  We have $a \ind^{\bdl,\L^0(C)} b \Leftrightarrow  a 
  \ind_{\acl^0(C)}^{\bdl,\L} b \Leftrightarrow a \ind_C^{\bdl,\L} b$, and 
  similarly for $\ind^0$, so the ``furthermore'' clause is an immediate 
  consequence of the definition of wgp.
\end{proof}

\begin{lemma}\label{l:wgpFacts}
  Let $a,b \in \K$ be tuples, let $C \subseteq \K$ be countable,
  and suppose $\bdl(a/C),\bdl(b/C) < \infty$.
  \begin{enumerate}[(i)]
    \item \label{li:wgpImage} If $\tp(a/C)$ is wgp and $e \in \acl^0(aC)$, 
      then $\tp(e/C)$ is wgp.
    \item \label{li:wgpTrans} If $\tp(b/C)$ and $\tp(a/Cb)$ are wgp, then so 
      is $\tp(ab/C)$.
    \item \label{li:wgpIndie} If $a \ind ^{\bdl}_C b$ and $\tp(a/C)$ is wgp, 
      then $\tp(a/Cb)$ is wgp.
  \end{enumerate}
\end{lemma}
\begin{proof}
  In each case, we use Lemma~\ref{l:wgpEventually}(ii).
  \begin{enumerate}[(i)]
    \item
      Suppose $\bdl(d/C) = 0$ but $e \nind ^0_C d$. Then $a \nind ^0_C d$, 
      contradicting $\tp(a/C)$ being wgp.
    \item
      Suppose $\bdl(d/C) = 0$ but $ab \nind ^0_C d$.
      Then $a \nind ^0_{Cb} d$ or $b \nind ^0_C d$, contradicting the assumption.
    \item
      Suppose $\bdl(d/Cb) = 0$ but $a \nind ^0_{Cb} d$.
      Then $a \nind ^0_C bd$,
      so $a \nind ^{\bdl}_C bd$ by wgp,
      so $a \nind ^{\bdl}_{Cb} d$, contradicting $\bdl(d/Cb)=0$.\qedhere
  \end{enumerate}
\end{proof}

%
%
%

\section{Elekes-Szabó triangles}
\label{s:EStri}
In this section, we use the abstractions of the previous section to distil a 
potential counterexample to Theorem~\ref{t:mainOrchardFiniteIntro} down to a 
complete type satisfying certain conditions on $\bdl$, what we term a 
\emph{wgp Elekes-Szabó triangle} for collinearity. This reduces our problem to 
characterising those cubic surfaces which admit such triangles.

We will refer to one-dimensional projective subspaces of $\P^n$ as \emph{lines}, 
and to two-dimensional projective subspaces as \emph{planes}.

\begin{definition}
  Let $K$ be an algebraically closed field. For the purposes of this paper, a 
  \defn{cubic surface} $S \subseteq  \P^3(K)$ is the zero-set $S = \{ [x:y:z:w] \in 
  \P^3(K) : f(x,y,z,w) = 0 \}$ of a homogeneous degree 3 polynomial $f \in 
  K[X,Y,Z,W]$ with no repeated prime factors; in other words, $S$ is the set 
  of $K$-rational points of the reduced (but possibly reducible) projective 
  variety defined by $f$.
  (The reducedness condition that there are no repeated prime factors is 
  relevant only in the case that $f$ splits into linear factors, i.e.\ that $S$ 
  is a union of planes, where it requires that the three planes be distinct.)

  The zero-sets of the factors $g_i$ of $f$ are the irreducible components of 
  $S$. We will often require these to be smooth; equivalently (by 
  \cite[Exercise~2.1.12]{Shaf-algGeomI}),
  for each $i$, the formal derivative $g_i'$ has no common zero with $g_i$ in 
  $\P^3(K)$.
\end{definition}

Let $S \subseteq  \P^3(\K)$ be a cubic surface over $\K$.

\begin{definition}
  Let $R_S(x_1,x_2,x_3)$ be the constructible relation:
  $(x_1,x_2,x_3)$ is a collinear triple of distinct points on $S$
  not contained in a line contained in $S$.
\end{definition}

Hence $R_{S}\subseteq S^{3}$. Note that if $R_S(a_1,a_2,a_3)$, then $a_i \in \acl^0(a_j,a_k)$ (for $\{i,j,k\} 
= \{1,2,3\}$).

\begin{definition} \mbox{} \label{d:ESTri}
  \begin{itemize}\item A \defn{wgp Elekes-Szabó triangle} (or \defn{wES-triangle}) for $R_S$ 
  is
  a triple $\a = (a_1,a_2,a_3) \in R_S$ such that each $\tp(a_i)$ is wgp and 
  broad,
  and $a_i \ind ^{\bdl} a_j$ for $i \neq  j$.

  A wES-triangle $\a$ is \defn{planar} if there is a {\cstr} plane $\pi \subseteq  
  \P^3$ with $a_1,a_2,a_3 \in \pi$.

  \item We say $R_S$ has \defn{no non-planar wES-triangles}
    if, for every choice of language $\L$ and $\xi\in\N^\U$ satisfying 
    Assumptions~\ref{a:L} and such that $S$ is {\cstr}, every wES-triangle for 
    $R_S$ is planar.
  \end{itemize}
\end{definition}


Allowing the language to vary in the definition of having no non-planar 
wES-triangles corresponds to considering arbitrary finitary configurations as 
in Theorem~\ref{t:mainOrchardFiniteIntro}; see \S\ref{ss:conclFin} for this 
reduction to a finitary statement.

The following lemma will allow us to rule out some degenerate cases when 
proving planarity of wES-triangles.
\begin{lemma} \label{l:wESTriLine}
  Suppose $\a$ is a wES-triangle for $R_S$ and some $a_i$ is contained in a  
  {\cstr} line. Then $\a$ is planar.
\end{lemma}
\begin{proof}
  Say $a_1$ is in a {\cstr} line $\ell$.
  By the definition of $R_S$, the projective subspace spanned by
  $\ell$ and $a_2,a_3$ is a plane $\pi \ni a_1,a_2,a_3$.
  As $a_2 \ind ^{\bdl} a_3$, by wgp we get $a_2 \ind ^0 a_3$, hence the field of definition of $\pi$ is contained in $\acl^0(a_2) \cap 
  \acl^0(a_3) = \acl^0(\emptyset )$.
  Since $\acl^0(\emptyset )$ is named by constants by Assumptions~\ref{a:L}, 
  $\pi$ is {\cstr}.
\end{proof}

Recall that planes $P_1,\ldots ,P_n$ are said to be \defn{collinear} if they share 
a common line $\ell \subseteq  \bigcap P_i$.
This is the one exceptional case, where we do get non-planar wES-triangles:

\begin{example} \label{e:collinearPlanes}
  If $S$ is the union of three distinct collinear planes, then $R_S$ has a 
  non-planar wES-triangle (for some choice of $\L$).

  Indeed, applying a projective linear transformation (recalling that $\PGL_2$
  acts 3-transitively on $\P^1$), we may assume that $S$ is the projective
  closure of the union of the three parallel affine planes
  $P_i := (i,\K,\K) \subseteq  \K^3$ for $i=-1,0,1$.

  Then we obtain an Elekes-Szabó configuration by taking grids. Explicitly,
  for $n \in \N$ let $A_n := ([-1,1]\times(-n,n)\times(-n,n)) \cap \Z^3$, so 
  $|A_n| = 3(2n-1)^2$.
  Then the number $T$ of collinear
  triples $((-1,a,b),(0,a',b'),(1,2a'-a,2b'-b))$ in ${A_n}^3$ satisfies
  $(c n)^4 \leq  T \leq  (2n-1)^4$ for some $c>0$ independent of $n$.
  So setting $A := \prod_{n \rightarrow  \U} A_n \subseteq  S$
  and $\xi := \prod_{n\rightarrow \U} n \in \R_S^\U$,
  we have $\bdl(A) = 2$
  and $\bdl(A^3 \cap R_S) = 4$.

  Choose $\L$ such that $A$ is $\emptyset$-definable,
  and let $(a_1,a_2,a_3) \in A^3 \cap R_S$ with $a_i \in P_i$ and 
  $\bdl(a_1,a_2,a_3) = 4$. Then $\bdl(a_i) = 2$ and $a_i \ind ^{\bdl} a_j$.
  Also $\tp(a_i)(\K)$ is wgp in $P_i$ for $i=1,2,3$;
  indeed, if $C \subseteq  P_i$ is a curve, then $|A \cap C| \leq  c'\xi$ for some finite 
  $c'$ by considering the co-ordinate projections, so $\bdl(\tp(a_i)(\K) \cap 
  C) \leq  \bdl(A \cap C) \leq  1 < 2 = \bdl(a_i)$.
  It follows that $P_i = \locus^0(a_i)$ and, by Lemma~\ref{l:wgpWgp}, that $a_i$ is 
  wgp. Then $\a$ is a wES-triangle, and it is not planar, since if $\pi \ni 
  a_1,a_2,a_3$ is a {\cstr} plane then $a_i \in \P_i \cap \pi$ contradicts $P_i 
  = \locus^0(a_i)$.
\end{example}

The remainder of this paper is devoted to showing that if $S$ is not of this 
form, then it has no non-planar wES-triangles. We consider the reducible and 
irreducible cases in separate sections, then state our final conclusions in 
\S\ref{s:concl} and deduce Theorem~\ref{t:mainOrchardFiniteIntro}.

\section{The reducible case}
\label{s:reducible}
Let $S \subseteq  \P^3(\K)$ be a reducible cubic surface over $\K$ with smooth components.
Let $\L$ be as in Definition~\ref{d:ESTri}; in particular, $S$ is {\cstr}.

We separate further into two cases according to whether $S$ has two or three
irreducible components; each case is treated in its own subsection below. In
both cases, the key tool will be a result obtained in Appendix~\ref{s:cohAct},
essentially as a corollary of \cite{BGT-lin}, which yields nilpotence of
algebraic groups arising from wES-triangles.

\subsection{Non-collinear planes}
\label{s:planes}
Suppose $S$ has three irreducible components. Then since $S$ is a reduced cubic surface, $S$ is the union of three distinct planes, $P_1,P_2,P_3$. We show that if these planes are not collinear, i.e. that they do not share a common line $\ell \subseteq \bigcap_i P_i$, then every wES-triangle for $R_S$ is planar. The necessity of the collinearity condition is demonstrated by Example~\ref{e:collinearPlanes}.

This case serves as a warm-up exercise for the case considered in the next subsection, where $S$ is the union of a smooth quadric and a plane; there the argument will be similar in outline, but the geometry will be somewhat more complicated.

\begin{proposition} \label{p:mainOrchardPlanar}
	 Suppose $P_1$, $P_2$, and $P_3$ are not collinear. Then every wES-triangle is $R_S$ planar.
\end{proposition}
\begin{proof}
	The strategy of the proof is based on applying Theorem~\ref{t:cohAct} to the group generated by certain compositions of projections between two of the given planes, centred at a point of the third plane.
	
  Since {\cstr} projective linear transformations preserve collinearity of
  planes and planarity of wES-triangles, by applying such a transformation
  we may assume $P_1 = \{[0:y:z:w] : y,z,w \}$ and $P_2 = \{[x:0:z:w] : x,z,w \}$.
  
  For $q\in \mathbb{P}^3\setminus (P_1\cup P_2)$, denote by $\gamma_q:P_1\to P_2$ the function which sends $p\in P_1$ to the unique point $p'\in P_2$ such that $p$, $p'$, and $q$ are collinear.

  Let $q,q' \in \P^3 \setminus (P_1 \cup P_2)$.
  We can write $q = [a:1:c:d]$ and $q' = [1:b':c':d']$.
  Then we calculate $\gamma_{q'}^{-1} \o \gamma_{q} : P_1 \rightarrow  P_1$ as
  \begin{align*} \gamma^{-1}_{q'}(\gamma_{q}([0:y:z:w])) &= \gamma^{-1}_{q'} ([-ay:0:z-cy:w-dy])
  \\&= [0:b'ay:z+(c'a-c)y:w+(d'a-d)y] .\end{align*} 
  So $\gamma_{q'}^{-1} \o \gamma_{q}$ corresponds to the element of $\PGL_3 \cong  
  \Aut(P_1)$ with matrix $$
  \begin{bmatrix}
    b'a & 0 & 0 \\
    c'a-c & 1 & 0 \\
    d'a-d & 0 & 1
  \end{bmatrix}
  ;$$
  in other words, $\gamma_{q'}^{-1} \o \gamma_{q}$ corresponds to 
  $$((c'a-c,d'a-d),b'a) \in \G_a^2 \rtimes \G_m$$
  for the action of $\G_a^2 \rtimes \G_m$ on $P_1$ defined by
  $$((a,b),c)[0:y:z:w] := [0:cy:z+ay:w+by].$$
  We identify $\gamma^{-1}_{q'}\o\gamma_{q}$ with this element of $\G_a^2\rtimes \G_m$.

  Let $(p_1,p_2,q)\in S^3$ be a wES-triangle for the collinearity relation
  $R_S\subseteq P_1\times P_2\times P_3$. By the definition of $R_S$,
  $(p_1,p_2,q)$ span a line which is not contained in $S$, so no two elements
  of the triple are in the same plane $P_i$, so we may assume $p_1 \in P_1$,
  $p_2 \in P_2$, and $q \in P_3 \setminus \{P_1 \cup P_2\}$.

  Take $(p'_1,q')\equiv_{p_2}(p_1,q)$ with $p'_1,q'\ind ^{\bdl}_{p_2} p_1,q$. Let $g:=\gamma^{-1}_{q'}\circ\gamma_{q}\in \G_a^2 \rtimes \G_m$, so $g(p_1)=p_1'$ (where the action of $\G_a^2 \rtimes \G_m$ on $P_1$ is that defined above). We will now verify the conditions of Theorem~\ref{t:cohAct} for $g \in \G_a^2 \rtimes \G_m$ and $p_1 \in P_1$.

  Since $q'\ind^{\bdl} p_2$, we get $q'\ind^{\bdl} p_1,q$ by transitivity, 
  hence $q'\ind^{\bdl}_{q}p_1$ by monotonicity, and then since 
  $q\ind^{\bdl}p_1$ we obtain $q,q'\ind^{\bdl} p_1$ and hence $g\ind^{\bdl} 
  p_1$. Similarly, $g\ind^{\bdl} p_1'$. Note that $\tp(g)$ is wgp: indeed, 
  since $q'\ind^{\bdl} q$, and $\tp(q')$ is wgp, and $g\in\acl^0(q,q')$, by 
  Lemma~\ref{l:wgpFacts} we see in turn that each of $\tp(q'/q)$, $\tp(q',q)$, 
  and $\tp(g)$ is wgp.
 
 
  Let $Y:=\locus^0(g)\subseteq \G_a^2 \rtimes \G_m$ and $Z:=\locus^0(p_1)\subseteq P_1$. Let $H:=\langle YY^{-1}\rangle \leq  \G_a^2 \rtimes \G_m$. Note that for any $p=[0:y:z:w]\in P_1$ and any $h=((a,b),c)\in \G_a^2 \rtimes \G_m$, we have $h(p)=[0:cy:z+ay:w+by]=p$ if and only if $ay=(c-1)z$ and $by=(c-1)w$. If $c\neq 1$, then $h$ only fixes the projective point $[0:c-1:a:b]$; then since $Z$ is infinite, $h$ cannot fix $Z$ pointwise. If $c=1$, then $h$ fixes the line $L_0=\{[0:0:z':w']:z',w'\}=P_1\cap P_2$, which is {\cstr}. Suppose a non-trivial element $h\in H$ fixes $Z$ pointwise. Then by the argument above, we must have $h=((a,b),1)$ for some $a,b$ and $Z=L_0$. Then $p_1$ is in a {\cstr} line, hence $(p_1,p_2,q)$ is planar by Lemma~\ref{l:wESTriLine}, and we are done. Therefore, we may assume no non-trivial element $h\in H$ fixes $Z$ pointwise. We may now apply Theorem \ref{t:cohAct} and conclude that $H \leq \G_a^2 \rtimes \G_m$ is a nilpotent algebraic subgroup.

  \begin{claim*}
    $H$ is of one of the following forms:
    \begin{enumerate}[(a)]\item $(V,1)$ for a vector subgroup $V \leq  \G_a^2$;
    \item $\{((\lambda-1)v, \lambda) : \lambda \in \G_m \}$ for some $v \in 
    \G_a^2$.
    \end{enumerate}
  \end{claim*}
  \begin{proof}
    We use only that $H$ is a connected nilpotent algebraic subgroup of 
    $\G_a^2 \rtimes \G_m$.
    Since $\G_a^2 \rtimes \G_m$ is solvable non-nilpotent and connected,
    $\dim(H) \leq  2$, so $H$ is abelian.

    Suppose $H$ is not of the form in (a).
    Then the projection $\pi_2(H) \leq  \G_m$ must be the whole of $\G_m$ since 
    $H$ is connected,
    so say $(v,2) \in H$.
    Then if $(w,\lambda) \in H$,
    then $(v,2)\cdot(w,\lambda) = (\lambda v+w,2\lambda) = (2w+v,\lambda2) = 
    (w,\lambda)\cdot(v,2)$,
    so $w=(\lambda-1)v$.
  \end{proof}

  Let $C:=\locus^0(q)\subseteq P_3$.
  Say $q=[a:1:c:d]$ and $q'=[1:b':c':d']$. 

  If $H \leq  (\G_a^2,1)$, then $g\in (\G_a^2,1)$,
  so $b'a = 1$.
  Then $q = [1:b':cb':db']$,
  so $q$ is contained
  in the plane $\pi$ spanned by $q'$ and $P_1 \cap P_2 = \{[0:0:z:w] : z,w\}$. Note that $q\in\pi$, $q'\ind^0 q$, and $\pi$ is $q'$-{\cstr}. Therefore, $C\subseteq \pi$, which is collinear with $P_1$ and $P_2$. By assumption, $P_3$ is not collinear with $P_1$ and $P_2$. Thus, $C\subseteq \pi\cap P_3$ is a line. By Lemma \ref{l:wESTriLine}, $(p_1,p_2,q)$ is planar.

  It remains to consider case (b) of the Claim.
  So say $H = \{((\lambda-1)v, \lambda) : \lambda \in \G_m \}$
  where $v=(v_1,v_2) \in \G_a^2$.
  Now $g = ((c'a-c,d'a-d),b'a) \in H$, so $(c'a-c,d'a-d) = (b'a-1)v$,
  so $(a,1,c,d) = a(1,0,c'-b'v_1,d'-b'v_2) + (0,1,v_1,v_2)$.
  So $q$ is on the line $L$ spanned by $[1:0:c'-b'v_1:d'-b'v_2]$ and $[0:1:v_1:v_2]$.
  But $\{v\}$ is {\cstr}: indeed, $Y$ and hence $H$ is {\cstr}, and $\{(v,2)\}$ 
  is the fibre over $2 \in \G_m$ of the projection.
  Hence $L$ is $q'$-{\cstr}.
  Since $q\ind^0 q'$ (since $q\ind^{\bdl} q'$ and 
  $\tp(q)$ is wgp), we get $C\subseteq L$. As $\dim^0(C)\geq 1$, we must have 
  $C=L$, and so $C$ is a line. By Lemma~\ref{l:wESTriLine} again, 
  $(p_1,p_2,q)$ is planar.
  \qedhere


\end{proof}

\subsection{Quadric surface}
\label{s:quadric}
Suppose that the cubic surface $S \subseteq  \P^3$ has two smooth irreducible 
components. Since $S$ is cubic, one component is a plane $P$, and the other is 
a smooth quadric surface $Q$.

We will prove (in Corollary~\ref{c:mainOrchardQuadric} below) that every
wES-triangle for $R_S$ is planar. For this purpose, we are again free to transform
$S$ by {\cstr} projective linear transformations of $\P^3$, since they preserve
collinearity and planarity.

By the principal axis theorem, see e.g.\ \cite[IV.1.4, Theorem 1]{serre},
after applying a projective linear transformation
we may assume $Q$ has projective equation $\sum_{i<4}\lambda_ix_i^2=0$ for 
some $\lambda_i \in \K$.
Since $Q$ is smooth we must have $\lambda_i \neq  0$ for all $i$,
so by a further projective linear transformation we may assume
$$Q = \{ [x_0:x_1:x_2:x_3] : \sum_{i<4} x_i^2=0 \} = \{ [\x] : \x\cdot\x=0 
\},$$
where we denote by $\cdot$ the scalar product on $\K^4$ defined by $\x\cdot\y=\sum_{i<4} x_iy_i$.

(It would be tempting at this point to apply a further projective linear 
transformation to send $P$ to the plane $x_3=0$ ``at infinity'', and work in 
affine co-ordinates $(x_0,x_1,x_2)$ with $Q$ as the ``sphere'' $\sum_{i<3} 
x_i^2=1$. Some of the constructions below become more transparent with this 
setup, and the reader may want to consider them in these terms. However, this 
is possible only in the case that, writing $P = \{ [\x] : \d\cdot\x = 0 \}$, 
we have $\d\cdot\d \neq  0$, so to handle the case where $\d\cdot\d \neq  0$ and the case where $\d\cdot\d =0$ uniformly we will continue 
to work projectively.)

Now for $x \in \P^3 \setminus Q$, denote by $\gamma_x : Q \rightarrow  Q$ the involution 
defined by collinearity with $x$; more precisely, the map such that for any $y 
\in Q$, the divisor of the intersection of $Q$ with the line through 
$x$ and $y$ is precisely $y+\gamma_x(y)$.


For $x \in \P^3$, let $\widetilde x$ be its lift to a one-dimensional subspace of 
$\K^4$,
namely $\widetilde x = \{ \x : [\x] = x \} \cup \{0\}$,
and for $A \subseteq  \P^3$ let $\widetilde A := \bigcup_{a \in A} \widetilde a \subseteq  \K^4$.
Denote the orthogonal space to a subspace $V \leq  \K^4$ by
$V^\perp = \{ w : \forall v \in V.\; w\cdot v = 0 \}$.
We use the same notation for a projective subspace $P \subseteq  \P^3$,
i.e.\ $P^\perp$ is the projective subspace with $\widetilde {P^\perp} = (\widetilde P)^\perp$.

\begin{lemma} \label{l:geiserRefl}
  Let $x \in \P^3 \setminus Q$.
  Let $\widetilde \gamma_x$ be the element of $O_4 \setminus \SO_4 \subseteq  \GL_4$ which is the 
  identity on $\widetilde x^\perp$ and acts as $-1$ on $\widetilde x$.
  Then $\gamma_x$ is the restriction to $Q$ of the image in $\PO_4 \leq  \PGL_4 
  \cong \Aut(\P^3)$ of $\widetilde \gamma_x$ under projectivisation.
\end{lemma}
\begin{proof}
  These conditions do uniquely define an element $\widetilde \gamma_x \in \GL_4$, since 
  $\widetilde x^\perp \cup \widetilde x$ spans $\K^4$, since $x \notin Q$. The matrix with respect 
  to a corresponding orthonormal basis is then $\operatorname{diag}(-1,1,1,1)$, so 
  $\widetilde \gamma_x \in O_4 \setminus \SO_4$.
  In particular, $\widetilde \gamma_x$ preserves $\cdot$, so its image $\gamma'_x$ in 
  $\PGL_4$ does restrict to a map $Q \rightarrow  Q$.

  Let $y = [\y] \in Q$. It remains to check that $y,x,\gamma'_x(y)$ are 
  collinear.
  But indeed, $\widetilde \gamma_x(\y) \in \y + \widetilde x$, so $\widetilde y,\widetilde x,\widetilde {\gamma'_x(y)}$ are 
  linearly dependent.
\end{proof}

Compositions of pairs of these involutions therefore generate a subgroup of 
$\PSO_4$. We aim to apply Theorem~\ref{t:cohAct}, and so we are interested in 
which such subgroups are nilpotent, as well as which subvarieties of $Q$ have 
non-trivial pointwise stabilisers.
The following fact clarifies these matters.

\begin{fact} \label{f:PSO4}
  \begin{enumerate}[(i)]  \item
    Any connected nilpotent algebraic subgroup of $\PSO_4$ is abelian.
  \item
    For any $\sigma \in \PSO_4 \setminus \{1\}$, the fixed set $\Fix(\sigma) \subseteq  Q$ is 
    a finite union of points and lines.
  \end{enumerate}
\end{fact}
\begin{proof}
    As explained in e.g.\ \cite[\S 18.2]{fultonHarris}, the action of $\PSO_4$
    on $Q$ can be understood as follows.
    We have $Q \cong \P^1\times\P^1$ (in fact, $Q$ is the image of the Segre embedding of 
    $\P^1\times\P^1$ into $\P^3$; see \cite[Exercise~I.2.15]{Hart}),
    and $\PSO_4 \cong  \PGL_2 \times \PGL_2$. Fix some such isomorphisms and treat them as identities.
    Then the action of $\PSO_4$ on $Q$ respects these product decompositions,
    i.e.\ $(\alpha,\beta)(x,y) = (\alpha(x),\beta(y))$, and each subvariety
    $\{x\} \times \P^1 \subseteq \P^3$ and $\P^1 \times \{y\} \subseteq \P^3$
    is a line. The stated results can be deduced as follows:

  \begin{enumerate}[(i)]
  \item
    If $N \leq \PSO_4$ is connected nilpotent, then so are its images in 
    $\PGL_2$ under the projection maps, so it suffices to verify that any 
    nilpotent connected algebraic subgroup of $\PGL_2$ is abelian. This 
    follows for example from the easily verified fact that any nilpotent Lie 
    algebra of dimension $\leq 2$ is abelian, and the simplicity of $\PGL_2$.

  \item Say $\sigma = (\alpha,\beta)$.
    Then $\Fix(\sigma) = \Fix(\alpha) \times \Fix(\beta)$.
    But $\Fix(\alpha) \subseteq  \P^1$ is either finite or the whole of $\P^1$,
    and the result follows.\qedhere
  \end{enumerate}
\end{proof}

\begin{lemma} \label{l:abColl}
  Let $a,b,c \in \P^3 \setminus Q$,
  suppose $a \notin \{b\} \cup b^\perp$,
  and suppose $\gamma_c\gamma_a$ commutes with $\gamma_b\gamma_c$.
  Then $c \in \ell \cup \ell^\perp$, where $\ell$ is the line spanned by $a,b$.
\end{lemma}
\begin{proof}
  Let $\widetilde \gamma_a,\widetilde \gamma_b,\widetilde \gamma_c$ be as in Lemma~\ref{l:geiserRefl}.
  We use the notation $E_\lambda(T)$ for the $\lambda$-eigenspace of $T$.

  \begin{claim} \label{c:fixedLine}
    $E_1(\widetilde \gamma_a\widetilde \gamma_b) = E_1(\widetilde\gamma_b\widetilde\gamma_a) = \widetilde {\ell^\perp} = \left<{\widetilde a,\widetilde b}\right>^\perp \leq  
    \K^4$.
  \end{claim}
  \begin{proof}
    By symmetry, it suffices to show $E_1(\widetilde\gamma_a\widetilde\gamma_b) = \left<{\widetilde a,\widetilde b}\right>$.
    By Lemma~\ref{l:geiserRefl}, $\left<{\widetilde a,\widetilde b}\right>^\perp$ is contained in the fixed 
    sets of both $\widetilde \gamma_a$ and $\widetilde \gamma_b$, and hence of the composition;
    for the converse, suppose $\widetilde \gamma_a\widetilde \gamma_b(\beta) = \beta$.
    Then $\widetilde \gamma_a(\beta) = \widetilde \gamma_b(\beta)$,
    so $\beta-\widetilde \gamma_a(\beta) = \beta-\widetilde \gamma_b(\beta)$ is in the 
    intersection of the $(-1)$-eigenspaces of $\widetilde \gamma_a$ and $\widetilde \gamma_b$, 
    which is trivial since $a \neq  b$,
    so $\widetilde \gamma_a(\beta) = \beta = \widetilde \gamma_b(\beta)$,
    so $\beta \in \widetilde a^\perp \cap \widetilde b^\perp = \left<{\widetilde a,\widetilde b}\right>^\perp$.
    \subqed{c:fixedLine}
  \end{proof}

  Now $\gamma_c\gamma_a\gamma_b\gamma_c = \gamma_b\gamma_c\gamma_c\gamma_a = 
  \gamma_b\gamma_a$, so $\gamma_b\gamma_a\gamma_c = \gamma_c\gamma_a\gamma_b$,
  and so $\widetilde \gamma_b\widetilde \gamma_a\widetilde \gamma_c = 
  \lambda\widetilde \gamma_c\widetilde \gamma_a\widetilde \gamma_b$ for some scalar $\lambda \in \K$.

  If $x \in \widetilde {\ell^{\perp}}$,
  then $\widetilde \gamma_b\widetilde \gamma_a\widetilde \gamma_c(x) = 
  \lambda\widetilde \gamma_c\widetilde \gamma_a\widetilde \gamma_b(x) = \lambda\widetilde \gamma_c(x)$,
  so $\widetilde \gamma_c(x) \in E_\lambda(\widetilde \gamma_b\widetilde \gamma_a)$.
  So $\dim(E_\lambda(\widetilde \gamma_b\widetilde \gamma_a)) \geq  2$.

  Suppose $\lambda \neq  1$. Now $E_1(\widetilde \gamma_b\widetilde \gamma_a)$ is the 2-dimensional 
  space $\widetilde {\ell^\perp}$ by the claim,
  so $\widetilde \gamma_b\widetilde \gamma_a \in \SO_4$ has a basis of eigenvectors with 
  eigenvalues $(\lambda,\lambda,1,1)$, so $\lambda^2 = 1$ and hence $\lambda = 
  -1$.

  Then $E_{-1}(\widetilde \gamma_b\widetilde \gamma_a) = \widetilde \ell$, since it is orthogonal to 
  $E_1(\widetilde \gamma_b\widetilde \gamma_a)$. So in particular, if $a' \in \widetilde a \setminus \{0\}$, then 
  $-a' = \widetilde \gamma_b\widetilde \gamma_a(a') = \widetilde \gamma_b(-a')$, so $-a' \in (\widetilde b)^\perp$, 
  contradicting the assumption $a \notin b^\perp$.

  So $\lambda = 1$, therefore
  $\widetilde\gamma_c(E_1(\widetilde\gamma_a\widetilde\gamma_b))
  = E_1(\widetilde\gamma_b\widetilde\gamma_a)$,
  so by the claim
  $\widetilde \gamma_c$ preserves $\widetilde{\ell^\perp}$,
  and hence $\gamma_c$ preserves $\ell^\perp$.
  Since $\widetilde \gamma_c \in O_4$, also $\gamma_c$ preserves $ \ell = 
  (\ell^\perp)^\perp$.

  Now suppose $c \notin \ell^\perp$, so say $x \in \widetilde \ell \setminus \widetilde c^{\perp}$.
  Then $\widetilde \gamma_c(x) - x \in \widetilde c \setminus \{0\}$ by Lemma~\ref{l:geiserRefl},
  but $\widetilde \gamma_c(x) \in \widetilde \ell$ as $\gamma_c$ preserves $\ell$,
  so $\widetilde c \cap \widetilde \ell \neq  \{0\}$,
  so $c \in \ell$, as required.
\end{proof}

\begin{theorem} \label{t:quadricLines}
  Suppose $a \in Q$, $c \in \P^3 \setminus Q$,
  $\tp(a)$ and $\tp(c)$ are wgp,
  and $c \ind ^{\bdl} a$ and $c \ind ^{\bdl} \gamma_c(a)$.

  Then either $\locus^0(a)$ or $\locus^0(c)$ is contained in a line.
\end{theorem}
\begin{proof}
  Suppose not.

  Let $Z := \locus^0(a) \subseteq  Q$.
  Let $\theta : \P^3 \setminus Q \rightarrow  \PO_4$ be the ({\cstr}) embedding 
  given by Lemma~\ref{l:geiserRefl}, so $\theta(c)a = \gamma_c(a)$,
  and let $Y := \locus^0(\theta(c)) \subseteq  \PO_4$.
  Note that (by considering determinants) $Y^{-1}Y = Y^2 \subseteq  \PSO_4$.

  By Fact~\ref{f:PSO4}(ii),
  since $Z$ is not contained in a line,
  no non-trivial element of $\PSO_4$ fixes $Z$ pointwise.

  By Theorem~\ref{t:cohAct}, the subgroup $H$ of $\PSO_4$ generated by $Y^{-1}Y = 
  Y^2$ is nilpotent, and hence abelian by Fact~\ref{f:PSO4}(i).

  Let $(c_1,c_2) \in \locus^0(c)^2$ be generic.
  If $c_1 \in (c_2)^\perp$, then this closed condition holds for all $(x,y) \in
  \locus^0(c)^2$, which implies that $\locus^0(c) \subseteq  Q$, contradicting $c \notin 
  Q$.

  Let $\ell$ be the line spanned by $c_1$ and $c_2$.
  Then by Lemma~\ref{l:abColl},
  $\locus^0(c) \setminus Q \subseteq  \ell \cup \ell^{\perp}$; but $\locus^0(c)$ is 
  irreducible, so it must be contained in a line, contrary to assumption.
\end{proof}

\begin{corollary} \label{c:mainOrchardQuadric}
  Every wES-triangle for $R_S$ is planar.
\end{corollary}
\begin{proof}
  Let $(a_1,a_2,a_3) \in R_S$ be a wES-triangle.

  Permuting, we may assume that $a_3 \in P \setminus Q$ and $a_1,a_2 \in Q$, so $a_2 = 
  \gamma_{a_3}(a_1)$. Then by Theorem~\ref{t:quadricLines}, $\locus^0(a_1)$ or 
  $\locus^0(a_3)$ is contained in a line, and we conclude by 
  Lemma~\ref{l:wESTriLine}.
\end{proof}

\section{The irreducible case}
\label{s:irreducible}
Let $S \subseteq  \P^3(\K)$ be an irreducible smooth cubic surface over $\K$.
Let $\L$ be as in Definition~\ref{d:ESTri}; in particular, $S$ is {\cstr}.

\subsection{Lines in $S$}
\label{ss:lines}
\begin{definition} \label{d:good}
  $x \in S$ is \defn{good} if $x$ is not contained in any line contained in 
  $S$.
\end{definition}

It will be convenient to work with a modified version of $R_S$ which allows 
only good points:

\begin{definition} \label{d:RPrime}
  Let $R_S'(x,y,z)$ be the ({\cstr}) relation:
  $(x,y,z)$ is a collinear triple of distinct good points on $S$.
\end{definition}

\begin{fact}[{\cite[Theorem~V.4.9]{Hart}}] \label{f:lines}
  $S$ contains only finitely many (namely 27) lines.
  In particular, the set of good points is Zariski open, and so a generic 
  point of $S$ is good.
\end{fact}

Let $\Pi$ be the (Grassmannian) variety of projective planes in $\P^3$.

\subsection{Coplanarity of energy}
Let $E_S$ be the {\cstr} set
\begin{align*} E_S := \{ ((x_1,x_2),(y_1,y_2)) \in S^2 \times S^2 &: 
\text{$x_1,x_2,y_1,y_2$ are not collinear} \\
& \wedge \exists z \in S. \bigwedge_{i=1,2} R_S'(x_i,y_i,z) \} .\end{align*}

We show in Corollary~\ref{c:fixedCurves} a result on planarity of curves of
fixed points of compositions of four sufficiently generic Geiser involutions,
which we now interpret as showing that any sufficiently generic $K_{2,s}$ for
$E_S$ is almost entirely coplanar.

\begin{theorem} \label{t:coplanarK2s}
  There exists $m \in \N$ (depending only on $S$) such that
  if $a,b,c,d \in S$ and $B \subseteq  S^2$
  and $\{(a,b),(c,d)\} \times B \subseteq  E_S$
  and $\trd^0(a,b) = 4 = \trd^0(c,d)$
  and $\trd^0(a,b,c,d) > 4$,
  and $|B| \geq m$,
  then there is $(e,f) \in B$ such that $a,b,c,d,e,f$ are coplanar.
\end{theorem}
\begin{proof}
  Let $F_{a,b,c,d} := \{e \in S : \exists f \in S.\; ((a,b),(e,f)),((c,d),(e,f)) \in E_S \}$.
  Then by the definition of $E_S$, we have
  $$F_{a,b,c,d} \subseteq \{ e \in S : \exists z_1,f,z_2\in S.\; 
  (e,a,z_1),(z_1,b,f),(f,d,z_2),(z_2,c,e) \in R_S'\}.$$

\begin{center}
\begin{tikzpicture}[scale=0.4]
\draw[gray, thick] (-5,1) -- (1,-5);
\draw[gray, thick] (-1,-5) -- (5,1);
\draw[gray, thick] (5,-1) -- (-1,5);
\draw[gray, thick] (1,5) -- (-5,-1);
\filldraw[black] (-4,0) circle (1.5pt) node[anchor=east]{$e$};
\filldraw[black] (4,0) circle (1.5pt) node[anchor=west]{$f$};
\filldraw[black] (0,-4) circle (1.5pt) node[anchor=north]{$z_1$};
\filldraw[black] (0,4) circle (1.5pt) node[anchor=south]{$z_2$};
\filldraw[black] (-2,-2) circle (1.5pt) node[anchor=north east]{$a$};
\filldraw[black] (2,-2) circle (1.5pt) node[anchor=north west]{$b$};
\filldraw[black] (-2,2) circle (1.5pt) node[anchor=south east]{$c$};
\filldraw[black] (2,2) circle (1.5pt) node[anchor=south west]{$d$};
\end{tikzpicture}
\end{center}
  
  This precisely means that $F_{a,b,c,d}$ is contained in the set of strongly
  $(a,b,c,d)$-fixed points as defined in Definition~\ref{d:stronglyFixed}.
  Corollary~\ref{c:fixedCurves} then states that if $F=F_{a,b,c,d}$ is
  infinite and $a,b,c,d$ satisfy the genericity conditions in the statement,
  then $a,b,c,d$ span a plane $\pi \in \Pi$ and $F\setminus\pi$ is finite.

  By the existence of uniform bounds on cardinalities of finite fibres of a
  constructible set (known to model theorists as elimination of
  $\exists^{\infty}$ in the strongly minimal theory $\operatorname{ACF}_0$),
  there is $m$ such that for any $a,b,c,d \in S$,
  $F=F_{a,b,c,d}$ is infinite as soon as $|F| \geq m$,
  and if furthermore $a,b,c,d$ span a plane $\pi$ and $|F \setminus \pi|$ is finite
  (in particular, if the genericity conditions are satisfied),
  then $|F \setminus \pi| < m$.

  This $m$ is then as required: if $a,b,c,d$ and $|B| \geq m$ are as in the statement,
  then $F=F_{a,b,c,d}$ is infinite since $|F| = |B| \geq m$,
  so $a,b,c,d$ span a plane $\pi$
  and $|B| > |F \setminus\pi|$,
  so there is $(e,f) \in B$ with $e \in F \cap \pi$,
  and then by the collinearities, also $f \in \pi$.
\end{proof}


\subsection{Planarity of Elekes-Szabó triangles}
\label{s:irredPlanarTri}
\begin{lemma} \label{l:chromatic}
  Let $\Gamma$ be an internal set, and for $i \in \omega$ let $R_i \subseteq  
  \Gamma^2$ be an internal symmetric binary relation with finite maximum 
  degree $\leq d_i$, meaning that the neighbourhoods of points have uniformly 
  bounded finite size $|R_ix| \leq  d_i \in \N$.

  Then there exists an internal anticlique $\Gamma' \subseteq  \Gamma$ for the graph 
  $(\Gamma,\bigvee_i R_i)$ with $\bdl(\Gamma') = \bdl(\Gamma)$.

  Moreover, if $\pi : H \rightarrow  \Gamma$ is an internal map, then there is an 
  internal anticlique $\Gamma' \subseteq  \Gamma$ with $\bdl(\pi^{-1}(\Gamma')) = 
  \bdl(H)$.
\end{lemma}
\begin{proof}
  We show the ``moreover'' clause, since the other statement follows by taking 
  $\pi = \operatorname{id}$.

  Say $\Gamma = \prod_{s \rightarrow  \U} \Gamma_s$ and $R_i = \prod_{s \rightarrow  \U} 
  R_{i,s}$.

  We recursively define a descending chain of internal subsets $A_k \subseteq  
  \Gamma$ such that $A_k$ is an $R_i$-anticlique for all $i < k$ and
  $\bdl(\pi^{-1}(A_k)) = \bdl(H)$.
  Let $A_0 := \Gamma$.
  Given $A_k$, say
  $A_k = \prod_{s \rightarrow  \U} A_{k,s}$;
  then for $\U$-many $s$,
  $(A_{k,s},R_{k,s})$ is a graph with maximal degree $\leq d_k$ and hence (by 
  a standard easy argument) chromatic number at most $d_k+1$.
  Thus we can write $A_{k,s}=\bigcup_{j=0}^{d_k} E_i$ for some disjoint $R_{i+1}$-anticliques $E_0,\dots,E_{d_k}$, so   $\pi^{-1}(A_{k,s})=\bigcup_{j=0}^{d_k} \pi^{-1}(E_i)$. 
  Hence for some $i_0$ the anticlique $A_{k+1,s}:=E_{i_0}\subseteq A_{k,s}$ satisfies $|\pi^{-1}(A_{k+1,s})| \geq  
  \frac{|\pi^{-1}(A_{k,s})|}{d_k+1}$.
  Then $A_{k+1} := \prod_{s \rightarrow  \U} A_{k+1,s}$ is an internal $R_k$-anticlique 
  with $\bdl(\pi^{-1}(A_{k+1})) = \bdl(\pi^{-1}(A_k)) = \bdl(H)$, as required.

  Now by Lemma~\ref{l:intInType}, there is an internal subset
  $A \subseteq  \bigwedge_{k \in \omega} \pi^{-1}(A_k)$ with $\bdl(A) = \bigwedge_{k \in \omega} 
  \pi^{-1}(A_k) = \bdl(H)$,
  and then $\Gamma' := \pi(A)$ is as required.
\end{proof}

\begin{theorem} \label{t:mainOrchardIrred}
  Every wES-triangle for $R_S$ is planar.
\end{theorem}
\begin{proof}
  Let $(a_1,a_2,a_3) \in R_S$ be a wES-triangle.

  If some $a_i$ is not good, then by Fact~\ref{f:lines} it is contained in a 
  {\cstr} line, and we conclude by Lemma~\ref{l:wESTriLine}. So we may 
  assume that $\a \in R_S'$.

  For $\{i,j,k\} = \{1,2,3\}$ we have $\bdl(a_i) = \bdl(a_i/a_j) = 
  \bdl(a_k/a_j) = \bdl(a_k)$ and $0 < \bdl(a_i) < \infty$,
  so for notational convenience we assume (by changing the scaling parameter 
  $\xi$) that $\bdl(a_i)=1$ for $i=1,2,3$.

  \underline{\bf Step 1}:
  Suppose for a contradiction that $\trd^0(a_i) = 2$ for some $i \in \{1,2,3\}$.
  Permuting, we may assume $\trd^0(a_1) = 2$.

  Let $a_1'a_2' \equiv _{a_3} a_1a_2$ with $a_1'a_2' \ind ^{\bdl}_{a_3} a_1a_2$,
  and let $e := ((a_1,a_1'),(a_2,a_2')) \in E_S$.

  \begin{claim} \label{c:ewgp}
    $\tp(e)$ is wgp, and $\bdl(e) = 3$.
  \end{claim}
  \begin{proof}
    Since $a_1' \ind ^{\bdl} a_3$,
    we have $a_1' \ind ^{\bdl} a_1a_2$,
    and using also $\tp(a_1')$ is wgp and that $e \in \acl^0(a_1'a_1a_2)$, we 
    see by Lemma~\ref{l:wgpFacts} that each of
    $\tp(a_1'/a_1a_2)$, $\tp(a_1'a_1a_2)$, and $\tp(e)$ is wgp, and that
    $\bdl(e) = \bdl(a_1'a_1a_2) = \bdl(a_1') + \bdl(a_1a_2) = 1 + 2 = 3$.
    \subqed{c:ewgp}
  \end{proof}

  By Lemma~\ref{l:intInType}, let $E \subseteq  \tp(e)(\K)$ be an internal subset such that 
  $\bdl(E) = \bdl(e) = 3$,
  and for $i=1,2$ let $Y_i := \pi_{i,i'}(E)$,
  where $\pi_{i,i'} : E_S \rightarrow  S^2$ is the co-ordinate projection such 
  that $\pi_{i,i'}(e) = (a_i,a_i')$.
  Then $Y_i \subseteq  \tp(a_ia_i')(\K)$ and hence $\bdl(Y_i) \leq 2$,
  and the fibres of $E$ under $\pi_{i,i'}$ are in bijection via another 
  projection with a subset of $\tp(a_{3-i})$,
  so $3 = \bdl(E) \leq  1 + \bdl(Y_i)$.
  So $\bdl(Y_i) = 2$.

  Now $a_1 \ind ^0 a_1'$ since $\tp(a_1)$ is wgp, so every element of $Y_1$ is 
  generic in $S^2$ over $\acl^0(\emptyset )$.
  If two elements of $Y_1$ are interalgebraic over $\acl^0(\emptyset )$, then this is 
  witnessed by a generically finite-to-finite correspondence of $S^2$ with 
  itself,
  so the relation of interalgebraicity on $Y_1$ is the union of internal 
  relations with bounded maximal degree.
  Applying Lemma~\ref{l:chromatic},
  we obtain an internal set
  $Y_1' \subseteq  Y_1$ such that setting $E_0 := E \cap \pi_{1,1'}^{-1}(Y_1')$, we 
  have $\bdl(E_0) = \bdl(E) = 3$, and if $x,x' \in Y_1'$ and $x \neq  x'$,
  then $\trd^0(x) = 4 = \trd^0(x')$ and $\trd^0(x,x') > 4$.
  As above, $\bdl(Y_1') = 2$.

  Now let $E' := E_S \cap (Y_1' \times Y_2) \supseteq E_0$.
  For $s \in \N^\U$, let
  \[F_s := \bigcup \{ \{x,x'\} \times (B \cap \pi^2) : \{x,x'\} \times B \subseteq  
  E',\; x \neq  x',\; |B| = s,\; \pi \in \Pi,\; x,x' \in \pi^2 \}.\]
  Note that $F_s$ is internal, since the existential quantification of $B$ can
  be taken to be over the set of all internal subsets of $S^2$, which is
  itself an internal set.

  Let $m \in \N$ be as in Theorem~\ref{t:coplanarK2s}.
  By that result, if $s\geq m$ then $F_s$ is a $(2,s)$-transversal of $E'$,
  as defined in Definition~\ref{d:transversalPsf}.

  Now $\bdl(E') = 3 = \bdl(Y_1') + \frac12\bdl(Y_2) > 
  \max(\bdl(Y_1'),\bdl(Y_2)) = 2$,
  so by Corollary~\ref{c:distIncBdl}, for some $s$ with $\bdl(s) > 0$,
  setting $F := F_s$,
  we have $\bdl(E' \setminus F) < \bdl(E') = 3$,
  and so also $\bdl(E_0 \setminus F) < 3$.

  We may assume that $E_0$ and $F$ are also definable in $\L$, and that $e \in 
  E_0$; indeed,
  we can add $E_0$ and $F$ as new predicates to $\L$ and close off to a 
  countable language $\L'$ for which $\bdl$ is continuous,
  then since $E_0 \subseteq  \tp^{\L}(e)(\K)$, we have $\bdl(\tp^{\L}(e) \cup \{ x \in 
  E_0 \}) = \bdl(E_0) = 3$, so we can complete this partial type to a complete 
  $\L'$-type $p$ with $\bdl(p) = 3$, and so we can assume $e \vDash  p$ without 
  changing $\bdl(e)=3$ nor the $\L$-type of $e$.
  Note that we still have $\bdl(a_i) = 1$, since
  $$1 = \bdl(S_1) \geq  \bdl(a_1) = \bdl(e) - 
  \bdl(e/a_1) = \bdl(e) - \bdl(a_2,a_1'/a_1) \geq  3 - 2 = 1.$$

  Since $\bdl(E_0 \setminus F) < \bdl(e)$, we must have $e \in F$.

  The co-ordinates of any element of $E_S$ span a plane in $\P^3$.
  Let $\pi_E : E_S \rightarrow  \Pi$ be the corresponding map.
  Let $\pi := \pi_E(e)$ be the plane spanned by $e$.

  Then by the definition of $F$,
  $\pi_E^{-1}(\pi) \cap F$ contains some $\{x,x'\} \times (B \cap \pi^2)$, where 
  by Theorem~\ref{t:coplanarK2s} we have $|B \cap \pi^2| \geq  |B| - m = s - m$, so 
  $\bdl(B \cap \pi^2) \geq  \bdl(s)$.
  Then by Lemma~\ref{l:genericFibre},
  $$\bdl(e/\pi) = \bdl(\pi_E^{-1}(\pi) \cap F) \geq  \bdl(B \cap \pi^2) \geq  \bdl(s) > 0.$$

  Set $\eta := \frac13\bdl(e/\pi) > 0$.
  Then $\bdl(\pi) = \bdl(e) - \bdl(e/\pi) = 3 - 3\eta$.
  Also $\eta < 1$ since $\bdl(e/\pi) < \bdl(e) = 3$ by wgp,
  since $\trd^0(\pi) > 0$ since $\trd^0(a_1) > 1$.

  \begin{claim} \label{c:bdlEqEta}
    $\bdl(a_1/\pi) = \bdl(a_2/\pi) = \bdl(a_2/\pi a_1) = \eta$.
  \end{claim}
  \begin{proof}
    First, we show that if $a \in \pi$ with $\bdl(a) = 1$, then $\bdl(a/\pi) \leq  
    \eta$.

    Indeed, let $(a, a', a'')$ be a $\bdl$-independent sequence in 
    $\tp(a/\pi)$,
    i.e.\ let $a' \vDash  \tp(a/\pi)$ with $\bdl(a'/a\pi) = \bdl(a/\pi)$
    and let $a'' \vDash  \tp(a/\pi)$ with $\bdl(a''/aa'\pi) = \bdl(a/\pi)$.
    If $\bdl(a/\pi) = 0$, then certainly $\bdl(a/\pi) \leq  \eta$.
    Otherwise, $a,a',a''$ are three distinct points on $\pi$,
    so $\pi \in \acl^0(a,a',a'')$.
    Then
    $$\bdl(a/\pi) + \bdl(\pi) - \bdl(a) = \bdl(\pi/a) = \bdl(a',a''/a) - 
    \bdl(a',a''/a,\pi) \leq  2 - 2 \bdl(a/\pi),$$
    so $3\bdl(a/\pi) \leq  1 + 2 - \bdl(\pi) = \bdl(e) - \bdl(\pi) = \bdl(e/\pi)$.

    Hence $\bdl(a_1/\pi),\bdl(a_2/\pi),\bdl(a_1'/\pi) \leq  \frac13\bdl(e/\pi)$.
    On the other hand, $\bdl(e/\pi) = \bdl(a_1,a_2,a_1'/\pi)$, so by 
    additivity we must have the desired equalities.
    \subqed{c:bdlEqEta}
  \end{proof}

  Consider the $a_1$-{\cstr} projective plane $\Pi_{a_1}$ of 
  projective lines in $\P^3$ passing through $a_1$, and the corresponding 
  projection map $\pi_{\Pi_{a_1}} : \P^3 \setminus \{a_1\} \rightarrow  \Pi_{a_1}$.
  Then $p := \pi_{a_1}(a_2)$ is a point lying on the line $l := 
  \pi_{a_1}(\pi)$ in $\Pi_{a_1}$.

  Now $a_1$ is good, and $a_2 \in S$, so $a_2$ is interalgebraic with $p$ 
  over $a_1$.
  Also $\pi$ is interalgebraic with $l$ over $a_1$,
  so
  $$\bdl(l/a_1) = \bdl(\pi/a_1) = \bdl(a_1/\pi) + \bdl(\pi) - \bdl(a_1) = 
  2-2\eta,$$
  and
  $$\bdl(pl/a_1) = \bdl(l/a_1) + \bdl(p/la_1) = 2-2\eta + \bdl(a_2/\pi a_1) 
  = 2-2\eta + \eta = 2-\eta.$$
  Also $\bdl(p/a_1) = \bdl(a_2/a_1) = \bdl(a_2) = 1$.

  By Corollary~\ref{c:Toth}, writing $I$ for the point-line incidence relation on 
  $\Pi_a$ and setting $A_1 := \tp(p/a_1)$ and $A_2 := \tp(l/a_1)$, we obtain
  \begin{align*} 2-\eta = \bdl(pl/a_1) &\leq  \bdl((A_1\times A_2)\cap I) \\
  &\leq  \max(\frac23\bdl(p/a_1) + \frac23\bdl(l/a_1),\bdl(p/a_1),\bdl(l/a_1)) 
  \\
  &\leq  \max(\frac23 + \frac23(2-2\eta), 1, 2-2\eta) \\
  &= \max(2-\frac43\eta,1,2-2\eta), \end{align*}
  which contradicts $0 < \eta < 1$.

  \underline{\bf Step 2}:
  So $\trd^0(a_i) < 2$ for $i=1,2,3$; since $\bdl(a_i) = 1 > 0$, actually 
  $\trd^0(a_i) = 1$.

  Now for $i\neq j$ we have $a_i \notin \acl^0(a_j)$, since $\bdl(a_i/a_j) = 1 
  > 0$,
  so $\trd^0(a_i) = 1 = \trd^0(a_i/a_j)$.
  In this situation, Corollary~\ref{c:curvesCoplanar} yields that 
  $a_1,a_2,a_3$ lie on a {\cstr} plane, as required.
\end{proof}

\section{Elekes-Szabó for collinearity on a cubic surface}
\label{s:concl}
In this section, we put everything together and obtain our main conclusions.

\begin{theorem} \label{t:mainOrchard}
  Suppose $S \subseteq  \P^3(\K)$ is a cubic surface with smooth irreducible 
  components.
  Then $R_S$ has no non-planar wES-triangles if and only if $S$ is not the union of
  three collinear planes.
\end{theorem}
\begin{proof}
  We consider cases.
  \begin{itemize}\item
  $S$ is irreducible. By Theorem~\ref{t:mainOrchardIrred} (noting that we
  quantified at the start of that section over suitable $\L$), $R_S$ has no 
  non-planar wES-triangle.
  \item $S$ has two irreducible components. Then $S$ is the union of a plane 
    and a smooth quadric surface, and we conclude similarly by Corollary~\ref{c:mainOrchardQuadric}.
  \item $S$ has three irreducible components. Then $S$ is the union of three 
  distinct planes, and we have two subcases.
    \begin{itemize}
      \item The planes are not collinear. We conclude by Proposition~\ref{p:mainOrchardPlanar}.
      \item The planes are collinear. Then $R_S$ does have non-planar
        wES-triangles (for some $\L$) by Example~\ref{e:collinearPlanes}.
        \qedhere
    \end{itemize}
  \end{itemize}
\end{proof}

\subsection{Finitary consequences}
\label{ss:conclFin}
We first extract the consequences for pseudofinite sets of having no 
non-planar wES-triangles.

\begin{lemma}\label{l:wESPlanarInternal}
  Suppose $R_S$ has no non-planar wES-triangles.

  Let $A_1,A_2,A_3 \subseteq  \P^3$ be internal sets with $\bdl(A_i) = 1$ and
  $\bdl(\prod_i A_i \cap R_S) = 2$.

  Then for some plane $\pi\subseteq \P^3$,
  we have $\bdl(\prod_i A_i \cap R_S \cap \pi^3) = 2$
  (and hence $\bdl(A_i \cap \pi) = 1$ for $i=1,2,3$).
\end{lemma}
\begin{proof}
  Fix a language $\L$ satisfying Assumptions~\ref{a:L} in which each $A_i$ is 
  $\emptyset$-definable and such that $S$ is {\cstr}.

  Let $\a = (a_1,a_2,a_3) \in \prod_i A_i \cap R_S$ with $\bdl(\a) = 2$,
  so $\bdl(a_i) = 1 = \bdl(a_i/a_j)$ for $i\neq j$.

  By Lemma~\ref{l:wgpEventually}(i), let $c$ be such that $\bdl(c)=0$ and 
  $\tp(\a/c)$ is wgp. By Lemma~\ref{l:wgpConstants}, replacing $\L$ with 
  $\L^0(c)$, we may assume that $\tp(\a)$ is wgp.

  By Lemma~\ref{l:wgpFacts}(\ref{li:wgpImage}), $\tp(a_i)$ is wgp for 
  $i=1,2,3$.

  Now $\a$ is a wES-triangle, so by assumption
  we obtain a {\cstr} plane $\pi \ni a_1,a_2,a_3$.
  Then $(\x \in \prod A_i \cap R_S \cap \pi^3) \in \tp(\a)$,
  so $$2 = \bdl(\prod_i A_i \cap R_S) \geq  \bdl(\prod_i A_i \cap R_S \cap \pi^3) 
  \geq  \bdl(\a) = 2,$$
  hence  $\bdl(\prod_i A_i \cap R_S \cap \pi^3) =2$. 
  
  As the relation $R_S$ is algebraic in each co-ordinate,  for distinct $i,j\in \{1,2,3\}$ we have  $$2=\bdl(\prod_i A_i \cap R_S \cap \pi^3) \leq \bdl((A_i\cap \pi)\times (A_j\cap \pi))=\bdl(A_i\cap \pi)+ \bdl(A_j\cap \pi)\leq 1+1=2,$$ so $\bdl(A_i\cap \pi)=1$.
\end{proof}

Finally, we deduce a finitary statement.
The version stated in the introduction, 
Theorem~\ref{t:mainOrchardFiniteIntro}, is precisely the symmetric case 
$A_1=A_2=A_3=A$ of the following statement.
\begin{corollary} \label{c:mainOrchardFinite}
  For any $\epsilon > 0$ there exists $\eta > 0$ and $N_0 \in \N$ such that
  if $S \subseteq  \P^3(\C)$ is a cubic surface with smooth irreducible components
  which is not the union of three collinear planes, 
  and $A_1,A_2,A_3 \subseteq  S$ are finite subsets with $|A_1| = |A_2| = |A_3| =: N 
  \geq  N_0$
  such that $|R_S \cap \prod_i A_i| > N^{2-\eta}$,
  then for some plane $\pi \subseteq  \P^3(\C)$
  we have $|R_S \cap \pi^3 \cap \prod_i A_i| > N^{2-\epsilon}$.
\end{corollary}
\begin{proof}
  Otherwise, there is $\epsilon > 0$ and for each $n \in \N$ such a cubic 
  surface $S_n \subseteq  \P^3(\C)$ and finite subsets giving a counterexample with 
  $|(A_i)_n| = N > n$ and $\eta = \frac1n$.

  Let $S := \prod_{n \rightarrow  \U} S_n \subseteq  \P^3(\K)$. We can write each $S_n$ as the 
  zero set of an irreducible cubic polynomial, or the product of irreducible 
  linear or quadratic polynomials. It is a standard and easy fact that 
  smoothness and irreducibility are polynomially definable properties of the 
  coefficients of such a polynomial, so $S$ is also a cubic surface with 
  smooth irreducible components. In the case of three planes, collinearity is 
  also definable, so $S$ is not the union of collinear planes.
  We have $\prod_{n \rightarrow  \U} R_{S_n} = R_S$.

  Setting $A_i := \prod_{n \rightarrow  \U} (A_i)_n$ and $\bdl := \bdl_\xi$ where $\xi 
  := |A_i|$, we get $\bdl(A_i)=1$ for $i\in \{1,2,3\}$ but, by the choice of $(A_i)_n$, for any plane $\pi\subseteq  \P^3(\K)$ we have $\bdl(\prod_i A_i \cap R_S \cap \pi^3) \leq 2-\epsilon<2$. Thus $A_1$, $A_2$, and $A_3$ do not satisfy the conclusion of Lemma~\ref{l:wESPlanarInternal}, hence contradicting Theorem~\ref{t:mainOrchard}.
\end{proof}

\begin{question}
  How does $\eta$ depend on $\epsilon$? Can we take $\eta = \frac12\epsilon$?

  Also, does the following incomparable version of 
  Corollary~\ref{c:mainOrchardFinite} hold
  (corresponding to a version of Theorem~\ref{t:mainOrchard} for the ``fine''
  pseudofinite dimension)?
  For any $c > 0$ there exist $C > 0$ and $N_0 \in \N$ such that
  if $S \subseteq  \P^3(\C)$ is a smooth irreducible cubic surface
  which is not the union of three collinear planes,
  and $A_1,A_2,A_3 \subseteq  S$ are finite subsets with $|A_1| = |A_2| = |A_3| =: N \geq  N_0$
  such that $|R_S \cap \prod_i A_i| > CN^2$,
  then for some plane $\pi \subseteq  \P^3(\C)$
  we have $|R_S \cap \pi^3 \cap \prod_i A_i| > cN^2$.
\end{question}


\appendix
\section{Coherent group actions}
\label{s:cohAct}

In this appendix, we apply the results of \cite{BGT-lin} on the nilpotence of 
characteristic 0 algebraic groups admitting Zariski-dense approximate 
subgroups, or rather their reformulation in terms of \emph{wgp-coherence} 
given in \cite{bgth}, to obtain an analogous criterion (Theorem~\ref{t:cohAct} 
below) for an algebraic group action.
This is the key tool used in our analysis of reducible cubic surfaces in
\S\ref{s:reducible}.

In this appendix, we denote group operations and actions by concatenation.

\begin{definition}
  We say the graph $\Gamma_G = \{(x,y,z) : xy=z\}$ of the group operation of a 
  connected algebraic group $G$ over $\K$ is \defn{wgp-coherent} if there are 
  wgp in $G$ internal subsets $A_1,A_2,A_3 \subseteq  G$ such that $\bdl(\prod_j A_j 
  \cap \Gamma_G) = 2\bdl(A_1)= 2\bdl(A_2)=2\bdl(A_3)$.
\end{definition}

\begin{fact}[{\cite[Theorem~2.14$(i)\Rightarrow(iv)$, Remark~2.15]{bgth}}]
  \label{f:wgpCohNilp}
  If the graph of the group operation of a connected algebraic group $G$ over 
  $\K$ is wgp-coherent, then $G$ is nilpotent.
\end{fact}

\begin{lemma} \label{l:unifWgpGroup}
  Let $G$ be a connected algebraic {\cstr} group.
  Let $(V_b)_b$ be a {\cstr} left-invariant family of proper 
  subvarieties of $G$,
  where left-invariance means that for all $b$ and $g \in G$ there is $b'$ such 
  that $gV_b = V_{b'}$.

  Suppose $g \in G$ is generic in $G$ with $\tp(g)$ wgp,
  and let (by Lemma~\ref{l:wgpWgp} and Remark~\ref{r:wgpUnif}) $\epsilon>0$ be such that 
  $\bdl(\tp(g) \cup \{ x \in V_b \}) \leq  \bdl(g) - \epsilon$ for all $b$.

  Let $h \in G$ with $h \ind ^{\bdl} g$.
  Then also $\bdl(\tp(hg) \cup \{ x \in V_b \}) \leq  \bdl(hg) - \epsilon$
  for all $b$.
\end{lemma}
\begin{proof}
  Fix $b$.
  Let $\alpha := \bdl(\tp(hg) \cup \{ x \in V_b \})$.
  We may assume that $\bdl(hg/b) = \alpha$ and $h,g \ind ^{\bdl}_{hg} b$;
  indeed, let (by Extension) $c$ realise an extension of the partial type 
  $\tp(hg) \cup \{ x \in V_b \}$ to a complete type over $b$ such that also 
  $\bdl(c/b) = \alpha$,
  then say $h',g',c \equiv  h,g,hg$ with $h',g' \ind ^{\bdl}_{c} b$ (we can find such 
  $h',g'$ by applying Extension to the type over $c$ corresponding to 
  $\tp(h,g/hg)$). Since $\tp(h',g',c) = \tp(h,g,hg)$ and our statement only 
  concerns this type, we may assume $(h,g,hg) = (h',g',c)$.

  Then \begin{align*} \alpha &= \bdl(hg/b) \\
  &= \bdl(g,h/b) - \bdl(g,h/hg,b) \\
  &= \bdl(g,h/b) - \bdl(g,h/hg),\end{align*}
  and
  $$\bdl(g,h/b) = \bdl(h/b) + \bdl(g/h,b) \leq  \bdl(h) + \bdl(g) - \epsilon$$ 
  (since $g \in h^{-1}V_b$ and $h^{-1}V_b=V_{b'}$ for some $b'$ by left-invariance).
  But
  $\bdl(h) + \bdl(g) - \bdl(g,h/hg) = \bdl(hg)$
  (by independence of $h$ from $g$),
  so we conclude $\alpha \leq  \bdl(hg) - \epsilon$, as required.
\end{proof}

\begin{theorem} \label{t:cohAct}
  Let $G$ be an algebraic group acting on an irreducible variety $X$ by 
  regular maps, and suppose $G$, $X$, and the action are all 
  {\cstr}.

  Let $g \in G$ and $x \in X$, and suppose $\tp(g)$ and $\tp(x)$ are wgp
  and $g \ind ^{\bdl} x$ and $g \ind ^{\bdl} gx$.

  Let $Y := \locus^0(g)$,
  and let $H := \left<{Y^{-1}Y}\right> \leq  G$.

  Suppose further that no non-trivial element of $H$ fixes $Z := 
  \locus^0(x)$ pointwise,
  i.e.\ $\{ h \in H : \forall z \in Z.\; hz=z \} = \{ 1 \}$.

  Then $H$ is nilpotent.
\end{theorem}

\begin{proof}
  First, we argue that by changing $x$ and the action,
  we may assume that $x$ has trivial stabiliser under $H$.

  \begin{claim} \label{c:trivStab}
    For some $n \in \N$, generics of $Z^n$ have trivial stabiliser under the 
    diagonal action of $H$ on $Z^n$.
  \end{claim}
  \begin{proof}
    Let $(z_i)_{i \in \omega}$ be a sequence of independent generics of $Z$,
    i.e.\ each $z_i$ is generic in $Z$ over $z_{<i} := \{z_j : j < i \}$.
    Then the stabilisers $(H_{z_{<m}})_m$ form a descending chain,
    which stabilises by Noetherianity of algebraic subgroups. 
    So for some $n$, $H_{z_{<n}} = H_{z_{<n}z_n}$.
    Then $z_n$ lies in the set
    $$\{ z \in Z : \forall h \in H.\; (hz_{<n} = z_{<n} \Rightarrow  hz=z) \} = 
    \bigcap_{h \in H_{z_{<n}}} \Fix(h) ,$$
    which is closed (since the action is by regular maps) and is 
    $z_{<n}$-{\cstr},
    so is equal to $Z$ by genericity of $z_n$.
    Hence any element of $H_{z_{<n}}$ is in the pointwise stabiliser of $Z$,
    which is trivial by assumption.
    \subqed{c:trivStab}
  \end{proof}

  Now let $n$ be as in the claim, and let $\bar x = (x_i)_{i<n}$ be a 
  $\bdl$-independent $n$-tuple of copies of $x$ over $g$,
  i.e.\ $x_i \equiv _g x$ and $x_i \ind ^{\bdl}_g x_{<i}$.
  By wgp, $\bar x$ is generic in $Z^n$.
  So $\bar x$ has trivial stabiliser with respect to the diagonal action of 
  $H$ on $X^n$,
  and $g \ind ^{\bdl} \bar x$ and $g \ind ^{\bdl} g\bar x$.
  So by redefining $x$ as $\bar x$, we may assume that $x$ has trivial 
  stabiliser with respect to $H$.

  Now let $g_0 := g$, and iteratively define $g_i$ for $i \in \omega$ as 
  follows.

  Suppose $g_i$ satisfies:
  \begin{equation*}
    \tag{*}
    g_i \ind ^{\bdl} x \text{ and } g_i \ind ^{\bdl} g_ix \text{ and } \tp(g_i) 
    \text{ is wgp, and if } i > 0 \text{ then } g_i \in H.
  \end{equation*}
  Let $g_i',x' \equiv _{g_ix} g_i,x$ with $g_i',x' \ind ^{\bdl}_{g_ix} g_i,x$.
  Then $g_i,g_i' \ind ^{\bdl} x$;
  indeed, $g_i' \ind ^{\bdl} g_i,x$,
  since $g_i' \ind ^{\bdl} g_i,g_ix$,
  since $g_i \ind ^{\bdl} g_ix$.
  By symmetry, also $g_i,g_i' \ind ^{\bdl} x'$.

  Let $g_{i+1} := g_i'^{-1}g_i$, so $x' = g_{i+1}x$,
  and so $g_{i+1}$ satisfies (*):
  $g_{i+1} \ind ^{\bdl} x$ and $g_{i+1} \ind ^{\bdl} g_{i+1}x$ and (by 
  Lemma~\ref{l:wgpFacts}) $\tp(g_{i+1})$ is wgp.

  Now $g_{i+1} \in H$; for $i>0$ this follows from $g_i \in H$, and for $i=0$ it 
  follows from $g_0 \in Y$.
  But $x$ has trivial stabiliser with respect to $H \ni g_{i+1}$,
  so $g_{i+1}$ is interdefinable with $g_{i+1}x$ over $x$,
  so $\bdl(g_{i+1}) = \bdl(g_{i+1}/x) = \bdl(g_{i+1}x/x) \leq  \bdl(g_{i+1}x) = 
  \bdl(g_{i+1}x/g_{i+1}) = \bdl(x/g_{i+1}) = \bdl(x)$.
  Meanwhile, $\bdl(g_{i+1}) \geq  \bdl(g_{i+1}/g_i') = \bdl(g_i/g_i') = \bdl(g_i)$.

  So this iterative process yields a sequence $(g_i)_{i<\omega}$
  with non-decreasing $\bdl$ bounded above by $\bdl(x)$.
  For $i>0$, it follows from the independence that $g_i$ is generic in 
  $(Y^{-1}Y)^i$,
  which by Zilber indecomposability is equal to $H$ for all large enough $i$,
  say for all $i \geq  i_0$.

  Let $\eta := \sup_i \bdl(g_i)$.

  \begin{claim} \label{c:wgpCoh}
    There exist internal sets $A_1,A_2,A_3 \subseteq  H$ such that:
    \begin{enumerate}[(i)]\item $\bdl(A_j) \leq  \eta$;
    \item $\bdl(\prod_j A_j \cap \Gamma_H) \geq  2\eta$;
    \item each $A_j$ is wgp in $H$.
    \end{enumerate}
  \end{claim}
  \begin{proof}
    We prove this by $\aleph_1$-compactness.

    Let $((X^k_b)_b)_{k \in \omega}$ enumerate the {\cstr} 
    left-invariant families of proper subvarieties of $H$.
    Observe that for $i > i_0$, each of $g_i$ and $g_i^{-1}$ can be written as $hg$ 
    with $g \equiv  g_{i_0}$ and $h \ind ^{\bdl} g$: indeed, by construction each of $g_i$ and $g_{i}^{-1}$ is a product of a $\bdl$-independent sequence of elements satisfying $\tp(g_{i_0})$ or $\tp(g^{-1}_{i_0})$,
     with the last element $g$ satisfying $\tp(g_{i_0})$, so we can let $h$ be the product of all the other
      elements in the sequence. Then $h \ind^{\bdl} g$ by transitivity.
    
    Then $\tp(g) =  \tp(g_{i_0})$ is wgp and generic in $H$,
    so by Lemma~\ref{l:unifWgpGroup},
    for each $k$ we find $\epsilon_k > 0$
    such that
    $$\forall b.\; \bdl(\tp(g_i)(\K) \cap X^k_b)
    \leq  \bdl(g_i) - \epsilon_k
    \leq  \eta - \epsilon_k,$$
    for all $i>i_0$, and the same inequality holds for $\tp(g_i^{-1})$.

    Now given $m>0$, by definition of $\eta$ and the independence $g_i' 
    \ind ^{\bdl} g_i$, we can find $i>i_0$ such that $\bdl(g_i,g_i') \geq  
    2\eta-\frac1m$.
    Then a compactness argument yields $\emptyset $-definable sets $A_1 \ni 
    (g_i')^{-1}$, $A_2 \ni g_i$, and $A_3 \ni g_{i+1}$
    such that $\bdl(A_j) \leq  \eta+\frac1m$,
    and $\bigwedge_{k < m} \forall b.\; \bdl(A_j \cap X^k_b) \leq  \eta - \epsilon_k + 
    \frac1m$;
    the compactness argument for this last condition may require some 
    explanation, one way to see it is as follows:
    $\tp(g_i)$ is equivalent to the conjunction of a chain of $\emptyset $-definable 
    sets $\bigwedge_{l \in \omega} B_l$, and if there were $\beta>0$ such that for 
    all $l$ there is $b$ with $\bdl(B_l \cap X^k_b) \geq  \eta - \epsilon_k + 
    \beta$, then by $\aleph_1$-compactness we would find $b_0$ such that 
    $\bdl(B_l \cap X^k_{b_0}) \geq  \eta - \epsilon_k + \beta$ for all $l$, 
    contradicting $\bdl(\tp(g_i) \cap X^k_{b_0}) \leq  \eta - \epsilon_k$.

    Now if $V_{b_0} \subsetneq  H$ is a proper subvariety (equivalently, since $H$ is 
    connected, a lower dimension subvariety), it is part of a 
    {\cstr} family $(V_b)_b$ of such,
    and then $\{ hV_b : h \in H,\; b \}$ is a left-invariant {\cstr} 
    family of such.
    So every proper subvariety of $H$ is of the form $X^k_b$ for some $k$ and 
    $b$.

    Meanwhile, $((g_i')^{-1}, g_i,g_{i+1}) \in \Gamma_H$, so
    $$\bdl(\prod_j A_j \cap \Gamma_H) \geq  \bdl((g_i')^{-1},g_i,g_{i+1}) =
    \bdl(g_i,g_i') \geq  2\eta - \frac1m .$$

Summarising, for every $m>0$ we have found  $\emptyset$-definable sets $A_1$, $A_2$, and $A_3$ such that $\bdl(A_j)\leq \eta+\frac 1m$, $ \bdl(\prod_j A_j \cap \Gamma_H) \geq  2\eta - \frac1m $, and $\bigwedge_{k < m} \forall b.\; \bdl(A_j \cap X^k_b) \leq  \eta - \epsilon_k + 
\frac1m$.
   By $\aleph_1$-compactness, we may find internal sets $A_1$, $A_2$, and $A_3$ satisfying these conditions for all $m$ simultaneously, and hence satisfying conditions (i), (ii), and (iii) of the claim.
    \subqed{c:wgpCoh}
  \end{proof}

  Hence $\Gamma_H$ is wgp-coherent, and nilpotence of $H$ follows from 
  Fact~\ref{f:wgpCohNilp}.
\end{proof}

\begin{remark} \label{r:bourgain}
  Although the characteristic 0 assumption is necessary for Theorem~\ref{t:cohAct} 
  as stated, it may be interesting to consider a positive characteristic 
  version. In particular, the main result of \cite{Bourgain-SL2} can be seen 
  to be of this form.
\end{remark}

\section{Incidence bounds after eliminating complete bipartite graphs}
\label{s:STvariant}
We explain how to adapt results on Szemerédi-Trotter-like incidence bounds for
$K_{d,s}$-free binary relations satisfying certain tameness conditions to 
apply also to such relations which are not actually $K_{d,s}$-free, by 
bounding what is left once (parts of) all instances of $K_{d,s}$ are deleted 
(even though this deletion need not preserve the tameness). This is used in 
the proof of Proposition~\ref{t:mainOrchardIrred}.

\begin{definition}
  Given $d,s \in \N$ and a binary relation $E \subseteq  X \times Y$,
  a \defn{$(d,s)$-transversal} of $E$ is a subset $F \subseteq  E$ such that
  for every $A \times B \subseteq  E$ with $|A| \geq  d$ and $|B| \geq  s$,
  for every $a \in A$ there exists $b \in B$ such that $(a,b) \in F$.
\end{definition}

In particular, the union of all instances of $K_{d,s}$ in $E$ is a 
$(d,s)$-transversal.

\begin{proposition} \label{p:constInc}
  Let $W_1$, $W_2$, and $V \subseteq  W_1 \times W_2$ be constructible sets over a 
  field $L$ of characteristic 0.
  Then there exists $t > 1$ such that
  for any $d > 1$,
  there is $C > 0$ such that
  for any $s \in \N$
  and any $A_i \subseteq_{\fin}  W_i(L)$,
  if $F$ is a $(d,s)$-transversal of $E := V(L) \cap (A_1 \times A_2)$, then
  $$|E \setminus F| \leq  Cs(|A_1|^{\frac{(t-1)d}{td-1}} |A_2|^{\frac{td-t}{td-1}} + 
  |A_1| + |A_2|).$$
\end{proposition}

In the case that $E$ is $K_{d,s}$-free, so $F = \emptyset$ is a 
$(d,s)$-transversal, Proposition~\ref{p:constInc} is essentially 
\cite[Theorem~9]{ES-groups} (the bound proven there is slightly weaker, but 
would still be sufficient for our purposes).
A stronger form of that theorem is proven in the general setting of distal 
cell decompositions as \cite[Theorem~2.8]{CPS}, and we will explain how to 
adapt the proof of the latter to obtain a corresponding generalisation, 
Proposition~\ref{p:distInc}, of Proposition~\ref{p:constInc}. We refer to 
\cite{CPS} for the definition of a distal cell decomposition. A reader who 
would prefer to restrict to the context of Proposition~\ref{p:constInc} could 
instead similarly adapt the proof of \cite[Theorem~9]{ES-groups}.

The proof of \cite[Theorem~2.8]{CPS} rests on \cite[Fact~2.10]{CPS}, so we 
first adapt the latter to the case of $(d,s)$-transversals, following the 
proof in \cite[Theorem 2.1]{FoxPach+}.
\begin{proposition}\label{p:Fox}
For every $c\in\R$ and $d\in \N$ and $t>1$ there is some constant
$C=C(c,t,d)$ such that the following holds.
Let $R\subseteq X\times Y$ be a bipartite graph such that $R$ admits a distal cell decomposition $\mathcal{T}$ with $|\mathcal{T}(B)|\leq c|B|^t$ for all $B\subseteq_{\fin} Y$. Then, for any $A\subseteq X$, $B\subseteq Y$ with $|A|=m$ and $|B|=n$, if $F$ is a $(d,s)$-transversal of $E:=R\cap(A\times B)$, we have
\[|E\setminus F|\leq Cs(m^{1-\frac1t}n+m).\]
 \end{proposition}
 \begin{proof}
The proof is essentially the same as the proof of \cite[Theorem 2.1]{FoxPach+}, except in one place where we use $F$ being a $K_{d,s}$-transversal instead of $E$ being $K_{d,s}$-free. 
 
     Let $\mathcal{F}:=\{E_a:a\in A\}$. Then by \cite[Remark 2.11]{CPS}, we have $\pi_\mathcal{F}(z)\leq cz^t$ for all $z\in\N$, where $\pi_\mathcal{F}(z):=\max\{|\mathcal{F}\cap B'|:B'\subseteq B, |B'|=z\}$. 
     Let $\mathcal{F}^*:=\{E_b:b\in B\}$.

     Recall that a subset $D_0$ of a set $D$ \emph{crosses} another subset $D_1 \subseteq D$ if both $D_1\cap D_0\neq\emptyset$ and $D_1\cap (D\setminus D_0)\neq\emptyset$.
     Assume for now that $|A| \geq d$. We use the Packing Lemma, \cite[Lemma~2.5]{FoxPach+}, in the following form: there is $\tilde{c}=\tilde{c}(c,t,d)\in\R$ such that, for any $\delta>0$, if every subset of $A$ of cardinality $d$ is crossed by at least $\delta$ sets from $\mathcal{F}^*$, then $m = |\mathcal{F}| \leq \tilde{c}(\frac n\delta)^t$.

     As in \cite[Observation 2.6]{FoxPach+}, it follows that
     there are distinct $x_1,x_2,\ldots,x_d\in A$ such that at most $2c'nm^{-\frac1t}$ sets from $\mathcal{F}^*$ cross $\{x_1,x_2,\ldots,x_d\}$, where $c'=c'(c,t,d)$. Indeed, set $c' = \tilde{c}^{\frac1t}$, and suppose that all $d$-element subsets of $A$ are crossed by at least $2c'nm^{-\frac1t}$ sets from $\mathcal{F}^*$. Then by the Packing Lemma, $m\leq \tilde{c}\left(\frac n{2c'nm^{-\frac1t}}\right)^t=m2^{-t}$, contradicting $t>0$.

Let $B':=\{b\in B: E_b \text{ crosses } \{x_1,\ldots,x_d\}\}$, so $|B'|\leq 2c'nm^{-\frac1t}$. Now consider $x_1\in A$, and suppose there are $y_1,\ldots, y_s\in B\setminus B'$ such that $(x_1,y_i)\in E$ for all $i$. Then $(x_i,x_j)\in E$ for all $i,j$, as $E_{y_j}$ does not cross $\{x_1,\ldots,x_d\}$. Therefore, by the definition of $(d,s)$-transversal, there exists some $y_{j_0}$ such that $(x_1,y_{j_0})\in F$. Let $E':=E\setminus F$. The above argument gives that $|E'_{x_1}\cap (B\setminus B')|\leq s-1$. Thus,   \[|E'_{x_1}|\leq 2c'nm^{-\frac1t}+s-1.\]
Now we remove $x_1$ from $A$ and repeat the argument until there are fewer than $d$ elements in $A$. Therefore, \[|E'|\leq (d-1)n+\sum_{i=d}^{m}(2c'ni^{-\frac1t}+s-1)\leq sm+nd+2c'n(\sum_{i\leq m}i^{-\frac1t})\leq Cs(m^{1-\frac1t}n+m)\] for some $C=C(c',t,d)=C(c,t,d)$ as desired.
 \end{proof}

\begin{proposition} \label{p:distInc}
  Let $d,t \in \N_{>1}$ and $c \in \R$.
  There exists $C$ such that the following holds.
  Let $R \subseteq  X \times Y$ be a binary relation which admits a distal cell 
  decomposition $\calT$ with $|\calT(B)| \leq  c|B|^t$ for all $B \subseteq_{\fin}  Y$.
  Then for all $s \in \N_{>1}$ and $A \subseteq_{\fin}  X$ and $B \subseteq_{\fin}  Y$,
  if $F$ is a $(d,s)$-transversal of $E := R \cap (A\times B)$, then
  $$|E \setminus F| \leq  Cs(|A|^{\frac{(t-1)d}{td-1}} |B|^{\frac{td-t}{td-1}} + |A| + 
  |B|).$$
\end{proposition}
\begin{proof}
  In the case that $E$ is $K_{d,s}$-free, so $F = \emptyset$ is a 
  $(d,s)$-transversal, this is precisely the statement of 
  \cite[Theorem~2.10]{CPS}. We briefly explain how to adapt the proof of that 
  theorem.

  First, note that $E' := E \setminus F$ is $K_{d,s}$-free, so \cite[Fact~2.9]{CPS} 
  applies to it.
  By the Proposition~\ref{p:Fox}, \cite[Fact~2.10]{CPS} also goes through to bound $|E'|$. Now, the rest of the proof follows verbatim, except at the point that $K_{d,s}$-freeness is used; our assumption that $F$ is a $(d,s)$-transversal substitutes for this in exactly the same way as in Proposition~\ref{p:Fox}.
\end{proof}
\begin{proof}[Proof of Proposition~\ref{p:constInc}]
  We may assume $L = \C$, since any finitely generated subfield of $L$ embeds 
  in $\C$.
  Identifying $\C = \R + i\R$ with $\R^2$,
  $V(\C)$ is semi-algebraic,
  i.e.\ definable in the distal structure $(\R;+,\cdot)$,
  so $V(\C)$ admits a distal cell decomposition; see e.g.\ 
  \cite[Fact~2.13]{CPS}.
  So Proposition~\ref{p:distInc} applies to yield the necessary bound.
\end{proof}

\subsection{Pseudofinite consequences}

\begin{definition}\label{d:transversalPsf}
If $E \subseteq  X \times Y$ is an internal binary relation and $d,s \in \N^\U$,
call an internal subset $F \subseteq  E$ a \defn{$(d,s)$-transversal} of $E$ if for 
every internal $A \times B \subseteq  E$ with $|A| \geq  d$ and $|B| \geq  s$, for every $a 
\in A$ there exists $b \in B$ such that $(a,b) \in F$.
\end{definition}

Note that $\prod_{i \rightarrow  \U} E_i$ is a $(\prod_{i \rightarrow  \U} d_i, \prod_{i \rightarrow  \U} 
s_i)$-transversal if and only if $E_i$ is a $(d_i,s_i)$-transversal for 
$\U$-many $i$.

Repeating the proof of \cite[Lemma~2.15]{BB-cohMod}, we obtain the following 
corollary of Proposition~\ref{p:constInc}.

\begin{corollary} \label{c:distIncBdl}
  Let $W_1$, $W_2$, and $V \subseteq  W_1 \times W_2$ be constructible over $\K$.
  Then there exists $\epsilon > 0$ such that
  if $s \in \N^\U$, and $A_i \subseteq  W_i(\K)$ are internal,
  then setting $E := V(\K) \cap (A_1 \times A_2)$,
  if $F$ is a $(2,s)$-transversal of $E$, then
  $$\bdl(E \setminus F) \leq 
  \bdl(s) + \max(\bdl(A_1) + \frac12\bdl(A_2) - \epsilon(\bdl(A_1) - 
  \frac12\bdl(A_2)),\bdl(A_1),\bdl(A_2)).$$
\end{corollary}
\begin{proof}
  As in the proof of Proposition~\ref{p:constInc}, $V$ admits a distal cell 
  decomposition $\calT$, say with $|\calT(B)| \leq  c|B|^t$ for all $B \subseteq_{\fin}  W_2$. We claim that $\epsilon := \frac1{2t-1} > 0$ works. Consider internal sets $A_1=\prod_{i \rightarrow  \U} A_{1,i}\subseteq W_1(\K)$ and $A_2=\prod_{i \rightarrow  \U} A_{2,i}\subseteq W_2(\K)$. We may assume  $A_{1,i}$ and $A_{2,i}$ are finite. For $\U$-many $i$ we have that  $F_i$ is a $(2,s_i)$-transversal of $E_i$ (where $s=\prod_{i\rightarrow\U}s_i$, etc.), so by Proposition~\ref{p:constInc} we get 
	
	\[|E_i\setminus F_i|\leq Cs_i(|A_{1,i}|^{1-\frac1{2t-1}}|A_{2,i}|^{\frac 12 +\frac 12 \cdot \frac1{2t-1}} +|A_{1,i}|+|A_{2,i}|)\]\[\leq \ 3C s_i\max(|A_{1,i}|^{1-\epsilon}|A_{2,i}|^{\frac 12 +\frac 12 \epsilon},|A_{1,i}|,|A_{2,i}|)\]
	Applying $\log_{\xi_i}$ and taking ultralimits yields the desired inequality.
\end{proof}

We conclude this appendix with a similar pseudofinite translation of a result 
of Tóth on Szemerédi-Trotter over the complex numbers; the precise exponents 
in the bound play a crucial role in the proof of 
Theorem~\ref{t:mainOrchardIrred}.
\begin{fact}[\cite{Toth}] \label{f:Toth}
  There exists $c \in \R$ such that the following holds.
  Let $A_1$ and $A_2$ be finite sets of points and lines respectively in the 
  complex projective plane. Let $E := I \cap (A_1\times A_2)$ be the incidence 
  relation. Then
  $$|E| \leq  c(|A_1|^{\frac23}|A_2|^{\frac23} + |A_1| + |A_2|).$$
\end{fact}

We deduce:
\begin{corollary} \label{c:Toth}
  Let $I$ be the point-line incidence relation of a projective plane over 
  $\K$. Let $A_1$ and $A_2$ be $\bigwedge$-internal sets of points and lines 
  respectively in this plane. Let $E := I \cap (A_1 \times A_2)$. Then
  $$\bdl(E) \leq 
  \max(\frac23(\bdl(A_1) + \bdl(A_2)),\bdl(A_1),\bdl(A_2)).$$
\end{corollary}
\begin{proof}
  In the case that $A_1$ and $A_2$ are internal, this follows directly from 
  Fact~\ref{f:Toth}, as in the proof of Corollary~\ref{c:distIncBdl}.

  The $\bigwedge$-internal case follows. Indeed, for any $\epsilon>0$ we find 
  internal sets $A_i' \supseteq A_i$ of points and lines such that $\bdl(A_i') \leq  
  \bdl(A_i) + \epsilon$. Then
  \begin{align*} \bdl(E) &\leq  \bdl(I \cap (A_1'\times A_2'))\\
  &\leq  \max(\frac23(\bdl(A_1) + \bdl(A_2) + 2 \epsilon), \bdl(A_1) + \epsilon, 
  \bdl(A_2) + \epsilon)\\
  &\leq  \max(\frac23(\bdl(A_1) + \bdl(A_2)), \bdl(A_1), \bdl(A_2)) + 
  \frac43\epsilon .\end{align*}
  Since $\epsilon>0$ was arbitrary, we conclude.
\end{proof}

\section{Fixed points of compositions of Geiser involutions}
\label{s:geiser}
Let $\K$ be an algebraically closed field of characteristic 0,
and let $S \subseteq  \P^3$ be a smooth irreducible cubic surface over $\K$.
In this appendix, we prove a purely algebro-geometric result,
Theorem~\ref{t:coplanar}, restricting the set of fixed points of the
composition of 4 suitably generic Geiser involutions on $S$.

We use the definitions of ``good'', $R_S'$, and $\Pi$ from \S\ref{ss:lines}, 
as well as Fact~\ref{f:lines}. We also use the notation $\trd^0$, $\acl^0$, 
$\locus^0$, and $\ind^0$ from Definition~\ref{d:zero} to indicate that we work 
over an algebraically closed subfield $\acl^0(\emptyset)$ of $\K$, and we 
assume that $S$ is defined over $\acl^0(\emptyset)$. This appendix is 
otherwise independent of the main text.

\subsection{Algebraic geometry preliminaries}
Our primary reference for the algebraic geometry we will use is \cite{Hart}, 
and especially \cite[Chapter~V]{Hart} on surfaces and \cite[Section~V.4]{Hart} 
on cubic surfaces in $\P^3$.

We adopt the convention that varieties are reduced but not necessarily irreducible.
A curve is a 1-dimensional variety, and a surface is a 2-dimensional variety.
A curve \emph{in} a variety is a 1-dimensional closed subvariety of the surface.
Except where otherwise mentioned, all varieties are over the algebraically 
closed field $\K$, and we often use the same notation for a variety and for 
its set of $\K$-points. More generally, we adopt ``Weil-style'' conventions
where they clarify the exposition; since we work over an algebraically closed
field, there is no conflict with the scheme-theoretic setting in
\cite{Hart}.

We work with Weil divisors \cite[II.6]{Hart}. Recall that on a smooth irreducible variety, 
Weil divisors coincide with Cartier divisors \cite[Proposition~II.6.11, 
Remark~II.6.11.1A]{Hart}.
A morphism $f : X \rightarrow  Y$ of smooth irreducible varieties induces homomorphisms $f^* : 
\Div(Y) \rightarrow  \Div(X)$ of divisor groups and $f^* : \Pic(Y) \rightarrow  \Pic(X)$ of 
divisor class groups \cite[Corollary~II.6.16, Exercise~II.6.8(a)]{Hart}; 
context will disambiguate which is meant by $f^*$.

We work with the multiplicity of a point $p$ on a curve $C$ in a surface, 
denoted $\mu_p(C)$, as in \cite[Remark~V.3.5.2]{Hart}.

We also work with blow-ups at points; see \cite[Section~V.3]{Hart} (where such 
blow-ups are termed ``monoidal transformations'').

\subsection{Collinearity on planar cubic curves}
Let $k$ be a field of characteristic 0 and let $C \subseteq  \P^2$ be an absolutely 
irreducible cubic planar curve over $k$.
Then, generalising the well-known definition of the group operation on an elliptic curve, collinearity induces an abelian group structure as follows; see e.g.\ 
\cite[Theorem~2.1(iii)$\Rightarrow$(iv), Example~2.2]{manin-cubic} or the references in the proof of
Fact~\ref{f:torsionFinite}.

\begin{fact} \label{f:cubicGroup}
  Let $C'$ be the set of non-singular points of $C(k)$. Suppose $C'$ is 
  non-empty, and let $u \in C'$.
  For $a,b,c \in C'$, write $a \o b = c$ if the divisor of the intersection of 
  some projective line in $\P^2$ with $C(k)$ is the sum of the divisors 
  corresponding to the points $a,b,c$.

  Then $x+y := u\o(x\o y)$ defines an abelian group structure on $C'$,
  and $x\o y = (u\o u) - x - y$ for any $x,y \in C'$.
\end{fact}
\begin{proof}
  We give the calculation for the final statement:
  $$(x\o y) + y + x = u\o((x \o y) \o y) + x = u\o x + x = u \o ((u\o x) \o x) = u\o u.$$
\end{proof}

\begin{fact} \label{f:torsionFinite}
  Let $(C';+)$ be as in Fact~\ref{f:cubicGroup}. Then for each $n$, the $n$-torsion 
  subgroup of $C'$ is finite.
\end{fact}
\begin{proof}
  We may assume $k$ is algebraically closed.
  If $C$ is smooth, this is a standard fact about elliptic curves; see e.g.\ 
  \cite[Example~3.4, p.179]{Shaf-algGeomI}.
  Otherwise, $C'$ is isomorphic to either the multiplicative or the additive 
  group of the field; this is well-known, we refer to
  \cite[Exercise~1.7.3]{Shaf-algGeomI} for the fact that $C$ is projectively
  equivalent to either $y^2=x^3+x^2$ (if the singularity is nodal) or
  $y^2=x^3$ (if the singularity is cuspidal), and
  \cite[Theorem~III.7]{SilvTate} which proves (even though the statement only
  discusses the rational points) that the group is then isomorphic to the
  multiplicative or additive group respectively.
  Finiteness of the $n$-torsion is then clear (since $\operatorname{char}(k) = 0$).
\end{proof}

\begin{fact} \label{f:nonsingPlanar}
  If three points $a,b,c \in C(k)$ are distinct and collinear, then each is 
  non-singular on $C$.
\end{fact}
\begin{proof}
  Since $C$ is cubic, it has at most one singular point (as in the proof of 
  Lemma~\ref{l:multDeg3Points} below), so say $b$ and $c$ are non-singular. But then 
  by Fact~\ref{f:cubicGroup}, $a = b\o c=u\circ u-b-c$ is also in the group $C'$ of non-singular 
  points, so $a$ is also non-singular.
\end{proof}

\subsection{Some geometry on the surface $S$}
\begin{fact} \label{f:pic}
  \begin{enumerate}[(a)]\item $S$ is the result of blowing up $\P^2$ at six distinct points 
  $P_1,\ldots ,P_6$,
  no three of which are collinear, and which do not lie on a common conic.

  \item The divisor class group $\Pic(S)$ is freely generated as an abelian group 
  by the divisor classes $l,e_1,\ldots ,e_6$, where $e_i$ is the class of the exceptional divisor at 
  $P_i$ and $l$ is the lift of the divisor class of a line in $\P^2$.

  \item The divisor class of a hyperplane section, $S \cap \pi$ for a plane $\pi \in 
  \Pi$, is $3l - \sum_{i=1}^6e_i$;
  conversely, any divisor with this class is a hyperplane section.

  \item A curve $C$ in $S$ with divisor class $al - \sum_{i=1}^6b_ie_i$ has 
  degree $d(C) = 3a - \sum_{i=1}^6b_i$ as a curve in $\P^3$.
  \end{enumerate}
\end{fact}
\begin{proof}
  \begin{enumerate}[(a)]\item \cite[Remark~V.4.7.1]{Hart}
  \item \cite[Proposition~V.4.8(a)]{Hart}
  \item \cite[Proposition~V.4.8(c)]{Hart}, and the converse statement follows from the fact \cite[Corollary~V.4.7]{Hart} that the projective embedding is that of this (very ample) linear system.
  \item \cite[Proposition~V.4.8(e)]{Hart} \qedhere
  \end{enumerate}
\end{proof}

\begin{lemma} \label{l:multDeg3Points}
  Let $C \subseteq  S$ be an irreducible curve. If $C$ is not contained in any plane $\pi 
  \subseteq  \P^3$,
  and $a_1,a_2,a_3 \in C$ are three distinct points on $C$,
  then $\sum_{i=1}^3 \mu_{a_i}(C) \leq  d(C)$.

  If $C$ is contained in a plane and is not a line, we still have 
  $\sum_{i=1}^2 \mu_{a_i}(C) \leq  d(C)$ for any two distinct points $a_1,a_2 \in 
  C$.
\end{lemma}
\begin{proof}
  First suppose $C$ is not contained in a plane,
  let $\pi \in \Pi$ be the plane containing $a_1,a_2,a_3$,
  and let $D$ be $\pi \cap S$.

  Then $d(C)$ is precisely the intersection product $C.D$; this can be seen 
  from first principles or by calculating with divisor classes (namely: $D$ 
  has class $3l - \sum_i e_i$, so if $C$ has class $al - \sum_i b_ie_i$ then, 
  by the formulae for the intersection pairing in 
  \cite[Proposition~V.4.8(b)]{Hart}, we have $C.D = 3a - \sum_i b_i = d(C)$).

  But by \cite[Exercise~V.3.2]{Hart},
  $C.D = \sum_{a \in A} \mu_a(C)\mu_a(D) \geq  \sum_i \mu_{a_i}(C)$,
  where $A$ is a set of generalised intersection points which includes 
  $a_1,a_2,a_3$.

  For the case that $C$ is contained in a plane $\pi_0$ and $a_1,a_2 \in C$ are 
  given, the same argument applies by taking $\pi$ to be any plane through 
  $a_1,a_2$ distinct from $\pi_0$.
\end{proof}

For $a \in S$, let $C_a$ be the intersection of $S$ with the tangent plane of $S$ at 
$a$ in $\P^3$,
$$C_a := T_a(S) \cap S.$$

\begin{lemma} \label{f:C_a}
  If $a \in S$ is good,
  then $C_a$ is irreducible, $d(C_a) = 3$,
  and $a$ is its unique singular point:
  $\mu_a(C_a) = 2$, and $\mu_b(C_a) = 1$ for $b \in C \setminus \{a\}$.
\end{lemma}
\begin{proof}
  $d(C_a) = 3$ by Fact~\ref{f:pic}(c,d).

  Let $C'$ be the irreducible component of $C_a$ containing $a$; this is 
  unique since $d(C_a) = 3$ and $a$ is good, so any irreducible component of $C_a$ containing $a$ has degree at least 2.
  Then $\mu_a(C') > 1$; this follows from \cite[Remark~V.3.5.2]{Hart}
  since $a$ is singular on $C'$ (or, more directly, from the proof of 
  \cite[Theorem~I.5.1]{Hart}).

  So by Lemma~\ref{l:multDeg3Points} applied with $a$ and any other point $b$ on 
  $C'$, we must have $\mu_a(C') = 2$, $\mu_b(C') = 1$, and $d(C') = 3$,
  and hence $C_a = C'$ is irreducible.
\end{proof}

\subsubsection{Geiser involutions}
Given a good point $a \in S$,
the \defn{Geiser involution} $\gamma_a$ through $a$ is the birational map 
$\gamma_a \in \operatorname{Bir}(S)$ defined generically on $S$ by setting $\gamma_a(x)$ to 
be the third point of intersection (counting multiplicity) with $S$ of the 
line in $\P^3$ through $a$ and $x$.

This $\gamma_a$ lifts to an automorphism of the blowup $\widetilde S_a$ of $S$ at $a$, 
$\widetilde \gamma_a \in \operatorname{Aut}(\widetilde S_a)$. Namely (as described in e.g.\ 
\cite[2.6.1]{CPR-fano} and \cite[4.1,Involutions.3]{BL-birationalThreefolds}), the 
morphism $\widetilde S_a \rightarrow  \P^2$ of linear projection through $a$ is a double 
covering, and $\widetilde \gamma_a$ is the corresponding involution which exchanges the 
sheets of this covering; in particular, it exchanges the exceptional curve of 
the blowup with the strict transform of $C_a$.
Correspondingly, $\gamma_a$ is defined on $S \setminus \{a\}$, and $\gamma_a(C_a \setminus 
\{a\}) = \{a\}$, and $\gamma_a$ restricts to a biregular self-bijection of $S 
\setminus C_a$.

By \cite[Proposition~V.3.2]{Hart}, we can identify $\Pic(\widetilde S_a)$ with $\Pic(S) 
\oplus \Z e_0 \cong  \Z^8$ where $e_0$ is the divisor class of the exceptional curve 
of the blowup. We record the following instance of 
\cite[Proposition~V.3.6]{Hart} in this context.

\begin{fact} \label{f:transDiv}
  If $C \subseteq  S$ is an irreducible curve with divisor class $d \in \Pic(S)$, and 
  the multiplicity of $a$ on $C$ is $\mu_p(C) = m \geq  0$,
  then the strict transform $\widetilde C$ has divisor class $d - me_0 \in \Pic(\widetilde S_a)$.
\end{fact}

The following lemma is stated as \cite[p.457 (8.24)]{Dolgachev}, but we give a 
proof.

\begin{lemma} \label{l:geiserAut}
  Let $a \in S$ be good.
  The involution of $\Pic(\widetilde S_a)$ induced by $\widetilde \gamma_a$ is given by:
  \begin{align*} \widetilde \gamma_a^*(l) &= 8l - 3\sum_{j=0}^6e_j \\
  \widetilde \gamma_a^*(e_i) &= 3l - \sum_{j=0}^6e_j - e_i .\end{align*}
\end{lemma}
\begin{proof}
  $\widetilde \gamma_a^*(e_0)$ is the divisor class of the strict transform of $C_a$.
  Now $C_a$ has divisor class $3l-\sum_{i=1}^6e_i$ in $\Pic(S)$ by 
  Fact~\ref{f:pic}(c),
  and $\mu_a(C_a) = 2$ by Lemma~\ref{f:C_a},
  so by Fact~\ref{f:transDiv},
  $\widetilde \gamma_a^*(e_0) = 3l - \sum_{i=1}^6e_i - 2e_0 = 3l - \sum_{i=0}^6e_i - 
  e_0$.

  Now let $p_0 \in \P^2$ be the image of $a$ under the morphism $S \rightarrow  \P^2$ of 
  blowing up $p_1,\ldots ,p_6$. So $\widetilde S_a$ is the blowup of $\P^2$ at 
  $p_0,\ldots ,p_6$.
  Since $a$ is good, it follows from \cite[Theorem~V.4.9]{Hart} that no three 
  of these seven points are collinear and no six lie on a common conic.

  Then $\widetilde \gamma_a$ can be described in terms of the Cayley-Bacharach theorem 
  \cite[Corollary~V.4.5]{Hart}: given an eighth point $p_7$, perhaps 
  infinitely near to (i.e.\ in the exceptional curve of) one of $p_0,\ldots ,p_6$, 
  by Cayley-Bacharach there is a unique ninth point $p_8$ (also possibly 
  infinitely near one of $p_0,\ldots ,p_6$) such that any cubic curve in $\P^2$ 
  passing through $p_0,\ldots ,p_7$ also passes through $p_8$. From the 
  description of the embedding of $S$ into $\P^3$ in 
  \cite[Corollary~V.4.7]{Hart}, it follows that $\widetilde \gamma_a(p_7) = p_8$.
  (See also \cite[Sections~7.3.5,8.7.2]{Dolgachev} for another exposition of 
  this description of Geiser involutions.)

  Now this description of $\widetilde \gamma_a$ is invariant under permuting 
  $p_0,\ldots ,p_6$, and it follows that also $\widetilde \gamma_a^*(e_i) = 3l - 
  \sum_{j=0}^6e_j - e_i$ for any $i=0,\ldots ,6$.

  Finally, since $\widetilde \gamma_a^*$ is an involution, we have
  \begin{align*} e_0 &= \widetilde \gamma_a^*(\widetilde \gamma_a^*(e_0)) \\
  &= 3\widetilde \gamma_a^*(l) - \sum_{j=0}^6\widetilde \gamma_a^*(e_j) - \widetilde \gamma_a^*(e_0) \\
  &= 3\widetilde \gamma_a^*(l) - (21l - 7 \sum_{j=0}^6e_j - \sum_{j=0}^6 e_j) -
  (3l - \sum_{j=0}^6e_j - e_0) \\
  &= 3\widetilde \gamma_a^*(l) - 24l + 9\sum_{j=0}^6e_j + e_0 ,\end{align*}
  so
  $\widetilde \gamma_a^*(l) = 8l - 3\sum_{j=0}^6e_j$.
\end{proof}

Let $a \in S$ be good, and define $\gamma_a^* : \Div(S) \rightarrow  \Div(S)$ to be the 
homomorphism defined on prime divisors by:
$\gamma_a^*(C_a) = 0$, and for any other irreducible curve $C \neq  C_a$,
$\gamma_a^*(C)$ is the irreducible curve with strict transform 
$\widetilde \gamma_a^*(\widetilde C)$, where $\widetilde C$ is the strict transform of $C$.
Note that for such $C \neq  C_a$,
we have set theoretically $\gamma_a(C \setminus \{a\}) \subseteq  \gamma_a^*(C)$,
since $\widetilde \gamma_a(\widetilde C) = ~\gamma_a^*(\widetilde C)$,
since $\widetilde \gamma_a$ is an involutive isomorphism.

\begin{lemma} \label{l:geiserCurve}
  If $a \in S$ is good,
  and $C \subseteq  S$ is an irreducible curve with divisor class $al - \sum_{i=1}^6 
  b_ie_i$,
  and $C \neq  C_a$,
  and $a$ has multiplicity $m \geq  0$ on $C$,
  then $\gamma_a^*(C)$ has class $(8a - 3\sum_{i=1}^6 b_i - 3m)l - 
  \sum_{i=1}^6 (3a - \sum_{j=1}^6 b_j - b_i - m)e_i$,
  and the multiplicity of $a$ on $\gamma_a^*(C)$ is $3a - \sum_{i=1}^6 b_i - 
  2m$.
\end{lemma}
\begin{proof}
  By Fact~\ref{f:transDiv} and Lemma~\ref{l:geiserAut},
  the class of $\widetilde \gamma_a^*(\widetilde C)$ is
  \begin{align*} &a(8l - 3\sum_{i=0}^6e_i) - \sum_{i=1}^6 b_i(3l - \sum_{j=0}^6e_j - 
  e_i) - m(3l - \sum_{j=0}^6e_j - e_0)
  \\= &(8a - 3\sum_{i=1}^6b_i - 3m)l - \sum_{i=1}^6 (3a - \sum_{j=1}^6b_j - 
  b_i - m) e_i - (3a - \sum_{i=1}^6b_i - 2m)e_0 ,\end{align*}
  and we conclude by Fact~\ref{f:transDiv}.
\end{proof}

For later use, we record the following consequence, which follows immediately 
by Fact~\ref{f:pic}(d).
\begin{corollary} \label{c:geiserCurveMD}
  If $a \in S$ is good,
  and $C \subseteq  S$ is an irreducible curve and $C \neq  C_a$,
  then
  $$
  d(\gamma_a^*(C)) = 2d(C) - 3\mu_a(C)
  ;\;
  \mu_a(\gamma_a^*(C)) = d(C)-2\mu_a(C)
  .$$
\end{corollary}
\begin{proof}
  With notation as in the previous Lemma,
  \begin{align*} d(\gamma_a^*(C)) &= 3(8a - 3\sum_{i=1}^6 b_i - 3m) - \sum_{i=1}^6 (3a 
  - \sum_{j=1}^6 b_j - b_i - m)
  \\&= 6a - \sum_{i=1}^62b_i - 3m
  \\&= 2d(C) - 3m .\end{align*}
  Meanwhile, $\mu_a(\gamma_a^*(C)) = 3a - \sum_{i=1}^6 b_i - 2m = d(C) - 2m$.
\end{proof}

We use the following corollary in the main text.
\begin{corollary}\label{c:curvesCoplanar}
  Let $a_1,a_2,a_3 \in S$ be collinear good points with $\trd^0(a_i) = 1 = 
  \trd^0(a_i/a_j)$ for $i\neq j$. Then $a_1,a_2,a_3$ lie on an 
  $\acl^0(\emptyset)$-constructible plane.
\end{corollary}
\begin{proof}
  For $i=1,2$, let $C_i := \locus^0(a_i)$, an irreducible curve over 
  $\acl^0(\emptyset)$.
  Then $\gamma_{a_3}^*(C_1) = C_2$, and $C_i \neq C_{a_3}$.
  We have $\mu_{a_3}(C_i) \leq  1$ since $\trd^0(a_3) > 0$.
  Analysing the possibilities in Corollary~\ref{c:geiserCurveMD},
  we must have $\mu_{a_3}(C_i) = 1$ and $\deg(C_i) = 3$.
  Then $C_1 \cap C_2 \ni a_3$, so $C_1 = C_2$ since $\trd^0(a_3) > 0$.

  Then by Lemma~\ref{l:geiserCurve},
  if $al-\sum_ib_ie_i$
  is the divisor class of $C := C_1 = C_2$,
  then $a = 3d - a - 3m = 6-a$ so $a = 3$,
  and $b_i = d - b_i - m = 2 - b_i$ so $b_i = 1$.
  So by Fact~\ref{f:pic}(c), $C$ is a plane cubic curve in $S$.
  We conclude that the $\acl^0(\emptyset)$-constructible plane $\pi$ spanned 
  by $C$ contains $a_1,a_2,a_3$, as required.
\end{proof}

\begin{lemma} \label{l:multImage}
  Let $a \in S$ be good.
  Let $b \in S \setminus \{a\}$.
  Let $C \subseteq  S$ be an irreducible curve, with $C \neq  C_a$.
  Then $\mu_{\gamma_a(b)}(\gamma_a^*(C)) \geq  \mu_b(C)$,
  and we have equality if $b \notin C_a$.
\end{lemma}
\begin{proof}
  If $b \notin C_a$ then $\gamma_a$ is biregular at $b$, i.e.\ it restricts to an 
  isomorphism on a Zariski neighbourhood of $b$, and the equality follows from 
  the definition of $\mu$.

  Otherwise, $b \in C_a \setminus \{a\}$, and by Lemma~\ref{f:C_a} we have $\mu_b(C_a) = 1$ 
  and $\mu_a(C_a) = 2$. Then as in the proof of Lemma~\ref{l:multDeg3Points},
  $$ d(C) = C.C_a \geq  \mu_a(C)\mu_a(C_a) + \mu_b(C)\mu_b(C_a) = 2\mu_a(C) + 
  \mu_b(C).$$
  But then by Corollary~\ref{c:geiserCurveMD} we have
  \[\mu_{\gamma_a(b)}(\gamma_a^*(C)) = \mu_a(\gamma_a^*(C)) = d(C) - 2\mu_a(C) 
  \geq  \mu_b(C).\qedhere\]
\end{proof}

\subsection{Fixed points of compositions}
\label{s:compositions}

\providecommand{\gei}{\gamma}

From now on we will write $\gei_{x_n}\circ 
\gei_{x_{n-1}}\circ\cdots\circ\gei_{x_1}(x_0)$ as $x_0x_1\cdots x_n$ for good 
points $x_1,\ldots ,x_n$ and any point $x_0$ on the cubic surface $S$. We say 
$x_0x_1\cdots x_n$ is well-defined if $x_0\neq x_1$, $x_0x_1\neq x_2$ and so
on, and hence $x_0$ is in the domain of definition of the composition of
rational maps $\gei_{x_n}\circ \gei_{x_{n-1}}\circ\cdots\circ\gei_{x_1}$. 
We will use without comment the following properties,
which follow immediately from the properties of Geiser involutions discussed above:
\begin{itemize}\item $xy$ is well defined if $y$ is good and $x\neq y$.
\item $xy = yx$ if $x \neq y$ and $x, y$ are good.
\item $xyz = (xy)z$ if $y, z$ are good and $x\ne y$ and $xy\ne z$.
\item $xyy = x$ if $x \neq y$, $y \neq xy$ and $y$ is good.
\item If $y$ is good and $x\neq y$, then $xy=y$ if and only if $x \in C_y$.
\end{itemize}

Note that if $R_S'(x_1,x_2,x_3)$ holds, then $x_ix_j=x_k$ for 
$\{i,j,k\}=\{1,2,3\}$.

The following notion of a strongly $(a_1,\ldots ,a_n)$-fixed point with $n=4$ 
is closely connected to the energy relation $E_S$ derived from the 
collinearity relation on $S$.
 Analysing the possible curves consisting of strongly $(a,b,c,d)$-fixed points for sufficiently generic points $a,b,c,d\in S$ in 
Theorem~\ref{t:coplanar} allows us to deduce in Theorem~\ref{t:coplanarK2s} information about complete bipartite graphs $K_{2,s}$ contained in $E_S$, which is crucial in obtaining our main result in the irreducible case, Theorem~\ref{t:mainOrchardIrred}.

\begin{definition}\label{d:stronglyFixed}
  Given points $a_i \in S$,
  a \defn{strongly $(a_1,\ldots ,a_n)$-fixed point}
  is an $x \in S$
  such that $xa_1\ldots a_n = x$,
  and for all $0 \leq i<n$, $R_S'(xa_1\ldots a_i,a_{i+1},xa_1\ldots a_{i+1})$ holds 
  (where in the case $i=0$, $xa_1\ldots a_i$ denotes $x$).
\end{definition}

Note that any strongly $(a_1,\ldots ,a_n)$-fixed point
is also a strongly $(a_n,\ldots ,a_1)$-fixed point.

\begin{theorem}\label{t:coplanar}
  Suppose we have four points, $a,b,c,d$, in a smooth cubic surface $S$ with 
  $\trd^0(a,b)=\trd^0(c,d)=4$ and $\trd^0(a,b,c,d)>4$.
  Let $F$ be the constructible set of strongly $(a,b,c,d)$-fixed points in $S$.

  Suppose $\ell \subseteq  S$ is an irreducible
  $\acl^0(a,b,c,d)$-constructible 
  curve which has infinite intersection with $F$.

  Then $\ell$ and the points $a,b,c,d$ all lie on a common plane.
\end{theorem}

First, we derive the corollary which we use in the main text.
\begin{corollary} \label{c:fixedCurves}
  Let $a,b,c,d \in S$ and $F$ be as in Theorem~\ref{t:coplanar}.

  If $F$ is infinite, then $a,b,c,d$ span a plane $\pi \in \Pi$, and $F 
  \setminus \pi$ is finite.
\end{corollary}
\begin{proof}
  As $F\subseteq S$, we have $\dim F\leq \dim S=2$.
  First suppose $\dim(F) = 2$.
  Let $\pi$ be the plane spanned by $a,b,c,d$ if they are coplanar, and set $\pi := \emptyset $ otherwise.
  Let $\ell \subseteq  S$ be an irreducible curve over $\acl^0(a,b,c,d)$ which is not 
  contained in $\pi$ and intersects $F$ in an infinite set; to see that such 
  exists, note that $S \setminus F$ is contained in a finite union of curves (since 
  $S$ is irreducible), and we can choose $\ell$ which is none of these curves.
  But by Theorem~\ref{t:coplanar}, $a,b,c,d$ are coplanar and $\ell \subseteq  \pi$, 
  contradicting the choice of $\ell$.

  So $\dim(F) = 1$. Since $F$ is $\acl^0(a,b,c,d)$-{\cstr}, there 
  is a finite set of points $Z$ such that $F \setminus Z$ is contained in a finite 
  union of irreducible curves $\ell_i$, with each $\ell_i$ 
  $\acl^0(a,b,c,d)$-{\cstr} and having infinite intersection with $F$.
  By Theorem~\ref{t:coplanar} again, $a,b,c,d$ span a plane $\pi$ and each $\ell_i 
  \subseteq  \pi$, so $F \setminus \pi \subseteq  Z$ is finite as required.
\end{proof}

The rest of this subsection is devoted to the proof of Theorem~\ref{t:coplanar}.
Let $a,b,c,d$ and $\ell$ satisfy the assumptions in Theorem~\ref{t:coplanar}. 

Let $\ell_{d,a}:=\ell$ and let
$\ell_{a,b} := \gei_a^*(\ell_{d,a})$,
and $\ell_{b,c}:=\gei_b^*(\ell_{a,b})$,
and $\ell_{c,d}:=\gei_c^*(\ell_{b,c})$.
We represent the situation schematically in the following diagram.
(Note that we do not actually yet know that $a$ lies on $\ell_{d,a}$ and so
on.)

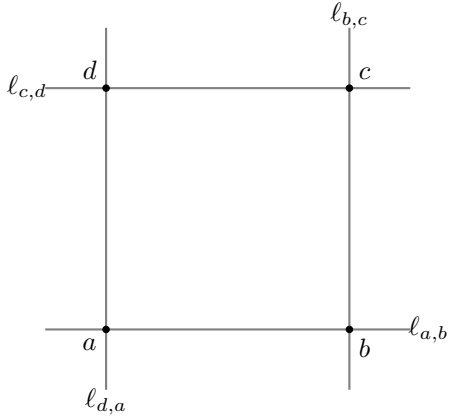
\begin{center}
\begin{tikzpicture}[scale=0.8]
\draw[gray, thick] (-2,-3) -- (-2,3);
\draw[gray, thick] (2,-3) -- (2,3);
\draw[gray, thick] (-3,2) -- (3,2);
\draw[gray, thick] (-3,-2) -- (3,-2);
\filldraw[black] (-2,-2) circle (1.5pt) node[anchor=north east]{$a$};
\filldraw[black] (-2,2) circle (1.5pt) node[anchor=south east]{$d$};
\filldraw[black] (2,2) circle (1.5pt) node[anchor=south west]{$c$};
\filldraw[black] (2,-2) circle (1.5pt) node[anchor=north west]{$b$};
\path (-2,-3.2) node(1){$\ell_{d,a}$};
\path (3.3,-2) node(1){$\ell_{a,b}$};
\path (2,3.2) node(1){$\ell_{b,c}$};
\path (-3.3,2) node(1){$\ell_{c,d}$};
\end{tikzpicture}
\end{center}

\begin{lemma} \mbox{} \label{l:ellBasics}
  \begin{enumerate}[(i)]\item
  $\ell_{d,a}, \ell_{a,b}, \ell_{b,c}, \ell_{c,d}$ are irreducible curves.
  \item Each $\ell_{x,y}$ is equal to neither $C_x$ nor $C_y$, and is not a line.
  \item
  $\ell_{a,b}$ contains infinitely many strongly $(b,c,d,a)$-fixed points,
  and\\
  $\ell_{b,c}$ contains infinitely many strongly $(c,d,a,b)$-fixed points,
  and\\
  $\ell_{c,d}$ contains infinitely many strongly $(d,a,b,c)$-fixed points.
  \item
  $\gei_d^*(\ell_{c,d}) = \ell_{d,a}$, and $\gei_d^*(\ell_{d,a}) = \ell_{c,d}$,
  and similarly $\gei_a^*(\ell_{a,b}) = \ell_{d,a}$ etc.\ 
  \end{enumerate}
\end{lemma}
\begin{proof}
  First note that $\ell_{d,a} \neq  C_a$,
  by the definition of strongly fixed points,
  so $\ell_{a,b}$ is an irreducible curve.
  Similarly $\ell_{d,a} \neq  C_d$,
  and $\ell_{d,a}$ is not a line by the goodness assumption in the definition 
  of $R_S'$.

  If $x$ is a strongly $(a,b,c,d)$-fixed point in $\ell_{d,a}$,
  then $x \in \ell_{d,a} \setminus C_a$,
  and $xa$ is a strongly $(b,c,d,a)$-fixed point in $\ell_{a,b}$.
  Now $\gei_a$ is injective on $S \setminus C_a$,
  so $\ell_{a,b}$ contains infinitely many strongly $(b,c,d,a)$-fixed points.

  The remaining cases of (i)-(iii) follow by repeating this argument.
  Finally, we also obtain in this way that $\gei_d^*(\ell_{c,d})$ is irreducible,
  and that if $x$ is a strongly $(a,b,c,d)$-fixed point in $\ell_{d,a}$
  then $x = (xabc)d \in \gei_d^*(\ell_{c,d})$.
  So $\ell_{d,a}$ and $\gei_d^*(\ell_{c,d})$ are two irreducible curves with infinite intersection,
  hence they are equal.
  Then since $\gei_d^*$ is an involution, also $\gei_d^*(\ell_{d,a}) = \ell_{c,d}$.
  The remaining cases of (iv) follow symmetrically.
\end{proof}

We prove Theorem~\ref{t:coplanar} by contradiction:
\begin{assumption} \label{a:nonplanar}
  $a,b,c,d$ and $\ell_{d,a}$ do not all lie on a common plane.
\end{assumption}

Note that it follows that also $a,b,c,d$ and $\ell_{a,b}$ are not co-planar, and so on.
This will allow us to argue ``by symmetry'' below.

We will prove a sequence of lemmas which gradually refine our picture of the situation,
until we eventually arrive at a contradiction.

\begin{lemma}\label{fixCl}
Let $F$ be the set of strongly $(a,b,c,d)$-fixed points.
  If $xabcd$ is well-defined and $x\in
\cl^{Zar}_S(F)$, then $xabcd=x$.
\end{lemma}
\begin{proof}
  Let $Z\subseteq S$ be the domain of definition of the rational map
  $\gei_d\circ\gei_c\circ\gei_b\circ\gei_a:S\to S$. Consider the map $f:Z\to S\times
  S$ given by $f(z)=(zabcd,z)$. Clearly $f$ is
  a regular map on $Z$. Hence, as $\Delta_S:=\{(x,x):x\in S\}$ is
  Zariski-closed in $S$, we get that $f^{-1}(\Delta_S)\supseteq F$ is Zariski-closed in $Z$. As the Zariski topology on $Z$ coincides with the subspace topology induced from the Zariski topology on $S$, we have $f^{-1}(\Delta_S)\supseteq\cl^{Zar}_{Z}(F)=\cl^{Zar}_{S}(F)\cap Z$. By assumption, $x\in Z$ and $x\in \cl^{Zar}_{S}(F)$, so $x\in f^{-1}(\Delta_S)$, as required.
\end{proof}

\begin{lemma}\label{lemNoInfFix}
Suppose $e,f$ are two good points such that $\trd^0(e,f)\geq 3$. Suppose $x$ is a good point which satisfies $xef=x$. Then $\trd^0(x/e,f)=0$.
\end{lemma}
\begin{proof}
  Suppose not. Then $x \notin \{e,f,ef\}$ ($ef$ is defined since $e \neq  f$), so $xe \neq  f$ and $xef \neq  f$, so $xe=xeff=xf$. Note that $xe\neq e$, since otherwise $ef = xef = x$. Suppose $xe\neq x$, then $e=exx=xex=xfx=fxx=f$, a contradiction. Thus, $xe=x=xf$ and so $e,f\in C_x$.

  Take $x' \equiv ^0_{e,f} x$ with $x' \ind ^0_{e,f} x$. In particular, $x'ef=x'$ and $\trd^0(x'/e,f,x)=\trd^0(x/e,f)\geq 1$ and $\trd^0(e,f,x,x')= \trd^0(e,f)+\trd^0(x/e,f)+\trd^0(x'/e,f,x)\geq 5$. We also have $e,f\in C_{x'}$. Since $C_{x'}$ and $C_x$ are distinct irreducible curves by Lemma~\ref{f:C_a}, $C_{x'} \cap C_x$ is a finite set, and so we get $\trd^0(e,f/x,x')=0$ and $\trd^0(e,f,x,x')=\trd^0(x,x')\leq 4$, a contradiction. 
\end{proof}

\begin{lemma}\label{b=c}
$a\neq d$ and $b\neq c$.
\end{lemma}
\begin{proof}
Suppose $a=d$. Then $\trd^0(b,c)\geq \trd^0(a,b,c,d)-\trd^0(a/b,c)\geq 5 - 2 = 3$.
Let $x \in \ell_{a,b}$ be generic over $\acl^0(a,b,c,d)$.
Then $x$ is good (by Lemma~\ref{l:ellBasics}(i)). So $xbc=xbcda=x$.
So by Lemma~\ref{lemNoInfFix}, $\trd^0(x/b,c) = 0$, contradicting the genericity.
 \end{proof}

\begin{lemma} \label{ab=c}
No three of $a,b,c,d$ are collinear; in particular, no two are equal.
\end{lemma}
\begin{proof}
  We may suppose for a contradiction that $a,b,c$ are collinear; the other cases follow by symmetry. Note that $a \neq b$ since $\trd^0(a,b) = 4$, and $b \neq c$ by Lemma~\ref{b=c}. Take $x\in\ell_{d,a}$ generic over $\acl^0(a,b,c,d)$. Note that $x$ is strongly $(a,b,c,d)$-fixed, and $x,xa \notin \{a,b,c,d\}$.

  Now consider $x,a,xa,b,xab$; they lie on the plane $\pi_0$ through $\{x,a,b\}$. Similarly, $xa,b,xab,c,xabc$ are contained in the plane $\pi_1$ through $\{xa,b,c\}$. Since $a,b,c$ are collinear and $b\notin \{a,c\}$, $\pi_0$ and $\pi_1$ contain $\{a,b,c\}$. In addition, $xa$ is in $\pi_0\cap\pi_1$, and $xa$ is not collinear with $\{a,b,c\}$ by genericity of $x$. Therefore, $\pi_0=\pi_1$. Hence $d=x(xabc)$ is also on the plane $\pi_0$, so $a,b,c,d$, and $x$ and hence $\ell_{d,a}$, are all on the plane $\pi_0$, contradicting Assumption~\ref{a:nonplanar}.
\end{proof}

We now begin to reason about the degrees and multiplicities involved in our hypothetical situation.
Let $m_a^b = \mu_a(\ell_{a,b})$ and $m_a^d = \mu_a(\ell_{d,a})$ be the 
multiplicities (possibly 0) of $a$ on $\ell_{a,b}$ and $\ell_{d,a}$ respectively.
We make corresponding definitions for the multiplicities of $b$, $c$, and $d$.
The reader may find the following diagram helpful.

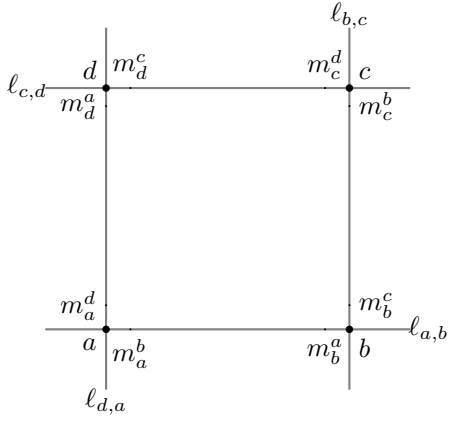
\begin{center}
\begin{tikzpicture}[scale=0.8]
\draw[gray, thick] (-2,-3) -- (-2,3);
\draw[gray, thick] (2,-3) -- (2,3);
\draw[gray, thick] (-3,2) -- (3,2);
\draw[gray, thick] (-3,-2) -- (3,-2);
\filldraw[black] (-2,-2) circle (1.5pt) node[anchor=north east]{$a$};
\filldraw[black] (-2,2) circle (1.5pt) node[anchor=south east]{$d$};
\filldraw[black] (2,2) circle (1.5pt) node[anchor=south west]{$c$};
\filldraw[black] (2,-2) circle (1.5pt) node[anchor=north west]{$b$};
\filldraw[black] (1.6,-2) circle (0.2pt) node[anchor=north]{$m_b^a$};
\filldraw[black] (-1.6,-2) circle (0.2pt) node[anchor=north]{$m_a^b$};
\filldraw[black] (1.6,2) circle (0.2pt) node[anchor=south]{$m_c^d$};
\filldraw[black] (-1.6,2) circle (0.2pt) node[anchor=south]{$m_d^c$};
\filldraw[black] (-2,1.7) circle (0.2pt) node[anchor=east]{$m_d^a$};
\filldraw[black] (2,1.7) circle (0.2pt) node[anchor=west]{$m_c^b$};
\filldraw[black] (-2,-1.6) circle (0.2pt) node[anchor=east]{$m_a^d$};
\filldraw[black] (2,-1.6) circle (0.2pt) node[anchor=west]{$m_b^c$};
\path (-2,-3.2) node(1){$\ell_{d,a}$};
\path (3.3,-2) node(1){$\ell_{a,b}$};
\path (2,3.2) node(1){$\ell_{b,c}$};
\path (-3.3,2) node(1){$\ell_{c,d}$};
\end{tikzpicture}
\end{center}

Denote the degrees of our curves by $d_{x,y} := d(\ell_{x,y})$.
By Corollary~\ref{c:geiserCurveMD} and Lemma~\ref{l:multImage}
applied to Lemma~\ref{l:ellBasics}(i),(iv),
we have in this notation:
\begin{lemma} \mbox{} \label{l:multFacts}
  \begin{enumerate}[(i)]\item $d_{a,b} = 2d_{d,a} - 3m_a^d$ and $m_a^b = d_{d,a} - 2m_a^d$.
  \item If $x \neq  a$ then $\mu_{xa}(\ell_{a,b}) \geq  \mu_x(\ell_{d,a})$,
  with equality if $xa \neq  a$.
  \end{enumerate}
  Corresponding statements hold for each of the four corners of the above 
  square and each direction (clockwise or anti-clockwise).
\end{lemma}

\subsubsection{Symmetry of degrees and multiplicities}

Our first aim is to show that the situation is symmetric in the sense that 
all the degrees $d_{x,y}$ are equal and all the multiplicities $m_x^y$ are 
equal.
We achieve this in Lemma~\ref{multsEq} below.

\begin{lemma}\label{nonZero}
  Each $m_x^y$ is non-zero.
\end{lemma}
\begin{proof}
  We prove $m_a^b\neq 0$ without using $\trd^0(a,b)=\trd^0(c,d)=4$, hence we get the other cases by symmetry. Suppose $m_a^b=0$. By Lemma~\ref{l:multFacts}(i), $m_a^d=d_{a,b}$ and $d_{d,a} = 2d_{a,b}$.

  If $ad=d$ (note that $a\neq d$ by Lemma~\ref{ab=c}), then $m_d^c = m_{ad}^c \geq m_a^d = d_{a,b}$ by Lemma~\ref{l:multFacts}(ii). By Lemma~\ref{ab=c}, $dc\neq b$, and then by Lemma~\ref{l:multFacts}(ii), $\mu_{dcb}(\ell_{a,b}) \geq  \mu_{dc}(\ell_{b,c}) \geq  \mu_d(\ell_{c,d}) = m_d^c \geq  d_{a,b}$, contradicting Lemma~\ref{l:multDeg3Points}. Thus $ad\neq d$.


By Lemma~\ref{ab=c}, $ad\neq c$ and $adc$ is well-defined. If $adc\neq b$, then $adcb$ is well-defined, and by another application of Lemma~\ref{l:multFacts}(ii), $\mu_{adcb}(\ell_{a,b}) \geq  m_a^d = d_{a,b}$, contradicting Lemma~\ref{l:multDeg3Points}. Thus $adc=b$, and $m_b^c = \mu_{adc}(\ell_{b,c}) \geq  m_a^d = d_{a,b}$. By Lemma~\ref{l:multFacts}(i), $d_{a,b} \leq  m_b^c = d_{a,b} - 2m_b^a$, so $m_b^a=0$ and $m_b^c=d_{a,b}$, and $d_{b,c} = 2d_{a,b} - 3m_b^a = 2d_{a,b} = d_{d,a}$. Since $\ell_{d,a}$ and $\ell_{b,c}$ have the same degree, and they are both images under Geiser involutions of $\ell_{c,d}$, by Lemma~\ref{l:multFacts}(i) we must have $m_d^c=m_c^d$. Similarly, $m_d^a=m_c^b=:m$. Since  $\ell_{a,b}$ is not a line, $d_{d,a} = 2d_{a,b} > 3$, hence $\ell_{d,a}$ is not planar (e.g.\ by Fact~\ref{f:pic}(c),(d)). Hence either $m_d^c = 0$ or $\ell_{c,d}$ is non-planar, and $\mu_{ad}(\ell_{c,d}) \geq m_a^d$, and $ad\notin\{c,d\}$, so $m_a^d+m_d^c+m_c^d\leq d_{c,d} = 4d_{a,b}-3m$ by Lemma~\ref{l:multDeg3Points}. Now $m_c^d = m_d^c = d_{d,a} - 2 m_d^a = 2d_{a,b} - 2m$ and $m_a^d = d_{a,b}$, so $d_{a,b}+2(2d_{a,b}-2m)\leq 4d_{a,b}-3m$, so $d_{a,b}\leq m$. Now $\ell_{d,a}$ is of degree $2d_{a,b}$ and has two distinct points $a$ and $d$ each of multiplicity at least $d_{a,b}$, contradicting Lemma~\ref{l:multDeg3Points} (via non-planarity of $\ell_{d,a}$).
\end{proof}

\begin{lemma} \label{l:nonplanar}
  None of the curves $\ell_{a,b},\ell_{b,c},\ell_{c,d},\ell_{d,a}$ are planar.
\end{lemma}
\begin{proof}
  If any one is planar then, by Lemma~\ref{nonZero}, they all lie on a single plane containing $\{a,b,c,d\}$, contradicting Assumption~\ref{a:nonplanar}.
\end{proof}

\begin{lemma}\label{irre}
Suppose two points $x$ and $y$ in $S$ are on a plane $\pi \in \Pi$ and $\trd^0(x,y)=4$. Let $C_{\pi}:=S\cap \pi$ be the planar curve through $\pi$. Then $C_{\pi}$ is irreducible.
\end{lemma}
\begin{proof}
Suppose not, then $C_{\pi}$ contains a line. Since $\trd^0(x) = 2$, $x$ is good, so is not on this line. So since there are only finitely many lines in $S$, we have $\pi \in \acl^0(x)$. Since $y\in C_{\pi}$, we have $\trd^0(y/\pi)=1$. Hence, $\trd^0(y/x)\leq \trd^0(y/\pi)=1$ which contradicts that $\trd^0(x,y)=4$.
\end{proof}

\begin{lemma} \label{ab=cd}
$ad\neq bc$ and $ab \neq cd$.
\end{lemma}
\begin{proof}
Suppose $ad=bc$. Then $a,b,c,d$ are on the same plane $\pi_0$. Take $x_1$ a generic point on $\ell_{d,a}$ over $\trd^0(a,b,c,d)$ and set $x_2:=x_1a$, $x_3:=x_1ab$ and $x_4:=x_1abc$. By genericity, $x_1,x_2,x_3$ and $x_4$ are good points. By Lemma~\ref{l:nonplanar}, $x_1 \notin \pi_0$. Then $x_1=ax_2=dx_4$, and $x_4,x_2$ are on the plane $\pi\neq \pi_0$ through $\{x_1,a,d\}$. Similarly, $x_3=bx_2=cx_4$, and $x_4,x_2$ are on the plane $\pi'\neq \pi_0$ through $\{x_3,b,c\}$. Now $\pi\neq\pi'$, since otherwise $x_1 \in \pi=\pi'=\pi_0$.
\begin{claim} \label{c:x2neqx4}
  $x_2 \neq  x_4$.
\end{claim}
\begin{proof}
  Suppose $x_2 = x_4 = x_2bc$.
  Then by Lemma~\ref{lemNoInfFix}, $\trd^0(b,c) = 2$, so $b \in \acl^0(c)$.
  Now $ad=bc$, and so $ad \notin \{a,d\}$ by Lemma~\ref{ab=c}, so $a \in \acl^0(b,c,d) = \acl^0(b,d)$,
  so $5 \leq  \trd^0(a,b,c,d) = \trd^0(b,d) \leq  4$, a contradiction.
  \subqed{c:x2neqx4}
\end{proof}
So since $ad=bc$ is on $\pi\cap\pi'$ and is distinct from $x_2$ and $x_4$, we have $x_2x_4=ad=bc$.

\begin{claim} \label{c:x1ad1}
  $\trd^0(x_1/a,d) = 1$.
\end{claim}
\begin{proof}
  First suppose $\trd^0(x_1,a)<4$.
  Then $\trd^0(x_1/a) \leq  1$, so
  $$1 = \trd^0(x_1/a,b,c,d) \leq  \trd^0(x_1/a,d) \leq  \trd^0(x_1/a) \leq  1,$$
  and we conclude.

  So suppose $\trd^0(x_1,a)=4$.
  Then $\ell_{\pi}:=S\cap\pi$ is an irreducible curve by Lemma~\ref{irre}. We use the additive group structure on non-singular points of $\ell_{\pi}$ as in Fact~\ref{f:cubicGroup}.
  Now $R_S'(x_1,a,x_2)$ and $R_S'(x_1,d,x_4)$ (since $x_1$ is strongly 
  $(a,b,c,d)$-fixed), so by Fact~\ref{f:nonsingPlanar}, $a,d,x_1,x_2,x_4$ are 
  non-singular on $\ell_{\pi}$.
  Choose a non-singular base point $u_\pi\in\ell_{\pi}$ and let $v_{\pi}:=u_{\pi}\circ u_{\pi}$, where $\circ$ is the operation defined in Fact~\ref{f:cubicGroup}. Now $v_{\pi}-a-d=a\o d = ad=x_2x_4= x_2\o x_4 = v_{\pi}-x_2-x_4$ and $v_{\pi}-d-x_4=dx_4=x_1=ax_2=v_{\pi}-a-x_2$. Hence, $2x_4=2a$ and $2x_2=2d$. Therefore, $x_2,x_4\in\acl^0(\pi,a,d)$ by Fact~\ref{f:torsionFinite},
  and so $x_1 = x_2a \in\acl^0(\pi,a,d)$. Since $a\neq d$, we have $\trd^0(\pi/a,d)=1$, hence $\trd^0(x_1/a,d)\leq 1$. Since $\trd^0(x_1/a,b,c,d)=1\leq \trd^0(x_1/a,d)$, we get $\trd^0(x_1/a,d)= 1$. 
  \subqed{c:x1ad1}
\end{proof}

Hence $\ell_{d,a} = \locus^0(x_1/\acl^0(a,d))$ is 
$\acl^0(a,d)$-{\cstr} and consequently $\ell_{a,b}=\gamma_a^*(\ell_{d,a})$ and
$\ell_{c,d}=\gamma_d^*(\ell_{d,a})$ are $\acl^0(a,d)$-{\cstr} as
well. Since $ad=bc$ and $ad\not\in\{b,c\}$, we get
$\trd^0(b/a,d,c)=\trd^0(c/a,d,b)=0$. By assumption $\trd^0(a,b,c,d)
\geq  5$, hence $\trd^0(b/a,d)\geq 1$ and $\trd^0(c/a,d)\geq 1$. Since
$m_b^a>0$ by Lemma~\ref{nonZero}, namely $b\in\ell_{a,b}$, we get $b$
is generic on $\ell_{a,b}$ over $\acl^0(a,d)$ as $\trd^0(b/a,d)\geq
1$, in particular $m_b^a=1$. Similarly, $m_c^d=1$ and $c$ is generic
over $a,d$ on $\ell_{c,d}$.  Hence $ba$ is a generic point on $\ell_{d,a}$ over
$\acl^0(a,d)$, since $\trd^0(ba/a,d)=\trd^0(b/a,d)\geq 1$. Therefore,
$\tp^0(x_1/a,d)=\tp^0(ba/a,d)$. Note that $ad=x_4x_2=(x_1d)(x_1a)$,
hence $ad=(bad)(baa)=badb$. Since $ad=bc$, we get $badb=cb$ and
$bad=c$ (since $cb=ad\neq b$ by Lemma~\ref{ab=c}), hence $ba=cd$. By
symmetry, $\ell_{a,b}$ (hence also $\ell_{b,c}$ and $\ell_{a,d}$) is
$\acl^0(a,b)$-{\cstr} and $c$ is a generic point on $\ell_{b,c}$
over $\acl^0(a,b)$ and $m_c^b=1$. Since
$1=m_c^b=d_{c,d}-2m_c^d=d_{c,d}-2$, we get $d_{c,d}=3$ and by symmetry
$d_{x,y}=3$ and $m_x^y=1$ for all $x,y\in\{a,b,c,d\}$.

Suppose $\ell_{a,b}$ has divisor class $\alpha l-\sum^6_{i=1} b_i
e_i$. Then $3=d_0=3\alpha-\sum^6_{i=1} b_i$. By
Lemma~\ref{l:geiserCurve}, $\ell_{d,a}=\gamma^*_{a}(\ell_{a,b})$ has
divisor class \[(8\alpha-3\sum_{i=1}^6b_i-3)\ell-\sum_{i=1}^6(3\alpha-\sum_{j=1}^6b_j-b_i-1)e_i=(6-\alpha)\ell-\sum_{i=1}^6(2-b_i)e_i,\]
and  $\ell_{c,d}$ has divisor class $\alpha l-\sum^6_{i=1} b_i e_i$.

Since $\gamma_{ad}(b)=c$ where $b$ and $c$ are generic on $\ell_{a,b}$
and $\ell_{c,d}$ over $\acl^0(a,d)$ respectively and $\ell_{a,b},
\ell_{c,d}$ are $\acl^0(a,d)$-{\cstr}, we get
$\gamma^*_{ad}(\ell_{a,b})=\ell_{c,d}$. As $da=bc\neq a$, by
Lemma~\ref{l:multFacts}(ii),
$\mu_{da}(\ell_{a,b})=\mu_{d}(\ell_{d,a})=1$.
Hence, $\ell_{c,d}=\gamma^*_{ad}(\ell_{a,b})$ also has divisor class
$(6-\alpha)\ell-\sum_{i=1}^6(2-b_i)e_i$. Therefore, $6-\alpha=\alpha$
and $2-b_i=b_i$ and $\alpha=3$, $b_i=1$ for $1\leq i\leq 6$, and so
$\ell_{c,d}$ is planar, contradicting Lemma~\ref{l:nonplanar}.

This shows $ad\neq bc$. Since we did not use $\trd^0(a,b)=4$ or $\trd^0(c,d)=4$, the 
same argument applies to show that $ab\neq cd$.
\end{proof}

\begin{lemma} \label{l:ab=b}
  $ab \notin \{a,b\}$ and $cd \notin \{c,d\}$.
\end{lemma}
\begin{proof}
  If $ab = a$, then $b \in C_a$, so $\trd^0(b/a) \leq  1$, contradicting $\trd^0(a,b)=4$.
  The other cases are similar.
\end{proof}

\begin{lemma}\label{multsEq}
  All curves $\ell_{d,a},\ell_{a,b},\ell_{b,c},\ell_{c,d}$ have degree $d_0$ for some $d_0>0$,
  and all the multiplicities $m_x^y$ are equal to $m:=d_0/3$.
\end{lemma}
\begin{proof}
  It suffices to show $m_x^y=m_x^z$ for all $x,y,z$.
  Indeed, setting $d_0 := d_{d,a}$, we then have $d_0-2m_a^d=m_a^b=m_a^d$, hence $m_a^d=d_0/3$ and $d_{a,b} = 2d_0-3m_a^d=d_0$. Symmetrically, $m_b^a=m_b^c=d_0/3$ and $d_{b,c} = d_0$. Similarly, $m_d^c=m_d^a=d_0/3=m_c^b=m_c^d$ and $d_{c,d} = d_0$, as required.

  Given the symmetry of the situation, it suffices to show $m_a^b=m_a^d$.

  By Lemma \ref{ab=cd}, Lemma \ref{ab=c} and Lemma \ref{nonZero}, we have that $m_x^y>0$ for all $x\in\{a,b,c,d\}$ and both $y$, $ab\neq cd$, $ad\neq bc$, and no three of $a,b,c,d$ are collinear.

  \begin{claim}\label{multgeq}
    If there exists a point $e\in\ell_{a,b} \setminus \{a,b\}$, with multiplicity $\mu_e(\ell_{a,b})\geq m_a^d$, then $m_a^b\geq m_b^a$. Moreover, if the first inequality is strict then the second is also strict.  (The same holds symmetrically for each curve in both directions.)
  \end{claim}
  \begin{proof}
    We first show that if $C \subseteq S$ is a non-planar irreducible curve,
    and $a,b,e \in C$ are distinct points,
    then $\mu_a(C) - \mu_b(C) \geq \mu_e(C) - (d(C) - 2\mu_a(C))$.

    Indeed, by Lemma~\ref{l:multDeg3Points}, we have
    \begin{align*} d(C)&\geq \mu_e(C)+\mu_a(C)+\mu_b(C)
    \\&= \mu_e(C) - (d(C) - 2\mu_a(C)) + d(C) - \mu_a(C) + \mu_b(C) ,\end{align*} 
    from which the inequality follows.

    Applying this with $C = \ell_{a,b}$,
    we conclude by Lemma~\ref{l:nonplanar} and Lemma~\ref{l:multFacts}(i).
    \subqed{multgeq}
  \end{proof}

  \begin{claim} \mbox{} \label{c:table}
    \begin{enumerate}[(i)]\item $m_a^b < m_b^a \Rightarrow  m_a^d = m_a^b$.
    \item $m_b^a < m_a^b \Rightarrow  m_b^c = m_b^a$.
    \item $m_c^d < m_d^c \Rightarrow  m_c^b = m_c^d$.
    \item $m_d^c < m_c^d \Rightarrow  m_d^a = m_d^c$.
    \end{enumerate}
  \end{claim}
  \begin{proof}
    First we show (i). So suppose $m_a^b < m_b^a$.
    The point $e:=adcb$ is well-defined (since $ad\neq c$ and $adc\neq b$). Now $e\neq b$; otherwise, since $m_a^b>0$, $ba=adcba=a$ by Lemma~\ref{fixCl}, contradicting Lemma~\ref{l:ab=b}. Note that $\mu_e(\ell_{a,b}) \geq  m_a^d$ by Lemma~\ref{l:multFacts}(ii), so if also $e\neq a$, then $m_a^b\geq m_b^a$ by Claim~\ref{multgeq}, contradicting our assumption. So $adcb=e=a$. Then $m_a^d = m_a^b$ by Lemma~\ref{l:multFacts}(ii), since $ad \neq d$ (else $dcb=adcb=a$ so $dc=ab$, contradiction), and $adc \neq c$ (since $cb\neq a$), and $adcb \neq b$ (since $a \neq b$).

   The other cases follow by symmetry.
   \subqed{c:table}
  \end{proof}

   By Claim~\ref{c:table}(i), we may assume that $m_a^b \geq  m_b^a$.

   Suppose $m_c^d>m_d^c$, then by Claim~\ref{c:table}(iv) we have $m_d^a=m_d^c$. Since $cd\not\in\{c,d\}$ and has multiplicity $m_c^d>m_d^c$ on the curve $\ell_{d,a}$, we get $m_d^a>m_a^d$ by Claim~\ref{multgeq}. Therefore $m_a^b>m_b^a$, since otherwise, applying Claim~\ref{multgeq} to $ba \notin \{a,d\}$ on $\ell_{d,a}$, we get $m_a^d\geq m_d^a$, contradicting $m_d^a > m_a^d$. 
  By Claim~\ref{c:table}(ii), $m_b^a=m_b^c$. It follows from Lemma~\ref{l:multFacts}(i) that $d_1 := d_{a,b} = d_{b,c}$. Since $m_d^c=m_d^a$, also $d_2 := d_{d,a} = d_{c,d}$. Applying Lemma~\ref{l:multFacts}(i) and using $m_c^d > m_d^c = m_d^a > m_a^d$, we obtain the following contradiction:
  $$d_1=2d_2-3m_a^d>2d_2-3m_c^d=d_1.$$

  Therefore, $m_c^d\leq m_d^c$, and recall we are assuming $m_b^a\leq m_a^b$.
  Now $dc \neq c$ and $ab \neq b$ by Lemma~\ref{l:ab=b},
  so applying Claim~\ref{multgeq} twice with points $dc$ and $ab$ on $\ell_{b,c}$, we get $m_c^b=m_b^c$, and so, applying Lemma~\ref{l:multFacts}(i) repeatedly, we obtain $m_b^a=m_c^d$, $d_{a,b}=d_{c,d}$, and $m_d^c=m_a^b$. Since $b\neq c$, either $bc\neq c$ or $cb\neq b$. Suppose $bc\neq c$, then $bc$ is a point on $\ell_{c,d}$ distinct from $d$ and $c$ of multiplicity $m_b^c=m_c^b$ by Lemma~\ref{l:multFacts}(ii), hence $m_c^d\geq m_d^c$ (by Claim~\ref{multgeq}) and so $m_c^d= m_d^c$. Similarly, if $cb\neq b$, then $m_b^a=m_a^b$. In conclusion, $m_b^a=m_a^b=m_c^d= m_d^c$, and it follows by Lemma~\ref{l:multFacts}(i) that $m_c^b=m_b^c=m_d^a=m_a^d$ and $d_{d,a}=d_{b,c}$.

  Suppose $da=a$, then by Lemma~\ref{l:multFacts}(ii), $m_d^a\leq m_a^b$. We now show that $m_d^a\geq m_a^b$, hence $m_a^d=m_d^a=m_a^b$ and we are done. Indeed, we consider two cases, $abc=c$ and $abc\neq c$. If $abc=c$, then $abc=dabc$ (since $da=a$), and $d=dabcd=abcd=cd$ (the first equality from Lemma~\ref{fixCl}), contradicting Lemma~\ref{l:ab=b}. Hence $abc\neq c$, then $\{abc,c,d\}$ are three distinct points on $\ell_{c,d}$. If $m_a^b>m_c^b$, then by Claim~\ref{multgeq}, $m_c^d>m_d^c$, a contradiction. Therefore, $m_a^b\leq m_c^b=m^a_d$ and we are done.

  Recall $ad\neq bc$ and $cb\neq a$, so we may assume that $da,cb,a$ are three distinct points on $\ell_{a,b}$. Then by Lemma~\ref{l:multFacts}(ii), the above inequalities, and Lemma~\ref{l:multDeg3Points}, we have $d_{a,b} \geq  \mu_{da}(d_{a,b}) + \mu_{cb}(d_{a,b}) + m_a^b \geq  m_a^d + m_c^b + m_a^b = 2m_a^d+m_a^b$. Since $m_a^b=d_{d,a}-2m_a^d$ by Lemma~\ref{l:multFacts}(i), we get $d_{d,a}\leq d_{a,b}$. Similarly, $\{ba, cd,a\}$ are distinct points on $\ell_{d,a}$, hence $d_{d,a}\geq 2m_a^b+m_a^d=2m_a^b+d_{a,b}-2m_a^b=d_{a,b}$. Therefore, $d_{d,a}=d_{a,b}$ and $m_a^b=m_a^d$ as desired.
\end{proof}


\subsubsection{Non-singularity of $a,b,c,d$}
Let $m=m_x^y=d_0/3$ be as in Lemma~\ref{multsEq}.
Our next aim, achieved in Lemma~\ref{m=1} below, is to show that $m=1$, i.e.\ that $a,b,c,d$ are not singular on the corresponding curves.

For this, we make limited use of the notion of \emph{infinitely near} points. For our purposes, it will suffice to make the following ad hoc restricted definition; see \cite[p.392, Definition]{Hart} for the more general definition.

\begin{definition} \label{d:1infnear}
  If $C$ is an irreducible curve on $S$, a \defn{1-infinitely near} point on $C$ is a point in the intersection of the strict transform $\widetilde C \subseteq  \widetilde S_y$ of $C$ under the blowup of $S$ at a point $y \in C$ with the exceptional divisor of this blowup. The multiplicity on $C$ of such a 1-infinitely near point is defined to be its multiplicity on $\widetilde C$.
\end{definition}

\begin{lemma} \label{l:genus}
  Let $C$ be an irreducible curve on $S$ with divisor class $\alpha\ell-\sum_{i=1}^{6} b_ie_i$, and let $P_1,\ldots ,P_n$ be distinct points or 1-infinitely near points on $C$, and let $r_i$ be the multiplicity of $P_i$ on $C$.
  Then
  \[ \sum_{i=1}^n \frac12r_i(r_i-1) \leq  \frac12(\alpha-1)(\alpha-2)-\frac12\sum_{i=1}^6b_i(b_i-1) .\]
\end{lemma}
\begin{proof}
  This follows from \cite[Example~V.3.9.2]{Hart},
  using that the 1-infinitely near points are among the infinitely near points,
  that the genus of the normalisation of $C$ is non-negative (\cite[Remark~IV.1.1.1]{Hart}),
  and that the right hand side is precisely the arithmetic genus of $C$ (\cite[Proposition~V.4.8(g)]{Hart}).
\end{proof}

\begin{theorem}\label{thm5pts}
  On at least one of the four curves $\ell_{d,a},\ell_{a,b},\ell_{b,c},\ell_{c,d}$,
  there are at least 5 distinct (possibly 1-infinitely near) points each of multiplicity at least $m$.
\end{theorem}

\begin{proof}
We prove this by contradiction. Suppose there are no five distinct (1-infinitely near) points of multiplicity at least $m$ on any one of the four curves $\ell_{d,a},\ell_{a,b},\ell_{b,c},\ell_{c,d}$. 

By Lemmas~\ref{ab=c} and \ref{ab=cd},
$a,d,cd,ba$ are distinct points on $\ell_{d,a}$, each with multiplicity $m$ by Lemma~\ref{l:multFacts}(ii).

Suppose $cb=b$. Now $\gamma_b^*(\ell_{b,c})=\ell_{a,b}$, and recall this means by definition that in the blowup $\widetilde S_b$, $\widetilde \gamma_b$ maps the strict transform $\widetilde \ell_{b,c}$ to the strict transform $\widetilde \ell_{a,b}$. Since $cb=b$, $\widetilde \gamma_b(c)$ is in the exceptional divisor of the blowup, so $\widetilde \gamma_b(c)$ is a 1-infinitely near point on $\ell_{a,b}$. Since $\widetilde \gamma_b$ is an automorphism on $\widetilde S_b$, this 1-infinitely near point has multiplicity $m$ on $\ell_{a,b}$. Then since $\gamma_a$ is a local isomorphism at $b$, it induces a map of the corresponding blowups which maps $\widetilde \gamma_b(c)$ to a 1-infinitely near point on $\ell_{d,a}$ (in the blowup at $ba$) with multiplicity $m$. This is a fifth 1-infinitely near point, distinct from $a,d,cd,ba$ since these are points in the usual sense, so this contradicts our assumption. We conclude that $cb \neq  b$.

By symmetry, $cb \notin \{c,b\}$ and $ad \notin \{a,d\}$.

Now $cba\in\{a,d,cd,ba\}$, since $\mu_{cba}(\ell_{d,a})\geq m$ by Lemma~\ref{l:multFacts}(ii).
But $cba\neq d$ by Lemma~\ref{ab=cd}, and $cba \neq  ba$ since $cb \neq  b$.

Suppose $cba=a$. Then $\widetilde \gamma_a(cb)$ is a 1-infinitely near point on $\ell_{d,a}$, and since $\widetilde \gamma_a$ is an automorphism on $\widetilde S_a$ and $cb\neq a$, its multiplicity is $\mu_{cb}(\ell_{a,b})$, which is at least $m$ by Lemma~\ref{l:multFacts}(ii). This contradicts our assumption.

Therefore, $cba=cd$. By symmetry (using $cb \notin \{c,b\}$ and $ad \notin \{a,d\}$), we have $abc=ad$ and $dab=dc$ and $bad=bc$. Now $a$ is co-planar with $\{c,b,d\}$, since $(cb)a=cd$. Let $C_\pi$ be the intersection of $S$ with the plane spanned by $\{a,b,c,d\}$. Then $C_\pi$ is an irreducible curve by Lemma~\ref{irre}. Since $C_\pi$ has degree 3 as a plane curve, if $x,y,z$ are distinct collinear points on $C_\pi$ then each of them is non-singular on $C_\pi$. So since $ab\not\in\{a,b\}$ and $cd\not\in\{c,d\}$, the points $a,b,c,d$ are non-singular on $C_\pi$.

We consider the additive structure on the set $C'$ of non-singular points of $C_\pi$ induced by collinearity defined in Fact~\ref{f:cubicGroup}, with respect to an arbitrary base point $u \in C'$.
Let $v=u\circ u$. Then $xy = x \circ y =v-x-y$ for all $x \neq  y\in C'$.
So $cba=cd$ implies $c+b-a=v-c-d$, and similarly we obtain $a+b-c=v-a-d$, $d+a-b=v-c-d$ and $b+a-d=v-b-c$. The first two equations imply $3a=3c$. Similarly the last two equations imply $3b=3d$. Adding the first equation to the third, we get $3c+3d=2v$. Similarly, $3a+3b=2v$. Now $\trd^0(\pi/a,b) = 0$; indeed, $\trd^0(a,b)=4$ and the family of planes passing through $a,b$ has dimension 1, so otherwise we would have that for any generic $x,y$ and generic plane through them, we have $3x+3y=2v$ on the curve of intersection (for any choice of basepoint), but then this equation would hold generically on such a curve, contradicting Fact~\ref{f:torsionFinite}. Since $3a=3c$, we have $\trd^0(c/a,\pi)=0$ by Fact~\ref{f:torsionFinite}. Similarly, $\trd^0(d/b,\pi)=0$. In conclusion, $\trd^0(a,b,c,d)=\trd^0(c,d/a,b)+\trd^0(a,b)=\trd^0(c,d/a,b,\pi)+4=4$, a contradiction.
\end{proof}

\begin{remark}
  In fact, with some more work it is possible to pick the five points in Theorem~\ref{thm5pts} to all be points in the usual sense.
\end{remark}

\begin{lemma}\label{m=1}
$m=m_a^b=m_a^d=1$ and $\ell_{d,a}$ (and hence each curve) has degree 3.
\end{lemma}
\begin{proof}
  By Theorem \ref{thm5pts} there are at least 5 distinct (1-infinitely near) points of multiplicity $m$ on one of the four curves, which we may assume to be $\ell_{d,a}$. Let $d_0=3m$ be the degree of $\ell_{d,a}$. Suppose $\ell_{d,a}$ has divisor class $\alpha\ell-\sum_{i=1}^{6} b_ie_i$. Then $d_0=3 \alpha-\sum_{i=1}^{6}b_i$ as in Fact~\ref{f:pic}(d).

  By Lemma~\ref{l:genus}, $\frac52m(m-1) \leq  \frac12(\alpha-1)(\alpha-2)-\frac12\sum_{i=1}^6b_i(b_i-1)$.

  Since $m=\frac{d_0}{3}$ and $d_0=3 \alpha-\sum_{i=1}^{6}b_i$, we get \[\frac{5d_0}{3}(\frac{d_0}{3}-1)\leq  \alpha^2+2-d_0-\sum_{i=1}^{6}b_i^2.\]

  By Cauchy-Schwarz, $\sum_{i=1}^{6}b_i^2\geq\frac{1}{6}(\sum_{i=1}^6 b_i)^2=\frac{1}{6}(3 \alpha-d_0)^2$. Hence, \[\frac{2}{9}d_0^2-\frac{2}{3}d_0-2 \leq -\frac13d_0^2+\alpha^2 - \frac16(3\alpha-d_0)^2
    = -\frac{1}{2}( \alpha-d_0)^2\leq 0.\]

  However, if $m\neq 1$, then $m\geq 2$ and $d_0\geq 6$ and $\frac{2}{9}d_0^2-\frac{2}{3}d_0-2\geq 2$, a contradiction. Therefore, $m=1$ and $d_0=3$.
\end{proof}

\subsubsection{Planarity}
Finally, we conclude the proof of Theorem~\ref{t:coplanar} by determining the divisor class of $\ell_{a,b}$.

\begin{lemma}
  $\ell_{a,b}$ is planar.
\end{lemma}
\begin{proof}
By the previous lemma, we know $\ell_{a,b}$ has degree $d_0=3$, and $a$ and $b$ each have multiplicity 1 on $\ell_{a,b}$.


Suppose $\ell_{a,b}$ has divisor class $\alpha l-\sum^6_{i=1} b_i e_i$. Then $3=d_0=3\alpha-\sum^6_{i=1} b_i$. By \cite[Remark~V.4.8.1]{Hart}, $\alpha\geq 1$ and $0\leq b_i\leq \max\{\alpha-1,1\}$ for all $1\leq i\leq 6$.  By Lemma~\ref{l:geiserCurve}, $\ell_{d,a}=\gamma^*_{a}(\ell_{a,b})$ and $\ell_{b,c}=\gamma^*_b(\ell_{a,b})$ each have divisor class \[(8\alpha-3\sum_{i=1}^6b_i-3)l-\sum_{i=1}^6(3\alpha-\sum_{j=1}^6b_j-b_i-1)e_i=(6-\alpha)l-\sum_{i=1}^6(2-b_i)e_i,\]
and $\ell_{c,d}$ has divisor class  $\alpha l-\sum^6_{i=1} b_i e_i$. Define $\beta:=6-\alpha$ and $b_i':=2-b_i$, so $\ell_{d,a}$ and $\ell_{b,c}$ each have divisor class $\beta l-\sum_{i=1}^{6}b_i'e_i$. Therefore, $1\leq\alpha,\beta\leq 5$. By replacing $\ell_{a,b}$ with $\ell_{b,c}$ if necessary (if we show $\ell_{b,c}$ is planar, then it follows that $\ell_{a,b}$ is planar), we may assume $1\leq \alpha\leq 3$. There are four cases:
\begin{enumerate}
\item
$\alpha=1$, $\beta=5$ and $b_i=0$ and $b_i'=2$ for $1\leq i\leq 6$.
\item
$\alpha=2$, $\beta=4$, there is $A\subseteq \{1,2,3,4,5,6\}$ with $|A| = 3$ such that $b_i=b_i'=1$ for $i\in A$ and $b_j=0$, $b_j'=2$ for $j\not\in A$.
\item
$\alpha=\beta=3$, there are $i_0\neq j_0\in \{1,2,3,4,5,6\}$ such that $b_{i_0}=2=b_{j_0}'$, $b_{i_0}'=0=b_{j_0}$ and $b_i=b_i'=1$ for $i\not\in\{i_0,j_0\}$.
\item
The planar case: $\alpha=\beta=3$ and $b_i=b_i'=1$ for all $1\leq i\leq 6$.
\end{enumerate}

In the first three cases, the intersection number (by \cite[Proposition~V.4.8(b)]{Hart}) of $\ell_{a,b}$ and $\ell_{c,d}$ is \[(\alpha l-\sum^6_{i=1} b_i e_i).(\alpha l-\sum^6_{i=1} b_i e_i)=\alpha^2-\sum_{i=1}^6b_i^2=1.\] However both $bc=\gamma_c(b)=\gamma_b(c)$ and $ad=\gamma_d(a)=\gamma_a(d)$ are on $\ell_{a,b}\cap\ell_{c,d}$ and by Lemma \ref{ab=cd}, $ad\neq bc$. Hence $\ell_{a,b}=\ell_{c,d}$. Similarly, the intersection number of $\ell_{b,c}$ and $\ell_{d,a}$ is \[(\beta l-\sum^6_{i=1} b_i' e_i).(\beta l-\sum^6_{i=1} b'_i e_i)=\beta^2-\sum_{i=1}^6b_i'^2=1.\] But $ab\neq cd\in\ell_{b,c}\cap\ell_{d,a}$, hence $\ell_{b,c}=\ell_{d,a}$.
Consider the intersection number of $\ell_{a,b}$ and $\ell_{b,c}$: \[(\alpha l-\sum^6_{i=1} b_i e_i).(\beta l-\sum^6_{i=1} b'_i e_i)=\alpha\beta-\sum_{i=1}^6b_ib_i'=5.\] Note that $a,b,c,d,ab,cd$ are all in $\ell_{a,b}\cap\ell_{b,c}$ (e.g.\ $ab \in \ell_{a,b} = \gamma_b^*(\ell_{b,c})$ since $\ell_{b,c} = \ell_{d,a} \ni a$). By Lemma \ref{ab=c}, $|\{a,b,c,d\}| = 4$. Since $\trd^0(a,b)=4$, we have $ab\neq a$ and $ab\neq b$. Hence $ab\not\in\{a,b,c,d\}$ by Lemma \ref{ab=c}. Similarly $cd\not\in \{a,b,c,d\}$. As $ab\neq cd$ by Lemma \ref{ab=cd}, we get $\{a,b,c,d,ab,cd\}$ is a set of size 6 which is contained in $\ell_{a,b}\cap\ell_{b,c}$. Hence $\ell_{a,b}=\ell_{b,c}$, but $\ell_{a,b}$ and $\ell_{b,c}$ have different divisor classes in the first three cases, a contradiction. 

Therefore, by Fact~\ref{f:pic}(c), $\ell_{a,b}$ is planar.
\end{proof}

This contradicts Lemma~\ref{l:nonplanar}, concluding the proof of Theorem~\ref{t:coplanar}.

\bibliographystyle{amsalpha}
\bibliography{cubicSurfaces}
\end{document}